\DeclareSymbolFont{AMSb}{U}{msb}{m}{n}
\documentclass[reqno,centertags,11pt,a4paper,noamsfonts]{amsart}
\pdfoutput=1
\usepackage[svgnames]{xcolor}
\usepackage{graphicx}
\usepackage[english]{babel}
\usepackage[babel=true,expansion=alltext,protrusion=alltext-nott,final]{microtype}
\usepackage[margin={3cm},marginparwidth={2.5cm}]{geometry}
\usepackage[utf8]{inputenc}
\usepackage{textcomp}
\usepackage[sb]{libertine}
\usepackage[libertine,bigdelims,vvarbb]{newtxmath}
\useosf
\usepackage[varqu,varl]{inconsolata}
\usepackage{mathtools}
\usepackage[T1]{fontenc}
\usepackage{textcomp}
\usepackage{amsthm}
\usepackage{booktabs}
\usepackage[inline]{enumitem}
\numberwithin{equation}{section}
\usepackage[normalem]{ulem}

\usepackage{float}
\usepackage{caption}
\usepackage{subcaption}
\usepackage{xspace}
\usepackage[scientific-notation=true]{siunitx}
\usepackage{pstricks}
\usepackage{tikz}
\usetikzlibrary{positioning}
\usetikzlibrary{calc}
\usepackage{pgfplots}
\pgfplotsset{width=10cm,compat=1.9}
\usepackage{bm}
\usepackage{sidecap}
\usepackage{graphicx,color}

\usepackage{amsmath}
\DeclareFontFamily{U}{mathx}{}
\DeclareFontShape{U}{mathx}{m}{n}{<-> mathx10}{}
\DeclareSymbolFont{mathx}{U}{mathx}{m}{n}
\DeclareMathAccent{\widehat}{0}{mathx}{"70}
\DeclareMathAccent{\widecheck}{0}{mathx}{"71}

\providecommand{\mr}[1]{\href{http://www.ams.org/mathscinet-getitem?mr=#1}{MR~#1}}
\providecommand{\zbl}[1]{\href{https://zbmath.org/?q=an:#1}{Zbl~#1}}

\providecommand{\doi}[2]{\href{https://doi.org/#1}{#2}}

\newcommand{\C}{\mathcal{C}}

\newcommand{\ii}{\imath}

\definecolor{light_gray}{gray}{0.75}
\definecolor{lighter_gray}{gray}{0.5}
\colorlet{light_blue}{blue!20}
\definecolor{dark_green}{rgb}{0.0, 0.6, 0.0}
\definecolor{royal_blue}{rgb}{0.0, 0.22, 0.66}
\definecolor{salmon}{rgb}{1.0, 0.55, 0.41}
\definecolor{gold}{rgb}{0.8, 0.63, 0.21}
\definecolor{navy_blue}{rgb}{0.0, 0.0, 0.5}
\definecolor{crimson}{rgb}{0.79, 0.0, 0.09}
\definecolor{amethyst}{rgb}{0.6, 0.4, 0.8}
\definecolor{alizarin}{rgb}{0.82, 0.1, 0.26}
\definecolor{amaranth}{rgb}{0.9, 0.17, 0.31}
\definecolor{azure}{rgb}{0.0, 0.5, 1.0}
\definecolor{canaryyellow}{rgb}{0.82, 0.41, 0.12}
\definecolor{carrotorange}{rgb}{0.8, 0.33, 0.0}
\definecolor{cadmiumgreen}{rgb}{0.0, 0.42, 0.24}
\definecolor{copper}{rgb}{0.72, 0.45, 0.2}
\definecolor{aqua}{rgb}{0.5, 1.0, 0.83}
\definecolor{awesome}{rgb}{1.0, 0.13, 0.32}
\definecolor{candyapplered}{rgb}{1.0, 0.03, 0.0}
\definecolor{caribbeangreen}{rgb}{0.0, 0.8, 0.6}
\definecolor{indigo}{rgb}{0.0, 0.25, 0.42}

\DeclareMathOperator{\weaklystar}{\rightharpoonup\kern-2.2ex ^* \, \,}

\def\XXint#1#2#3{{\setbox0=\hbox{$#1{#2#3}{\int}$ }
\vcenter{\hbox{$#2#3$ }}\kern-.6\wd0}}

\DeclareMathOperator{\curl}{curl}

\newcommand{\R}{\mathbb R}
\newcommand{\N}{\mathbb N}
\newcommand{\Z}{\mathbb Z}
\renewcommand{\C}{\mathbb C}

\newcommand\norm[1]{\lVert #1 \rVert}

\newcommand\scpr{\boldsymbol{\cdot}}

\newcommand{\ra}{\rightarrow}

\newcommand{\mL}{\mathrm{L}}

\renewcommand{\phi}{\varphi}

\newcommand{\mH}{\mathrm{H}}

\newcommand{\ee}{\mathrm{e}}

\theoremstyle{plain}
\newtheorem{theorem}{Theorem}[section]

\newtheorem{corollary}[theorem]{Corollary}
\newtheorem{lemma}[theorem]{Lemma}

\newtheorem*{theorem*}{Theorem}

\theoremstyle{definition}
\newtheorem{definition}[theorem]{Definition}
\newtheorem{remark}[theorem]{Remark}
\newtheorem*{remark*}{Remark}

\usepackage{empheq}
\usepackage[pagebackref=false]{hyperref}
\hypersetup{
	linktoc=all,
	bookmarksdepth=3,
	colorlinks=true,
	linkcolor=Green,citecolor=Orange,urlcolor=DarkBlue,
	pdftitle={GPvortices-with-magnetic-field},
	pdfauthor={Thiago Carvalho Corso, Gaspard Kemlin, Christof Melcher, and Benjamin Stamm},
	pdfsubject={},
	pdfkeywords={}}
\begin{document}
\numberwithin{table}{section}
\title{Simulation of the magnetic Ginzburg-Landau equation via vortex tracking}

\author[T.~Carvalho~Corso]{Thiago Carvalho Corso}
\address[T.~Carvalho Corso]{Institute of Applied Analysis and Numerical Simulation, University of Stuttgart, Pfaffenwaldring 57, 70569 Stuttgart, Germany}
\email{thiago.carvalho-corso@mathematik.uni-stuttgart.de}

\author[G.~Kemlin]{Gaspard Kemlin}
\address[G.~Kemlin]{LAMFA, Universit\'e de Picardie Jules Verne and CNRS, UMR 7352, 80039 Amiens, France}
\email{gaspard.kemlin@u-picardie.fr}

\author[Ch.~Melcher]{Christof Melcher}
\address[Ch.~Melcher]{Applied Analysis and JARA FIT, RWTH Aachen University, 52056 Aachen, Germany}
\email{melcher@math1.rwth-aachen.de}

\author[B.~Stamm]{Benjamin Stamm}
\address[B.~Stamm]{Institute of Applied Analysis and Numerical Simulation, University of Stuttgart, Pfaffenwaldring 57, 70569 Stuttgart, Germany}
\email{best@ians.uni-stuttgart.de}


\date{\today}
\thanks{\emph{Funding information}: DFG -- Project-ID 442047500 -- SFB 1481.\\[1ex]
\textcopyright 2025 by the authors. Faithful reproduction of this article, in its entirety, by any means is permitted for noncommercial purposes.}
\begin{abstract}
This paper deals with the numerical simulation of the 2D magnetic time-dependent
Ginzburg-Landau (TDGL) equations in the regime of small but finite (inverse)
Ginzburg-Landau parameter $\epsilon$ and constant (order $1$ in $\epsilon$)
applied magnetic field. In this regime, a well-known feature of the TDGL
equation is the appearance of quantized vortices with core size of order
$\epsilon$. Moreover, in the singular limit $\epsilon \searrow 0$, these
vortices evolve according to an explicit ODE system. In this work, we first
introduce a new numerical method for the numerical integration of this limiting
ODE system, which requires to solve a linear second order PDE at each time step.
We also provide a rigorous theoretical justification for this method that
applies to a general class of 2D domains. We then develop and analyze a
numerical strategy based on the finite-dimensional ODE system to efficiently
simulate the infinite-dimensional TDGL equations in the presence of a constant
external magnetic field and for small, but finite, $\epsilon$. This method
allows us to avoid resolving the $\epsilon$-scale when solving the TDGL
equations, where small values of $\epsilon$ typically require very fine meshes
and time steps. We provide numerical examples on a few test cases and justify
the accuracy of the method with numerical investigations. We end the paper showing that, in the mixed flow case, the limiting ODE system
is able to capture the crystallization process in which, for large times, the
vortices arrange into a stable pattern.
\end{abstract}
\maketitle
\setcounter{secnumdepth}{3}

\section{Introduction}

\subsection{Motivation}

The Ginzburg-Landau (GL) theory has played a central role in the description of superconductivity since its introduction by
Ginzburg and Landau \cite{GL50}. By encoding macroscopic quantum behavior through a complex order parameter, the GL framework bridges microscopic physics
to observable phenomena such as permanent superconducting currents, the expulsion of external magnetic fields from the superconducting sample (Meissner effect), and the different phase transitions between superconducting, mixed, and normal states. Its time-dependent extension, the time-dependent Ginzburg–Landau
(TDGL) equation, was later developed to describe nonequilibrium processes in
superconductors \cite{GE68,Sch66}. Moreover, while originally proposed on phenomenological grounds, the GL theory was later justified from the microscopic Bardeen-Cooper-Schrieffer (BCS) theory \cite{BCS57,Gor59,FHSS12}. However, we note that a rigorous derivation of BCS theory from many-body quantum mechanics remains a major challenge in mathematical physics.

 A particularly striking feature of the GL framework is the emergence of
quantized vortices when the intensity of the applied magnetic field reaches
a critical threshold \cite{Abr57}. In superconductors, these vortices govern dissipation through
flux motion and pinning, thereby determining critical currents and resistance
under applied fields \cite{BS65,Tin96,ST11}. Therefore, their dynamics, interactions, nucleation, and annihilation are
central to understanding the dynamical response of type-II superconductors in
technological applications. Beyond its original setting, GL-type models
have also found applications in the study of Bose–Einstein condensates, where
they serve as effective mean-field descriptions of coherent matter waves (see, e.g., \cite{Aft06} for a mathematical overview). As such, the accurate simulation of vortex dynamics within the TDGL equation is of both fundamental and applied importance.

From a computational perspective, however, the resolution of vortex structures
poses significant challenges. The characteristic vortex core size is set by the
(inverse) Ginzburg–Landau parameter $\epsilon$, which may be very small in
physically relevant regimes, e.g., for superconductors of type II. Standard
finite element discretizations, while powerful in handling complex geometries,
require extremely fine meshes to capture the core structure when $\epsilon \ll
1$, leading to prohibitively large computational costs even in the static case \cite{DGP92,DH24,BDH25,DDH25}. This
difficulty is amplified in time-dependent simulations, where vortices may move
and interact dynamically. Despite the central importance of vortex dynamics,
relatively few numerical methods are designed specifically for the
small-$\epsilon$ regime (see \cite{CKMS25} and references therein) and the
development of efficient and accurate approaches for TDGL simulations in this setting remains an open and difficult problem \cite{Du97,Du05,Bar05a,Bar05b,BMO11}.
In the specific regime of a constant external magnetic field $h_{\rm ex} = \mathcal{O}(1)$ and $\epsilon\searrow0$, the vortices move according to a finite dimensional ODE system (see for instance \cite{KS11} for a rigorous derivation of this limiting system, and references therein). Extending our previous work \cite{CKMS25} in the simpler case without magnetic field, we present in this paper a novel numerical method to tackle the simulation of the TDGL equations in the regime of small, but finite, $\epsilon$ under a constant external magnetic field. This method is designed to achieve higher accuracy for smaller values of $\epsilon$.

\subsection{Main contribution}

{The main contributions of this article can be summarized as follows:
\begin{enumerate}[label=(\roman*)]
\item We propose and implement a numerical method to compute approximate solutions of the time-dependent Ginzburg-Landau equation with magnetic fields in the regime of small but finite $\epsilon$. This method builds on rigorous analytical results concerning the dynamics of the PDE solutions in the limit as $\epsilon \searrow 0$. In particular, it avoids the use of extremely fine meshes that are required by standard FEM methods to resolve the vortex core.
\item We propose and implement a numerical scheme to efficiently simulate the vortex dynamics in the limit $\epsilon \searrow 0$.
Moreover, we provide a theoretical justification for this scheme that applies to general Lipschitz domains.
\end{enumerate}}

\subsection{Outline of the paper}

The paper is structured as follows. In the next section, we introduce the
necessary notations for precisely stating the TDGL equations and highlight some
of its basic properties. We then recall some rigorous results on the asymptotic
behavior of solutions to the TDGL equations in the limit where the inverse
Ginzburg-Landau parameter $\epsilon$ vanishes and the applied magnetic field is
kept constant. In Section~\ref{sec:ODE dynamics}, we describe the numerical
scheme for simulating the asymptotic dynamics of vortices. In addition, we
provide a theoretical justification for this numerical scheme, present numerical
simulations on the vortex trajectories, and numerically analyze the convergence
rates with respect to the discretization parameters. In
Section~\ref{sec:numerical scheme}, we present the numerical scheme to
efficiently approximate solutions of the TDGL equation in the regime of small
but finite $\epsilon$. We then illustrate our method with some numerical
experiments, and we compare the results with reference
solutions computed via finite elements to numerically justify the accuracy of
our method. Finally, in Section~\ref{sec:mixed-flow} we show that the limiting ODE system under a mixed-flow -- \emph{i.e.} adding dissipation
  -- is capable of producing vortices that arrange after some time into a stable
  lattice pattern, provided that the external magnetic field is large enough.

\section{The time-dependent magnetic Ginzburg-Landau equation}
\label{sec:GL theory}

Let us now introduce the necessary notations to state the time-dependent Ginzburg-Landau (TDGL) equations and recall a few basic facts about these equations. We then recall some important theoretical results regarding the asymptotic behavior of the solutions in the regime of interest to us. These results form the basis for our numerical method, which will be described in Section~\ref{sec:numerical scheme}.

Throughout this paper, we let $\Omega \subset \R^2$ be an open bounded and
simply connected subset with smooth boundary. In the phenomenological
description of superconductors via the Ginzburg-Landau theory, there are two
main variables of interest: the so-called order parameter $u: \Omega \rightarrow
\C$, describing the density of superconducting Cooper pairs, and the associated
magnetic vector potential $A: \Omega \rightarrow \R^2$ inside the
superconducting sample $\Omega$. The Ginzburg-Landau energy of the pair $(u,A)$
then reads
\begin{align}
    GL_\epsilon(u,A) = \frac12 \int_\Omega |\nabla_A u(x)|^2 + |\curl A(x) - h_{\rm ex}|^2 + \frac{(1-|u(x)|^2)^2}{2\epsilon^2} \mathrm{d} x, \label{eq:GLenergy}
\end{align}
where $\nabla_A u \coloneqq \nabla u - \ii A u$ is the covariant gradient,
$\curl A \coloneqq \partial_1 A_2 - \partial_2 A_1$ is the rotational of $A$
(a real number since the domain is two dimensional), and $h_{\rm ex}\in\R$ is
the applied external magnetic field. Typically, $h_{\rm ex}$ is assumed constant
throughout the sample $\Omega$.

The time evolution of the order parameter and magnetic field can then be modeled (see \cite{GE68,Sch66,Dor92}) via the following time-dependent Ginzburg-Landau equations:
\begin{align}
    &\delta_\epsilon \left(\partial_t u + \ii \Phi u\right) = - \nabla_u GL_\epsilon(u,A) = \nabla_A^2 u + \frac{1}{\epsilon^2}u (1-|u|^2),\label{eq:TDGL} \\
    &\sigma_\epsilon \left(\partial_t A + \nabla \Phi\right) = - \nabla_A GL_\epsilon(u,A) = -\curl h + (\ii u, \nabla_A u), \label{eq:TDmagnetic}
\end{align}
where $h = \curl A$ is the magnetic field induced by $A$, $(u,v) = \mathrm{Re}
(u \overline{v}) = \frac{u \overline{v} + \overline{u} v}{2}$ is the standard
inner-product on $\R^2 \simeq \C$, $\curl h = -\nabla^\perp h
= (\partial_2 h, - \partial_1 h)^T$ is the curl of the scalar magnetic field,
$\delta_\epsilon = \alpha_\epsilon + \ii
\beta_\epsilon$ is a complex relaxation parameter and $\sigma_\epsilon > 0$ is
the conductivity. In addition, we impose what is called the natural boundary
conditions\footnote{The term natural boundary conditions refers to the fact that Neumann boundary conditions are automatically satisfied by minimizers of the GL energy $GL_\epsilon(u,A)$ without any boundary constraints, i.e., in the spaces defined in~\eqref{eq:minimization spaces}-\eqref{eq:u space}. }
\begin{align}
    \nu \scpr \nabla_A u = 0, \quad \nu \scpr (\partial_t A + \nabla \Phi) = 0, \quad\mbox{and}\quad h = h_{\rm ex} \quad \mbox{on $\partial \Omega$.} \label{eq:naturalbc}
\end{align}
The parameter $\epsilon>0$ is a material constant. In this work, we are interested in the regime of small $\epsilon>0$, which corresponds to superconductors of type II.

At first sight, it may seem that the electric potential $\Phi$ is an additional
variable of the problem. However, note that by taking the divergence of
\eqref{eq:TDmagnetic}, the electric potential is, up to a constant, uniquely
determined as the solution of the following Poisson equation with Neumann
boundary conditions, where $j_A(u)  = (\ii u, \nabla_A u)$ is the
super-current:
\begin{align}
\Delta \Phi = \frac{1}{\sigma_\epsilon} \mathrm{div} \, j_A(u)\quad \mbox{in $\Omega$}, \quad \partial_\nu \Phi  = \frac{1}{\sigma_\epsilon} j_A(u) \scpr \nu \quad \mbox{in $\partial \Omega$.}
\end{align}
Moreover, note that the GL energy, and therefore the space of solutions to the GL equations, is invariant under the gauge transformation
\begin{align}\label{eq:gauge}
	u \mapsto u\ee^{\ii t \chi}, \quad A \mapsto A + \nabla \chi, \quad \Phi \mapsto \Phi - \partial_t \chi,
\end{align}
for a (sufficiently regular) function $\chi : \R_+ \times \Omega \rightarrow \R$. Nevertheless, note that all of the physical quantities of interest for the model, namely the density $|u|^2$, the super-current $j_A(u)$, the induced electric field $E = \partial_t A +\nabla \Phi$, and the induced magnetic field $h = \mathrm{curl}\, A$, are invariant under this gauge transformation. In particular, we are free to choose the most convenient gauge for calculations. Following \cite{Spi03,KS11}, we shall work with the Coulomb gauge here, i.e., we impose the additional conditions:
\begin{align*}
\mathrm{div}\, A = 0 \quad \mbox{in $\Omega$,} \quad A \scpr \nu = 0 \quad \mbox{in $\partial \Omega$} \quad \mbox{and}\quad \int_\Omega \Phi =0.
\end{align*}
Under these conditions, equation~\eqref{eq:TDGL} is globally well-posed in the energy space, see e.g., \cite[Remark 1.3]{KS11} and \cite{TW95}.

\subsection{Asymptotic regime and limiting dynamics}

In this paper, we are mostly interested in the regime of small $\epsilon>0$. In this regime, one typically observes the formation of so-called vortices, which correspond to isolated zeros of the density $|u|^2$ with a non-zero winding number (or degree). A central topic in the study of superconductors is the description of the dynamics of such vortices. In particular, there are several works dedicated to deriving and justifying such dynamics with various degrees of mathematical rigor, see, e.g. \cite{PR93, E94, TW95, LD97, OS98,GS06,SS04,SS07} and references therein. In this section, we recall the results from \cite{Spi02,Spi03,KS11}, which apply to the model equation studied here.

To this end, let us introduce the so-called renormalized energy which corresponds to the next-order contribution to the energy for minimizers of the GL energy with a given vortex configuration. Precisely, let $n\in \N$ and let $\Omega^n_\ast$ be the set
\begin{align*}
    \Omega_\ast^n \coloneqq \{ a = (a_1,a_2,..,a_n) \in \Omega^n : a_j \neq a_i \quad \mbox{for any $i \neq j$.} \}.
\end{align*}
Then, we can define the renormalized energy of a given vortex configuration $(a,d) \in \Omega_\ast^n \times \Z^d$ as
\begin{align}
    W_\Omega(a,d;h_{\rm ex}) = \lim_{\rho \searrow 0}  \inf_{\substack{u\in \mathcal{H}^1_\rho(a) \\ A\in \mH^1_{\rm div}(\Omega)}} \left\{ \frac12\int_{\Omega_\rho(a)} |\nabla_A u|^2 + \int_\Omega |\mathrm{curl}\, A - h_{\rm ex}|^2 \mathrm{d} x\right\} -\pi \sum_{j=1}^n d_j^2 \log \frac{1}{\rho} \label{eq:renormalized variational}
\end{align}
where $\Omega_\rho(a) \coloneqq \Omega \setminus \cup_{j=1}^n B_\rho(a_j)$,
\begin{align}
    &\mH^1_{\mathrm{div}}(\Omega) \coloneqq \{ A \in \mH^1(\Omega;\R^2) : \mathrm{div} A =0 \mbox{ on $\Omega$,}\quad  A\scpr \nu = 0 \mbox{ on $\partial \Omega$}\},\label{eq:minimization spaces} \\
    \intertext{and}
    &\mathcal{H}^1_\rho(a) \coloneqq \{ u \in \mH^1(\Omega_\rho(a);\mathbb{S}^1) : \quad \mathrm{deg}(u,a_j) = d_j \quad \mbox{for any $j=1,...,n$}\}. \label{eq:u space}
\end{align}
Here, $\mathrm{deg}(u,a_j)$ denotes the topological degree of $u$ around $a_j$,
\begin{align*}
    \mathrm{deg}(u,a_j) \coloneq \frac{1}{2\pi} \int_{\partial B_\rho(a_j)} \frac{1}{|u|^2}(\ii u, \partial_\tau u),
\end{align*}
which is well-defined for any non-vanishing $u\in \mH^{\frac12}(\partial B_\rho(a_j);\C)$ (see, e.g., \cite{BM21}).
By the results in \cite[Appendix A.1]{Spi03} (see also \cite[Section
4.1]{LD97}), the renormalized energy can be written more
explicitly\footnote{Here we adopt the sign convention for $\Xi$ in \cite{KS11}.
  Note that this is the opposite sign convention from \cite{Spi03,LD97}; in
  particular, the terms depending on $\Xi$ in the renormalized energy appear
  with the opposite sign there.} as (see Appendix~\ref{app:W} for
  more details)
\begin{align}
    W_\Omega(a,d;h_{\rm ex}) = - \pi \sum_{i \neq j} d_i d_j \log |a_i-a_j| - \pi \sum_{j=1}^n d_j \left(R(a_j) + \Xi(a_j)\right) +  \frac{h_{\rm ex}^2}{2} |\Omega| + \frac{h_{\rm ex}}{2} \int_{\partial \Omega} \partial_\nu \Xi, \label{eq:renormalized explicit}
\end{align}
where $R$ is the harmonic function satisfying
\begin{align}
    \Delta R = 0 \quad \mbox{in $\Omega$} \quad \mbox{and}\quad R(x) = - \sum_{j=1}^n d_j \log |x-a_j|, \quad x \in \partial \Omega, \label{eq:Req0}
\end{align}
and $\Xi$ solves
\begin{align}
    \Delta^2 \Xi- \Delta \Xi  = 2\pi \sum_{j=1}^n d_j \delta_{a_j} \quad \mbox{in $\Omega$,} \quad \Xi = 0 \quad \mbox{on $\partial \Omega$,} \quad \mbox{and}\quad \Delta \Xi = - h_{\rm ex} \quad \mbox{on $\partial \Omega$.} \label{eq:xieq0}
\end{align}

Moreover, let us introduce the constant $\gamma_0>0$ defined (see \cite{BBH94}) as
\begin{align}
    \gamma_0 \coloneqq \lim_{r \ra \infty} I(r,\epsilon) - \pi \log \frac{r}{\epsilon}, \label{eq:renormalized constant}
\end{align}
where
\begin{align*}
    I(r,\epsilon) = \inf \left\{ \frac{1}{2} \int_{B_r}  |\nabla u|^2 + \frac{(1-|u|^2)^2}{2\epsilon^2} : u \in H^1(B_r), u= \ee^{\ii \theta} \mbox{on $\partial B_r$}\right\},
\end{align*}
and define the gauged Jacobian as
\begin{align}
    J_A(u) \coloneqq \frac12 \curl j_A(u) + \frac12 \curl A.\label{eq:Jacobian}
\end{align}

We can now introduce the following definition of well-prepared states, that is,
a sequence of $\epsilon$-dependent states whose energy and vorticity is
compatible with a given vortex configuration.

\begin{definition}[Well-prepared states] \label{def:wellprepared} Let $\{(u_\epsilon, A_\epsilon)\}_{\epsilon>0}\subset H^1(\Omega) \times H^1(\Omega; \R^2)$ be a sequence of functions. Then we say that $\{(u_\epsilon, A_\epsilon)\}_{\epsilon>0}$ are well-prepared with respect to $a \in \Omega^n_\ast$ and $d \in \{\pm 1\}^n$ for some $n\in \N$ if it holds that
\begin{align}
    &GL_\epsilon(u_\epsilon,A_\epsilon; h_{\rm ex}) = n \left(\pi \log \frac{1}{\epsilon}  - \gamma_0\right) + W_\Omega(a,d;h_{\rm ex}) + o(1) \intertext{and}
    & \lim_{\epsilon \searrow 0} \norm{J_{A_\epsilon}(u_\epsilon) - \pi \sum_{j=1}^n d_j \delta_{a_j}}_{\dot{W}^{-1,q}} = 0, \label{eq:energybound}
\end{align}
for some $q\geq 1$, where $W_\Omega$ is the renormalized energy~\eqref{eq:renormalized explicit} and $\gamma_0$ is the constant defined in~\eqref{eq:renormalized constant}.
\end{definition}

The main result of \cite{KS11} is that such well-preparedness is preserved along
time, which can be stated as follows:
\begin{theorem}[Theorem 1.1 in \cite{KS11}] \label{thm:vortices converg} Let $(u_\epsilon, A_\epsilon , \Phi_\epsilon) : \Omega \times [0,\infty) \rightarrow \C \times \R^2 \times \R$ be a sequence of solutions of
\begin{align}\label{eq:GPE_mag}
    \begin{dcases}
    &\delta_\epsilon \left(\partial_t u_\epsilon + \ii \Phi_\epsilon u_\epsilon\right) = \nabla_{A_\epsilon}^2 u_\epsilon + \frac{1}{\epsilon^2}u_\epsilon (1-|u_\epsilon|^2),\\
    &\sigma_\epsilon \left(\partial_t A_\epsilon + \nabla \Phi_\epsilon\right) =
    -\curl h_\epsilon + (\ii \,u_\epsilon, \nabla_{A_\epsilon} u_\epsilon), \end{dcases}  \quad \mbox{on $\Omega \times \R_+$.}
\end{align}
with $h_\epsilon =\curl A_\epsilon$ and boundary conditions
\begin{align}
    \begin{dcases} &\nu \scpr \nabla_{A_\epsilon} u_\epsilon = 0, \\
    &\nu \scpr (\partial_t A_\epsilon + \nabla\Phi_\epsilon) = 0, \\
    &h_\epsilon = h_{\rm ex},
    \end{dcases} \quad \mbox{on $\partial \Omega$}
\end{align}
where the initial conditions $\{(u_\epsilon^0, A_\epsilon^0)\}_{\epsilon >0}$ are well-prepared with respect to some $a \in \Omega^n_\ast$ and $d \in \{\pm 1\}^n$. Moreover, we assume that
\begin{align*}
    \delta_\epsilon = \frac{\alpha_0}{|\log \epsilon|} + \ii \beta_0 \quad \mbox{and}\quad \frac{\sigma_0}{|\log \epsilon|} \leq \sigma_\epsilon = o(1),
\end{align*}
where $\alpha_0, \beta_0 \geq 0$ and $\sigma_0 >0$. Then for any time $t \leq T_{\max}$, where $T_{\max}>0$ denotes the maximal time of existence of the solution $a(t) :[0, T_{\max}) \rightarrow \Omega^n_\ast$ of the ODE
\begin{align}
    \begin{dcases} \left(\alpha_0 - d_j \beta_0 \mathbb{J}\right)\dot{a}(t) = -
      \frac{1}{\pi} \nabla_{a_j} W_\Omega\left(a(t),d; h_{\rm ex}\right) \quad
      \mbox{for $t>0$}, \\ a(0) = a, \end{dcases} \quad\mbox{where}\quad
    \mathbb{J} = \begin{pmatrix} 0 & 1 \\ -1 & 0\end{pmatrix},\label{eq:ODE}
\end{align}
it holds that $\left(u(t),A(t)\right)$ are well-prepared with respect to $(a(t), d)$, i.e.,
\begin{enumerate}[label=(\roman*)]
\item (Jacobian) The gauged Jacobian $J_{A(t)}(u(t))$ satisfies
\begin{align*}
    \lim_{\epsilon \searrow 0} \norm{J_{A_\epsilon(t)}(u_\epsilon(t)) - \pi \sum_{j=1}^n d_j \delta_{a_j(t)}}_{\dot{W}^{1,p}} =0 , \quad \mbox{for any $1\leq p <2$.}
\end{align*}
\item (Propagation of energy bound) We have
\begin{align*}
    GL_\epsilon\left(u_\epsilon(t), A_{\epsilon}(t);h_{\rm ex}\right) = n \left(\pi \log\frac{1}{\epsilon} - \gamma_0\right) + W_{\Omega}(a(t),d;h_{\rm ex}) + o(1).
\end{align*}
\end{enumerate}
\end{theorem}

\begin{remark*}[High applied magnetic fields] While in \cite{KS11}, the authors consider $h_{\rm ex} = \mathcal{O}(1)$, in \cite{SS04} the authors consider magnetic fields of order $h_{\rm ex} = \mathcal{O}(|\log \epsilon|)$ in the overdamped regime $\alpha=1$ and $\beta=0$. In this case, however, the magnetic field interaction dominates so that the limiting vortex dynamics has no vortex-vortex interaction.
\end{remark*}

\begin{remark*}[The Schr\"odinger and heat flows]
  Strictly speaking, only the case $\alpha_0, \beta_0>0$ (mixed-flow) is treated in \cite{KS11}. The case $\alpha_0 >0$ and
  $\beta_0=0$ (heat-flow) is addressed in \cite{Spi02}, and the case $\alpha_0 =0$ and
  $\beta_0 >0$ (Schr\"odinger-flow) can be found in \cite{Spi03}.
\end{remark*}

As a corollary of Theorem~\ref{thm:vortices converg}, one can also show that the super-current and the magnetic field converge to their limiting counterparts. This result is not explicitly stated in \cite[Theorem 1.1]{KS11} but follows from their proof, see \cite[Section 2.3]{KS11}.
\begin{corollary}[Convergence of magnetic field] Under the assumptions of Theorem~\ref{thm:vortices converg}, let $A_\epsilon$ be the solution satisfying the Coulomb gauge and $h_\epsilon = \mathrm{curl} \, A_\epsilon$, then we have
\begin{align*}
	\lim_{\epsilon \searrow 0} \norm{A_\epsilon(t) - A_\ast(t)}_{\mL^p(\Omega)} = 0, \quad \mbox{and}\quad \lim_{\epsilon \searrow 0}\norm{h_\epsilon(t) - h_\ast(t)}_{\mL^p(\Omega)} = 0,
\end{align*}
for any $0 <t< T_{\max}$ and $1\leq p <\infty$
where $h_\ast \coloneqq - \Delta \Xi$, $A_\ast \coloneqq \mathrm{curl} \, \Xi$, and $\Xi$ is the solution of~\eqref{eq:xieq0}.
\end{corollary}

\begin{remark}[Canonical harmonic map and convergence of super current]
  \label{rmk:canon}
  Similarly, it follows from \cite[Theorem 4.1]{KS11} that the super-current of the sequence of solutions $u_\epsilon(t)$ converge to the super-current of the canonical harmonic map with vortex configuration $(a(t),d)$. More precisely, the canonical harmonic map with vortex configuration $(a,d)$ is the unique (up to a constant phase factor) $\mathbb{S}^1$-valued map $u_\ast$ satisfying
\begin{align*}
	\mathrm{curl}\, j(u_\ast) = 2\pi \sum_{j=1}^n d_j \delta_{a_j}, \quad \mathrm{div}\, j(u_\ast) = 0, \quad \mbox{and}\quad \partial_\nu u_\ast = 0 \quad \mbox{on $\partial \Omega$}.
\end{align*}
Thus from \cite[Theorem 4.1, eq. (4.8)]{KS11}, we have
\begin{align*}
	\lim_{\epsilon \searrow 0} \, \norm{j_{A_\epsilon}(u_\epsilon)(t) - j_{A_\ast}(u_\ast)(t)}_{\mL^{4/3}(\Omega)} = 0, \quad \mbox{for any $0<t< T_{\max}$.}
\end{align*}
\end{remark}

\section{The limiting vortex dynamics}
\label{sec:ODE dynamics}

We now describe the procedure to numerically simulate the limiting ODE dynamics
\begin{align}
    \begin{dcases} \left(\alpha_0 - d_j \beta_0 \mathbb{J}\right)\dot{a}(t) = - \frac{1}{\pi} \nabla_{a_j} W_\Omega\left(a(t),d; h_{\rm ex}\right) \quad \mbox{for $t>0$}, \\ a(0) = a, \end{dcases} \label{eq:ODE1}
\end{align}
which is the central step in our method.

The starting point of this procedure is the following explicit formula
for the gradient of $W_\Omega(a,d;h_{\rm ex})$. This formula is not immediate from~\eqref{eq:renormalized explicit} but a rigorous derivation can be found in \cite[Appendix A.2]{Spi03}.

\begin{lemma}[Gradient of renormalized energy] \label{lem:renormalized gradient} Let $a \in \Omega^n_\ast$, $d\in \Z^n$ and $W_{\Omega}(a,d,h_{\rm ex})$ be the renormalized energy defined in~\eqref{eq:renormalized explicit}. Then we have
\begin{align}
    \nabla_{a_k} W_{\Omega}(a,d;h_{\rm ex}) = -2 \pi \sum_{j\neq k} d_j d_k \frac{a_k-a_j}{|a_k-a_j|^2} - 2\pi d_k \nabla R(a_k) - 2\pi d_k \nabla \Xi(a_k), \label{eq:simplified gradient 0}
\end{align}
where $R: \Omega \rightarrow \R$ is the unique solution of
\begin{align}
    \begin{dcases} \Delta R = 0 \quad &\mbox{in $\Omega$,}\\
    R(x) = - \sum_{j=1}^n d_j \log |x-a_j|, \quad &\mbox{for $x\in \partial \Omega$,} \end{dcases}  \label{eq:Req}
\end{align}
and $\Xi: \Omega \rightarrow \R$ is the unique solution of
\begin{align}
\quad \begin{dcases} \Delta^2 \Xi - \Delta \Xi = 2\pi \sum_{j=1}^n d_j \delta_{a_j}, \quad &\mbox{in $\Omega$,}\\
    \Xi = 0, \quad &\mbox{on $\partial \Omega$,}\\
    \Delta \Xi = -h_{\rm ex}, \quad &\mbox{on $\partial \Omega$.}
    \end{dcases} \label{eq:Xieq}
\end{align}
\end{lemma}
Note that formula~\eqref{eq:simplified gradient 0} requires, at each time step,
the solution of two linear PDEs: a Laplace problem for the function $R$ defined
in~\eqref{eq:Req0} and a modified biharmonic equation for the function $\Xi$
defined in~\eqref{eq:xieq0}. As thoroughly explained in \cite{CKMS25}, one can
efficiently solve the Laplace problem by projecting the boundary conditions on
the space of harmonic polynomials. However, as shown later in
Lemma~\ref{lem:simplified gradient}, it turns out that we do not need to compute
$R$ to evaluate the gradient of $W_\Omega$. Therefore, let us focus now on the
numerical scheme for solving the modified biharmonic equation~\eqref{eq:Xieq}.

\subsection{The modified biharmonic equation} To solve the modified biharmonic equation~\eqref{eq:Xieq}, we first transform it into a homogeneous equation with nonhomogeneous boundary conditions. For this, let us recall (see, e.g., \cite[Equation 18]{HJ18})
that a fundamental solution of the operator $\Delta^2 - \Delta$ is given by
\begin{align}
    G(x,y) = G_{\Delta-1}(x,y) -G_{\Delta}(x,y)= -\frac{1}{2\pi}(K_0(r) + \log (r)), \label{eq:simply GF}
\end{align}
with $r=|x-y|$ and where $K_0$ is the modified Bessel function of the second kind of order $0$. This can be easily seen from the fact that $\Delta^2 - \Delta = \Delta (\Delta-1) = (\Delta-1)\Delta$ and therefore
\begin{align*}
    -(\Delta^2 - \Delta)\frac{1}{2\pi}(\log(r) +K_0(r)) &= \Delta (\Delta-1) \left(\frac{-K_0(r)}{2\pi}\right) -(\Delta-1)\Delta\left(\frac{1}{2\pi} \log(r)\right) \\
    &=  \Delta (\delta_y) -(\Delta-1)(\delta_y) = \delta_y,
\end{align*}
where we used the fact that $\frac{1}{2\pi} \log(r)$ and $-\frac{1}{2\pi} K_0(r)$ are, respectively, fundamental solutions of $\Delta$ and $\Delta-1$.

With the fundamental solution at hand, we can remove the singular source term in~\eqref{eq:Xieq} at the cost of adding a smooth boundary term. More precisely, if we define
\begin{align}
    \Xi_p(x) \coloneqq 2\pi \sum_{j=1}^n d_j \left(G \ast \delta_{a_j} \right)(x)
    = - \sum_{j=1}^n d_j(K_0 + \log)(|x-a_j|) \quad \mbox{for $x \in \R^2$}, \label{eq:g function}
\end{align}
then the solution of~\eqref{eq:Xieq} is given by $\Xi = \Xi_h + \Xi_p$ where $\Xi_h$ solves
\begin{align}
    \begin{dcases} \Delta^2 \Xi_h- \Delta \Xi_h = 0 \quad  &\mbox{in $\Omega$}, \\
    \Delta \Xi_h = -h_{\rm ex} - \Delta \Xi_p,\quad &\mbox{on $\partial \Omega$,} \\
    \Xi_h = - \Xi_p \quad &\mbox{on $\partial \Omega$.}
    \end{dcases} \label{eq:Weq}
\end{align}
The function $\Xi_h$ can be further decomposed as follows.

\begin{lemma}[Homogeneous solution] \label{lem:W function} Let $\Xi_h$ be the unique solution of~\eqref{eq:Weq} in $\mH^2(\Omega)$, then we have
\begin{align*}
   \Xi_h = F - R + h_{\rm ex}
\end{align*}
where $R$ and $F$ are respectively the unique solutions in $\mH^1(\Omega)$ of~\eqref{eq:Req} and
\begin{subequations}\label{eq:Feq}
\begin{empheq}[left=\empheqlbrace]{align}
  &(\Delta - 1)F = 0,  &\text{in } \Omega, \label{eq:Feq_a}\\
  &F = -h_{\rm ex} + \sum_{j=1}^n d_j K_0(|x - a_j|),  &\text{on } \partial\Omega. \label{eq:Feq_b}
\end{empheq}
\end{subequations}
\end{lemma}

\begin{remark}[Uniqueness for the modified biharmonic equation in
  $\mH^2(\Omega)$] \label{rem:weak solution} Note that the boundary trace of
  $\Delta \Xi_h$ is a priori not well-defined for merely $\mH^2(\Omega)$
  functions. However, ~\eqref{eq:Weq} can be weakly formulated as follows: find
  $\Xi_h\in \mH^2(\Omega)$ such that $\Xi_h = -\Xi_p$ in $\partial\Omega$ and
\begin{align}
    \int_\Omega \Delta \Xi_h(\Delta - 1) \psi = -\int_{\partial \Omega} \partial_\nu \psi(x) \left(h_{\rm ex} + \Delta \Xi_p(x) \right) \mathscr{H}^1(\mathrm{d} x), \quad \mbox{for any $\psi \in \mH^1_0(\Omega) \cap \mH^2(\Omega)$,} \label{eq:weak Weq}
\end{align}
where $\partial_\nu \psi$ denotes the Neumann trace of $\psi$ along $\partial
\Omega$. The existence and uniqueness of solutions of~\eqref{eq:weak Weq} then
holds in any bounded Lipschitz domain $\Omega$, see
  Appendix~\ref{app:biharm} for more details.
\end{remark}
\begin{proof}
Let $F\in \mH^1(\Omega)$ be a weak solution of~\eqref{eq:Feq} and $R\in \mH^1(\Omega)$ be a weak solution of~\eqref{eq:Req}. Since the boundary functions are smooth on a neighbourhood of the boundary, it follows from standard elliptic regularity that $F, R \in \mH^2(\Omega)$. Thus, let us define $\widetilde{\Xi}_h = F - R + h_{\rm ex} \in \mH^2(\Omega)$. Then in the distributional sense we have
\begin{align*}
    \Delta (\Delta-1) \widetilde{\Xi}_h = \Delta (\Delta -1) F - (\Delta -1) \Delta R = 0 \quad \mbox{in $\Omega$.}
\end{align*}
Since
\begin{align}
    \Delta \widetilde{\Xi}_h = \Delta F = F, \label{eq:traceid}
\end{align}
we have $\Delta \widetilde{\Xi}_h \in \mH^2(\Omega)$ and therefore its boundary values are well-defined. Moreover, since
\begin{align}
    \Delta \Xi_p = \sum_{j=1}^n d_j 2\pi \Delta \left(G_{\Delta-1} - G_\Delta \right) \ast \delta_{a_j} = \sum_{j=1}^n d_j 2\pi G_{\Delta-1} \ast \delta_{a_j} = -\sum_{j=1}^n  d_j K_0(|x-a_j|), \label{eq:g Laplace}
\end{align}
from the boundary condition for $F$, we see that
\begin{align*}
    \Delta \widetilde{\Xi}_h \rvert_{\partial \Omega} = F\rvert_{\partial \Omega} = - h_{\rm ex} + \sum_{j=1}^n d_j K_0(|x-a_j|) = - h_{\rm ex} - \Delta \Xi_p.
\end{align*}
Similarly, since
\begin{align}
    \Xi_p(x) = -\sum_{j=1}^n d_j \left(\log(|x -a_j|) + K_0(|x-a_j|)\right), \label{eq:g expression}
\end{align}
we have
\begin{align*}
    \widetilde{\Xi}_h\rvert_{\partial \Omega} = F\rvert_{\partial \Omega} - R\rvert_{\partial \Omega} + h_{\rm ex} = \sum_{j=1}^n d_j\left( K_0(|x-a_j|) + \log(|x-a_j|)\right) = - \Xi_p.
\end{align*}
Hence, $\widetilde{\Xi}_h$ solves~\eqref{eq:Weq} (and ~\eqref{eq:weak Weq}) and therefore $\widetilde{\Xi}_h = \Xi_h$ by uniqueness of the solution.
\end{proof}

 \subsection{Simplified formulas for gradient, renormalized energy and magnetic field} We now use the decomposition in Lemma~\ref{lem:W function} to simplify some of the equations presented before. These simplified versions are the ones implemented in our algorithm.

 We start with a simplified formula for the gradient of $W_\Omega$. This formula shows that only the function $F$ is necessary to simulate the ODE dynamics.
 \begin{lemma}[Simplified gradient] \label{lem:simplified gradient} Let $a\in \Omega^n_\ast$, $d \in \{\pm 1\}^n$, and $F$ be the solution of~\eqref{eq:Feq}. Then
\begin{align}
    \nabla_{a_k} W_\Omega(a,d;h_{\rm ex}) = -2\pi \sum_{j\neq k} d_j d_k \frac{a_k-a_j}{|a_k-a_j|} K_1(|a_j-a_k|) - 2\pi d_k \nabla F(a_k), \label{eq:simplified gradient 1}
\end{align}
where $K_1$ is the modified Bessel function of the second kind of order $1$.
\end{lemma}
\begin{proof} Recalling~\eqref{eq:g expression}, we see that
\begin{align*}
    \nabla \Xi_p(x) =  -\sum_{j=1}^n d_j  \frac{x-a_j}{|x-a_j|}\left(\frac{1}{|x-a_j|} + \dot{K}_0(|x-a_j|)\right)  , \quad \mbox{for $x \not \in a$.}
\end{align*}
Next, note that by well-known identities for Bessel functions (see, e.g., \cite[Section 8.446 and equation 18 in 8.486]{GR07})
\begin{align*}
    \dot{K}_0(r) = -K_1(r), \quad \mbox{and}\quad \lim_{r \searrow 0} K_1(r) - \frac{1}{r} = 0.
\end{align*}
Therefore,
\begin{align}
    \nabla \Xi_p(a_k) &= -\lim_{x\rightarrow a_k} \sum_{j =1 }^n d_j \frac{x-a_j}{|x-a_j|}\left( \frac{1}{|x-a_j|} - K_1(|x-a_j|)\right)\nonumber \\
    &= - \sum_{j\neq k} d_j \frac{a_k-a_j}{|a_k-a_j|}\left(\frac{1}{|a_k-a_j|} - K_1(|a_k-a_j|)\right). \label{eq:middle}
\end{align}
As the solution of~\eqref{eq:Xieq} is given by $\Xi = \Xi_p  + \Xi_h$, where $\Xi_h = F - R + h_{\rm ex}$ by Lemma~\ref{lem:W function}, we conclude from~\eqref{eq:simplified gradient 0} and~\eqref{eq:middle} that
\begin{align*}
    \nabla_{a_k} W_{\Omega}(a,d;h_{\rm ex}) &= - 2\pi \sum_{j\neq k} d_j d_k  \frac{a_k-a_j}{|a_k-a_j|^2} - 2\pi d_k \nabla \Xi_p(a_k) - 2\pi d_k \nabla \Xi_h(a_k) - 2\pi d_k \nabla R(a_k),\\
    &= -2\pi \sum_{j\neq k} d_j d_k \frac{a_k-a_j}{|a_k-a_j|} K_1(|a_k-a_j|) - 2\pi d_k \nabla F(a_k).
\end{align*}
\end{proof}

Similarly, we can show that only the function $F$ is necessary to evaluate the renormalized energy.
\begin{lemma}[Renormalized energy] Let $(a,d) \in \Omega_\ast^n \times \Z^d$ and $F$ be the solution of~\eqref{eq:Feq}, then the renormalized energy of $(a,d)$ can be written as
\begin{multline}
    W_\Omega(a,d;h_{\rm ex}) = \pi \sum_{i \neq j} d_i d_j K_0(|a_i-a_j|) - \pi \sum_{j=1}^n d_j \left(F(a_j)+h_{\rm ex} - d_j(\log(2) - \gamma)\right) + \frac{h_{\rm ex}^2}{2} |\Omega| \\
    +\frac{h_{\rm ex}}{2} \int_\Omega \left(F(x) - \sum_{j=1}^n d_j K_0(|x-a_j|)\right) \mathrm{d} x,
\end{multline}
where $K_0$ is the modified Bessel function and $\gamma \approx 0.57721$ is the Euler-Mascheroni constant.
\end{lemma}
\begin{proof} The proof follows from a straigthforward calculation by plugging the following identities in~\eqref{eq:renormalized explicit}: $\Xi +R = F + \Xi_p + h_{\rm ex}$,
\begin{align*}
    \int_{\partial \Omega} \partial_\nu \Xi  = \int_{\Omega} \Delta \Xi = \int_{\Omega} F+ \Delta \Xi_p,
\end{align*}
equation~\eqref{eq:g Laplace}, and
\begin{align*}
    \Xi_p(a_j) = -\lim_{ x\ra a_j} \sum_{i=1} d_i (K_0+\log)(x) = - d_i \left(\log(2) - \gamma\right) - \sum_{i \neq j} d_i (K_0+\log)(|a_i-a_j|),
\end{align*}
where the last identity follows from the well-known asymptotics $K_0(r) = -\log(r) + \log(2) - \gamma +  o(1)$ as $r\searrow 0$ (see \cite[Section 8.443]{GR07}).
\end{proof}

Lastly, we also obtain a simple formula for the magnetic field.
\begin{lemma}[Magnetic field]\label{lem:simplified magnetic field} Let $\Xi$ be the solution of~\eqref{eq:Xieq} for some $(a,d) \in \Omega_\ast^n \times \Z^n$, then the magnetic field $h_\ast = -\Delta \Xi$ is given by
\begin{align*}
    h_\ast (x) = -F(x) +\sum_{j=1}^n d_j K_0(|x-a_j|),
\end{align*}
where $F$ is the solution of~\eqref{eq:Feq}.\end{lemma}
\begin{proof} This is immediate from the identities $\Xi = \Xi_p + F- R + h_{\rm ex}$, $\Delta R = 0$, $\Delta F = F$, and~\eqref{eq:g Laplace}.
\end{proof}

\subsection{Discretization of the modified biharmonic equation}
\label{ssec:discr_biharmonic}

As previously shown, to compute the solution of the modified biharmonic equation~\eqref{eq:Xieq}, it suffices to solve the harmonic equation~\eqref{eq:Req} for $R$ and the modified Helmholtz equation~\eqref{eq:Feq} for $F$. To this end, we shall use the following functions:
\begin{align}
    P_j(r,\theta) \coloneqq r^{|j|} \mathrm{exp}(\ii j \theta)\quad \mbox{and}\quad Q_j(r,\theta) \coloneqq I_{|j|}(r) \mathrm{exp}(\ii j \theta), \quad j \in \Z, \label{eq:Q functions}
\end{align}
where $I_{|\alpha|}$ denotes the modified Bessel function of the first kind of
order $|\alpha|$, and $(r,\theta)$ denote the polar coordinates centered at a
suitable point $x_0 \in \Omega$. More precisely, we use the harmonic polynomials
$\{P_j\}_{j\in\Z}$ to approximate solutions of $R$ and the Bessel functions
$\{Q_j\}_{j\in\Z}$ to approximate solutions for $F$.

Note that, from the defining equation for the Bessel function and the polar coordinate representation of the Laplacian, i.e.,
\begin{align*}
    r^2 \partial_r^2 I_{\alpha} + r \partial_r I_\alpha -(r^2 + \alpha^2) I_\alpha = 0\quad \mbox{and}\quad \Delta = \partial_r^2 + \frac{1}{r} \partial_r + \frac{1}{r^2} \partial_\theta^2,
\end{align*}
it is easy to see that each $Q_j$ solves the modified Helmholtz equation $\Delta Q_j - Q_j = 0$ in the whole~$\R^2$. Similarly, $P_j$ is harmonic in $\R^2$. In particular, the restrictions $P_j \rvert_{\Omega}$ and $Q_j \rvert_{\Omega}$ to any open subset $\Omega$ are classical solutions of the respective equations in $\Omega$. We can therefore project the boundary conditions on the spaces spanned by those functions to compute approximations for $R$ and $F$.

For this, we shall work with the $\mL^2(\Omega)$-orthogonal projections. More
precisely, for a given discretization parameter $m \in \N$, we denote by
$\mathbb{P}_m$ and $\mathbb{Q}_m$ respectively the $\mL^2(\partial
\Omega)$-orthogonal projection on the spaces
\begin{align}
    H_m \coloneqq \mathrm{span} \{ P_j \rvert_{\partial \Omega} : |j| \leq m\} \quad \mbox{and}\quad B_m \coloneqq \mathrm{span} \{ Q_j \rvert_{\partial \Omega} : |j| \leq m\}
\end{align}
For sufficiently regular domains (e.g. with piecewise smooth boundaries), such projections can be evaluated to numerical accuracy via suitable quadrature rules along the boundary $\partial \Omega$. Our numerical approximations $R_m$ and $F_m$ then correspond to the natural extension of the projected functions, i.e.,
\begin{align}
    R_m(x) \coloneqq \sum_{|j| \leq m} a_j^{(m)} P_j(x) \quad\mbox{and}\quad F_m(x) \coloneqq \sum_{|j|\leq m} b_j^{(m)} Q_j(x), \quad \mbox{for $x\in \Omega$,}\label{eq:approximate solutions}
\end{align}
with the coefficients $a_j^{(m)},b_j^{(m)}$ determined through the $\mL^2(\partial \Omega)$-orthogonal projections $\mathbb{P}_m,\mathbb{Q}_m$ onto
span$\{P_j\}_{|j| \leq m}$ and span$\{Q_j\}_{|j| \leq m}$ respectively:
\begin{align}
    &\sum_{|j|\leq m} a_j^{(m)} P_j \rvert_{\partial \Omega} =
    \mathbb{P}_m\left(-\sum_{j=1}^n d_j \log |\cdot-a_j|\right)\\
    &\sum_{|j| \leq m} b_j^{(m)} Q_j\rvert_{\partial \Omega} =\mathbb{Q}_m
    \left(-h_{\rm ex} + \sum_{j=1}^n d_j K_0(|\cdot-a_j|)\right).  \label{eq:coefficients}
\end{align}

It should be noted that the coefficients $\{a_j^{(m)}, b_j^{(m)}\}_{|j|\leq m}$
may depend on $m$ as the boundary restrictions of $\{P_j\}_{|j|\leq m}$ or $\{Q_j\}_{|j|\leq m}$ are, in
general, not $\mL^2(\partial \Omega)$-orthogonal. Moreover, a priori, it is not
clear that the functions $F_m$ and $R_m$ are able to approximate $F$ and $R$
with arbitrary accuracy or that the coefficients can be uniquely determined.
However, the following result shows that this is indeed the case.
\begin{theorem}[Basis property]\label{thm:basis} Let $\Omega \subset \R^2$ be a
  bounded simply connected domain with Lipschitz boundary. Then the families
  $\{P_j\rvert_{\partial \Omega}\}_{j \in \Z}$ (resp. $\{Q_j
  \rvert_{\partial \Omega}\}_{j \in \Z}$) centered at any point $x_0 \in \Omega$ are linearly independent and their spans are dense in $\mH^{\frac12}(\partial \Omega)$.
\end{theorem}

The key ingredient in the proof of Theorem~\ref{thm:basis} is the following Runge approximation property, whose proof can be found in \cite[Section 3]{KV22}.

\begin{lemma}[Theorem 3.9 in \cite{KV22}] \label{thm:runge} Let $L = \Delta$ or $L= \Delta -1$, then for any Lipschitz open sets $\Omega \subset \subset \Omega'$ such that $\Omega'\setminus \overline{\Omega}$ is connected, any weak solution $u \in \mH^1(\Omega)$ of $L u = 0$ in $\Omega$, and any $\epsilon>0$, there exists a weak solution $u' \in \mH^1(\Omega')$ of $Lu' = 0$ in $\Omega'$ such that
\begin{align*}
    \norm{u - u'}_{\mH^1(\Omega)} < \epsilon.
\end{align*}
\end{lemma}

\begin{proof}[Proof of Theorem~\ref{thm:basis}] As the proof is essentially the
  same for both families of functions, let us focus only on $\{Q_j\}_{j\in\Z}$. We first
  show that the functions $\{Q_j\rvert_{\partial \Omega}\}_{j \in \Z}$ are
  linearly independent. For this, suppose that $Q_j = \sum_{k \neq j} \alpha_k
  Q_k$ in $\partial \Omega$ for finitely many non-zero $\alpha_k$. Then $Q
  \coloneqq Q_j - \sum_{k\neq j} \alpha_k Q_k$ solves $\Delta Q - Q = 0$ and $Q
  \in \mH^1_0(\Omega)$. As this operator is strictly negative, we must have $Q
  =0$ a.e. in $\Omega$. However, this is not possible for the following reason.
  As the $I_\alpha(r)$ are analytic functions for $r>0$, for any finite index
  set $\mathcal{I} \subset \N$, we can find a $\delta >0$ small enough such that
  $I_k(\delta) \neq 0$ for any $k \in \mathcal{I}$. Consequently, the functions
  $Q_k\rvert_{\partial B_\delta(x_0)} = I_{|k|}(\delta) \mathrm{exp}(\ii k
  \theta)$ are linearly independent along $\partial B_\delta(x_0)$, which is
  contained in $\Omega$ for $\delta$ small enough. Note that this argument shows
  that $I_{|k|}(r) \neq 0$ for any $r>0$.

We now prove the density result, namely, that any $g \in \mH^\frac12(\partial \Omega)$ can be well approximated by linear combinations of $\{Q_j\rvert_{\partial \Omega}\}_{j\in\Z}$. Let $u$ be a solution to $Lu=0$ in $\Omega$, with $u=g\in H^{1/2}(\partial\Omega)$ on $\partial \Omega$.
Let $R>0$ be so large that $\Omega \subset\subset  B_R(x_0)$ and $B_R(x_0) \setminus \overline{\Omega}$ is connected (which is possible because $\Omega$ is simply connected and bounded). Then, note that by Lemma~\ref{thm:runge}, for any $\epsilon>0$ we can find $u'\in \mH^1(B_R(x_0))$ such that $\Delta u'- u' =0$ in $B_R(x_0)$ and
\begin{align}
    \norm{u'-u}_{\mH^1(\Omega)} <\epsilon \label{eq:epsilon close}
\end{align}
We can now approximate $u'$ in $\mH^1(B_R(x_0))$. More precisely, since $Q_j \rvert_{B_R(x_0)} = I_{|j|}(R) \mathrm{exp}(\ii j \theta)$ are nothing but Fourier modes up to a \emph{nonzero} factor (by the previous argument), from standard Fourier theory we can find finitely many $\alpha_k \in \C$ such that $v \coloneqq \sum_{k} \alpha_k Q_k$ satisfies
\begin{align*}
    \norm{v-u'}_{\mH^{\frac12}(\partial B_R(x_0))}< \epsilon.
\end{align*}
As $v-u'$ is a $\mH^1(\Omega)$ weak solution of $(\Delta -1) (v-u') = 0$ in $B_R(x_0)$, it follows from standard elliptic regularity that
\begin{align}
    \norm{v-u'}_{\mH^1(B_R(x_0))} \lesssim_R \norm{v-u'}_{\mH^{\frac12}(\partial B_R(x_0))} <  \epsilon.  \label{eq:epsilon close 2}
\end{align}
Hence, from the standard Sobolev trace theorem and estimates~\eqref{eq:epsilon close} and~\eqref{eq:epsilon close 2}, we have
\begin{align}
    \norm{v-u}_{\mH^{\frac12}(\partial \Omega)} \lesssim_\Omega \norm{v-u}_{\mH^1(\Omega)} \lesssim \norm{u-u'}_{\mH^1(\Omega)} + \norm{u'-v}_{\mH^1(B_R(x_0))} \lesssim \epsilon,
\end{align}
which completes the proof.
\end{proof}

As a corollary of Theorem~\ref{thm:basis}, we see that the numerical approximations converge to the exact solutions in the limit $m\rightarrow \infty$. This provides a theoretical justification for our approach.
\begin{corollary}[Convergence in the large $m$ limit] \label{cor:discrete approximation}
For any bounded open set $\Omega$ with Lipschitz boundary, the numerical approximations $F_m$ and $R_m$ defined in~\eqref{eq:approximate solutions} satisfy
\begin{align}
    \lim_{m\ra \infty} \norm{R-R_m}_{\mH^{\frac12}(\Omega)} = \lim_{m \rightarrow \infty} \norm{F-F_m}_{\mH^{\frac12}(\Omega)} = 0, \label{eq:H12 convergence}
\end{align}
where $F$ and $R$ are the exact solutions of~\eqref{eq:Feq} and~\eqref{eq:Req}.
\end{corollary}

\begin{proof}
As $\mH^{\frac12}(\partial \Omega)$ is densely embedded in $\mL^2(\partial
\Omega)$, the span of the sets $\{Q_j\rvert_{\partial \Omega}\}_{j\in\Z}$ or
$\{P_j\rvert_{\partial \Omega}\}_{j\in\Z}$ are dense in $\mL^2(\partial \Omega)$.
Consequently, we have
\begin{align*}
     \lim_{m \ra \infty} \norm{R- R_m}_{\mL^2(\partial \Omega)} = \lim_{m \ra \infty} \norm{F - F_m}_{\mL^2(\partial \Omega)} = 0.
\end{align*}
The convergence for $R$ in~\eqref{eq:H12 convergence} now follows from the
following deep estimate (see \cite{Dah77,Dah79}  or \cite[Theorems 5.2--5.4 and
Corollary 5.5.]{JK95})
\begin{align*}
    \norm{u}_{\mH^{\frac12}(\Omega)} \lesssim \norm{g}_{\mL^2(\partial \Omega)},
\end{align*}
for the unique harmonic function $u \in \mH^{\frac12}(\Omega)$ satisfying $u = g$ on $\partial \Omega$ (in the non-tangential sense). For the convergence of $F_m$ we need a similar estimate for solutions of $\Delta u - u= 0$. This estimate can be established via a perturbative argument. Precisely, let $u' \in \mH^{\frac12}(\Omega)$ be the unique harmonic function satisfying $u' = g\in \mL^2(\partial \Omega)$ in the non-tangential sense, then from the classical elliptic theory we can find a unique $w\in \mH^1_0(\Omega)$ such that $\Delta w - w = u'$ and $\norm{w}_{\mH^1(\Omega)} \lesssim \norm{u'}_{\mL^2(\Omega)}$. It then follows that $u \coloneqq u' + w$ solves
\begin{align*}
   \Delta u - u =0,\quad u=g \quad \mbox{on $\partial \Omega$ and satisfies} \quad \norm{u}_{\mH^{\frac12}(\Omega)} \lesssim \norm{g}_{\mL^2(\partial \Omega)}.
\end{align*}
\end{proof}

\begin{remark}[Pointwise convergence] Standard elliptic regularity theory (see,
  e.g. \cite{GT01}) shows that estimates~\eqref{eq:H12 convergence} implies interior
  pointwise $C^k$ convergence, i.e.,
\begin{align*}
     \lim_{m \ra \infty} \norm{R -R_m}_{C^k(V)} = \lim_{m \ra \infty} \norm{F-F_m}_{C^k(V)} = 0.
\end{align*}
for any $V$ compactly contained in $\Omega$ and $k\in \N$. In particular, this shows that the gradient of the renormalized energy can be exactly evaluated in the large $m$-limit.
\end{remark}

\begin{remark}[Explicit estimate vs general domains] In the case of a disk
  $\Omega = B_r(x_0)$, the boundary restrictions of $P_j$ and $Q_j$ are nothing
  but the Fourier modes $\mathrm{exp}(\ii j \theta)$. In particular, in this
  case one can obtain explicit convergence and error estimates, as done in
  \cite{CKMS25}. Nevertheless, we emphasize that the previous theorem justifies
  the method for a very general class of domains $\Omega$. \label{rmk:fourier}
\end{remark}

\subsection{Numerical simulation of the limiting ODE}\label{ssec:HD_sim}

Now that we have introduced a novel discretization method to solve the modified
biharmonic equation required to compute the gradient of the renormalized energy,
we end this section by presenting a few trajectories for various initial
conditions and external magnetic fields $h_{\rm ex}$.

The computational domain is the unit disk $\Omega=B_1(0)$,
with different initial conditions $a^0$ and $d$ for each
case. For the unit disk, the projection operators $\mathbb{P}_m$ and
$\mathbb{Q}_m$ introduced in Section~\ref{ssec:discr_biharmonic} take simpler
forms since, in this case, they both coincide with truncated Fourier series, as
mentioned in Remark~\ref{rmk:fourier}. Solving the modified Helmholtz equation
\eqref{eq:Feq} can thus be done as follows:
\begin{enumerate}
  \item Choose a maximum degree $m$, and fix it once and for all. Then,
    for any $g\in L^2(\partial\Omega)$,
    \[
      (\mathbb{Q}_m g)(e^{\ii \theta}) = \sum_{k=-m}^m \widehat{g}(k) e^{\ii k\theta},
      \quad\text{where}\quad
      \widehat{g}(k) = \frac1{2\pi} \int_0^{2\pi} e^{-\ii
        k\theta}g(e^{\ii \theta}){\rm d}\theta.
    \]
  \item Compute the Fourier coefficients $(\widehat{g}_{a}(k))_{-m\leq k \leq m}$
    of the Dirichlet boundary condition in \eqref{eq:Feq}
    \[[0,2\pi) \ni \theta \mapsto g_a(e^{\ii \theta}) \coloneqq - h_{\rm ex} +
      \sum_{j=1}^n d_jK_0(|(\cos \theta,\sin \theta)-a_j|)\] up to order $m$, for instance
    using a Fast Fourier Transform (FFT).
  \item Compute the (approximate) solution $F_m$ to \eqref{eq:Feq} by expanding
    $\mathbb{Q}_mg_a$ inside the disc:
    \begin{equation*}
      \forall\ r \in [0,1),\ \forall\ \theta\in[0,2\pi), \quad
      F_m(re^{\ii \theta}) = \sum_{j=-m}^m I_{|j|}(r)\widehat{g}_{a}(k)e^{\ii k\theta},
    \end{equation*}
    which is a solution of \eqref{eq:Feq_a} satisfying~\eqref{eq:Feq_b} approximately.
\end{enumerate}

We are now capable of simulating some trajectories. In the limiting ODE
\eqref{eq:ODE1}, we fix $\alpha_0 = 0$, $\beta_0 = 1$ so that we are in the case
of the Schrödinger flow. This ODE is solved with a 4th-order Runge--Kutta method
(RK4, see {\emph{e.g.}} \cite[Chapter II]{HWN93}) with time step $\delta
t=10^{-3}$ up to some time $T$. The gradient of the renormalized energy is
computed with the simplified formulation~\eqref{eq:simplified gradient 1}, where
$F$ solves \eqref{eq:Feq_a}-\eqref{eq:Feq_b} and is approximated with $m=64$ Fourier modes on the
boundary $\partial\Omega$. We consider 4 different initial conditions, each of
them being used twice with $h_{\rm ex}=3$ and $h_{\rm ex} = -2$. Note that
all the trajectories were obtained in a few seconds on a personal laptop: a
significant improvement when compared to the computational cost of the complete
PDE \eqref{eq:GPE_mag} for small $\epsilon$.

Such trajectories can already be analyzed from a physics perspective: in the case of vortices with positive degree, adding a
positive magnetic field $h_{\rm ex} > 0$ brings an additional counterclockwise
rotation speed to the vortices (as can be seen from Case 1 without magnetic
field and Case 1 with $h_{\rm ex} = 3$ where the vortices span a shorter
distance within the same time), while adding a negative magnetic field $h_{\rm
  ex} < 0$ brings an additional clockwise rotation speed (as can be seen from
Case 2 without magnetic field and Case 2 with $h_{\rm ex} = -2$).

\begin{table}[h!]
  \caption{Initial configurations for simulation of the Hamiltonian dynamics.}
  \centering
  \footnotesize
  \begin{tabular}{cccc}
    \toprule
    \textbf{Case}       & $n$ & $a^0$                                & $d$                                          \\
    \hline
    \textbf{Case $1$} & $2$          & $\left((-0.5,0.0), (0.5,0.0)\right)$                          & $(1,1)$                                                \\
    \textbf{Case $2$} & $3$          & $\left((-0.5,0.0), (0.5,0.0), (0.0,0.0)\right)$               & $(1,1,-1)$                                             \\
    \textbf{Case $3$} & $2$          & $\left((-0.75,0.0), (0.75,0.0)\right)$                        & $(-1,1)$                                               \\
    \textbf{Case $4$} & $4$          & $\left((-0.6,-0.6), (-0.6,0.6), (0.6,0.6), (0.6,-0.6)\right)$ & $(1,-1, 1, -1)$                                        \\
    \bottomrule
  \end{tabular}
\end{table}

\begin{figure}[p!]
  \includegraphics[width=0.32\linewidth]{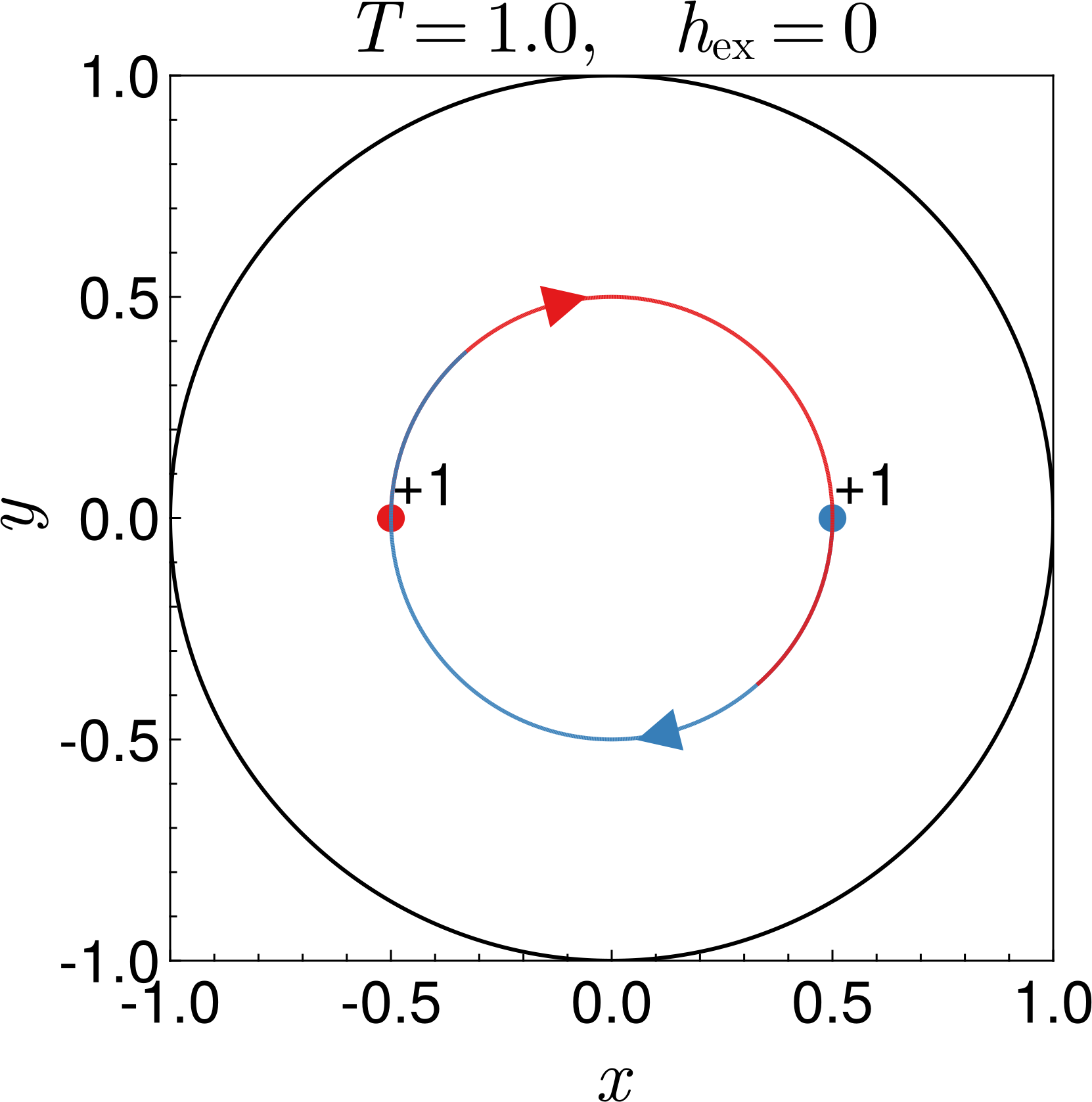}\hfill
  \includegraphics[width=0.32\linewidth]{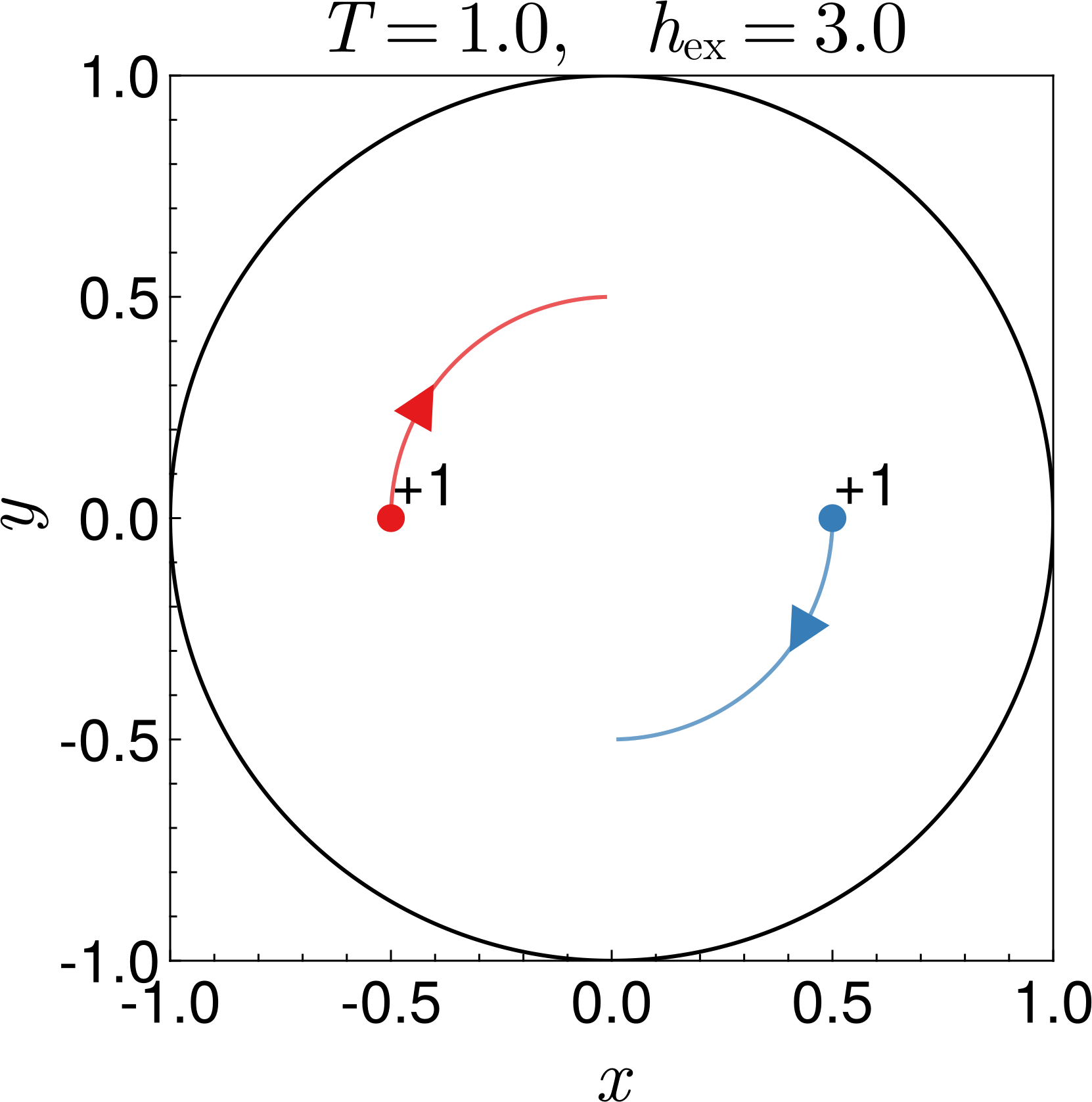}\hfill
  \includegraphics[width=0.32\linewidth]{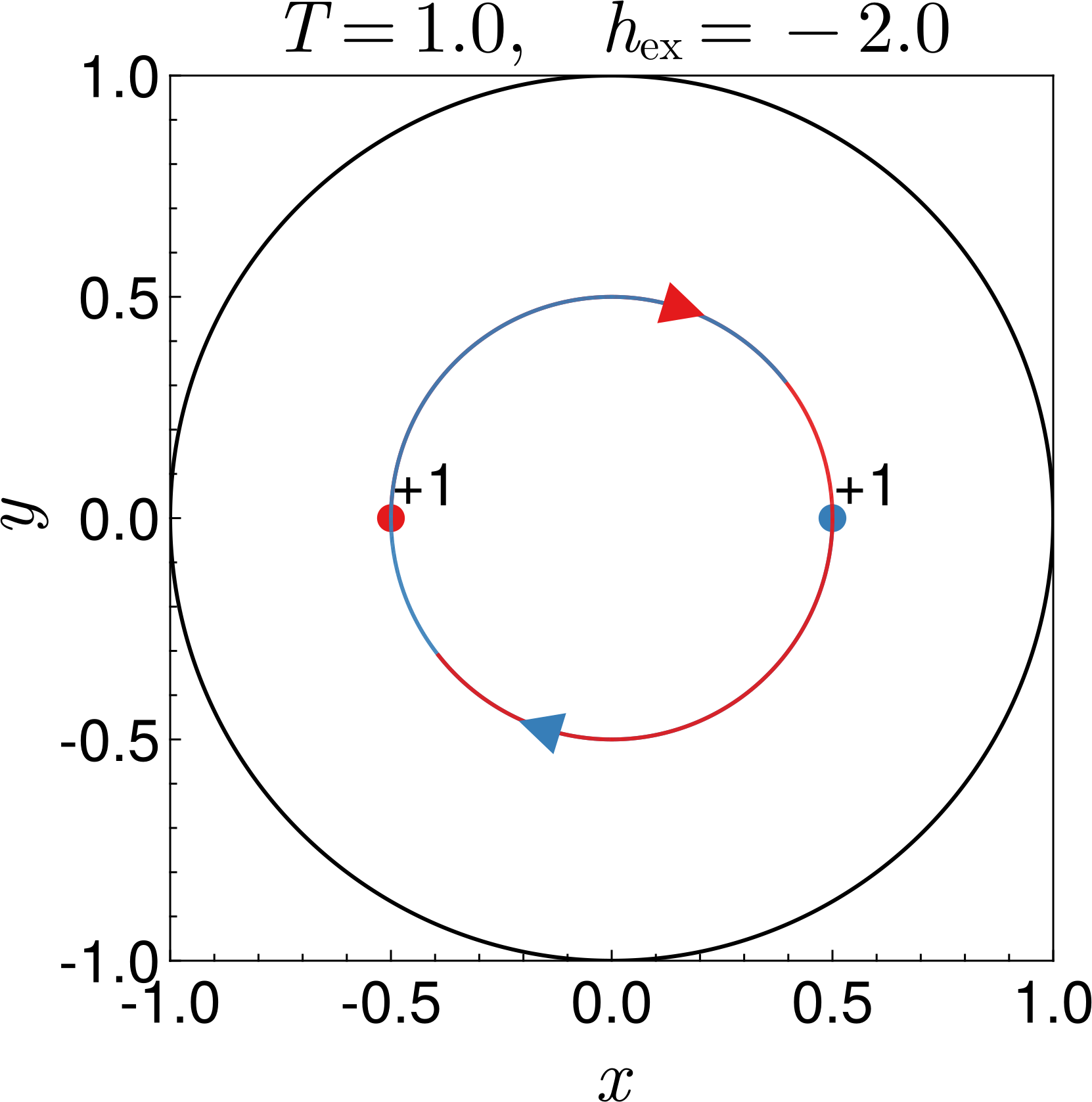}\\
  \includegraphics[width=0.32\linewidth]{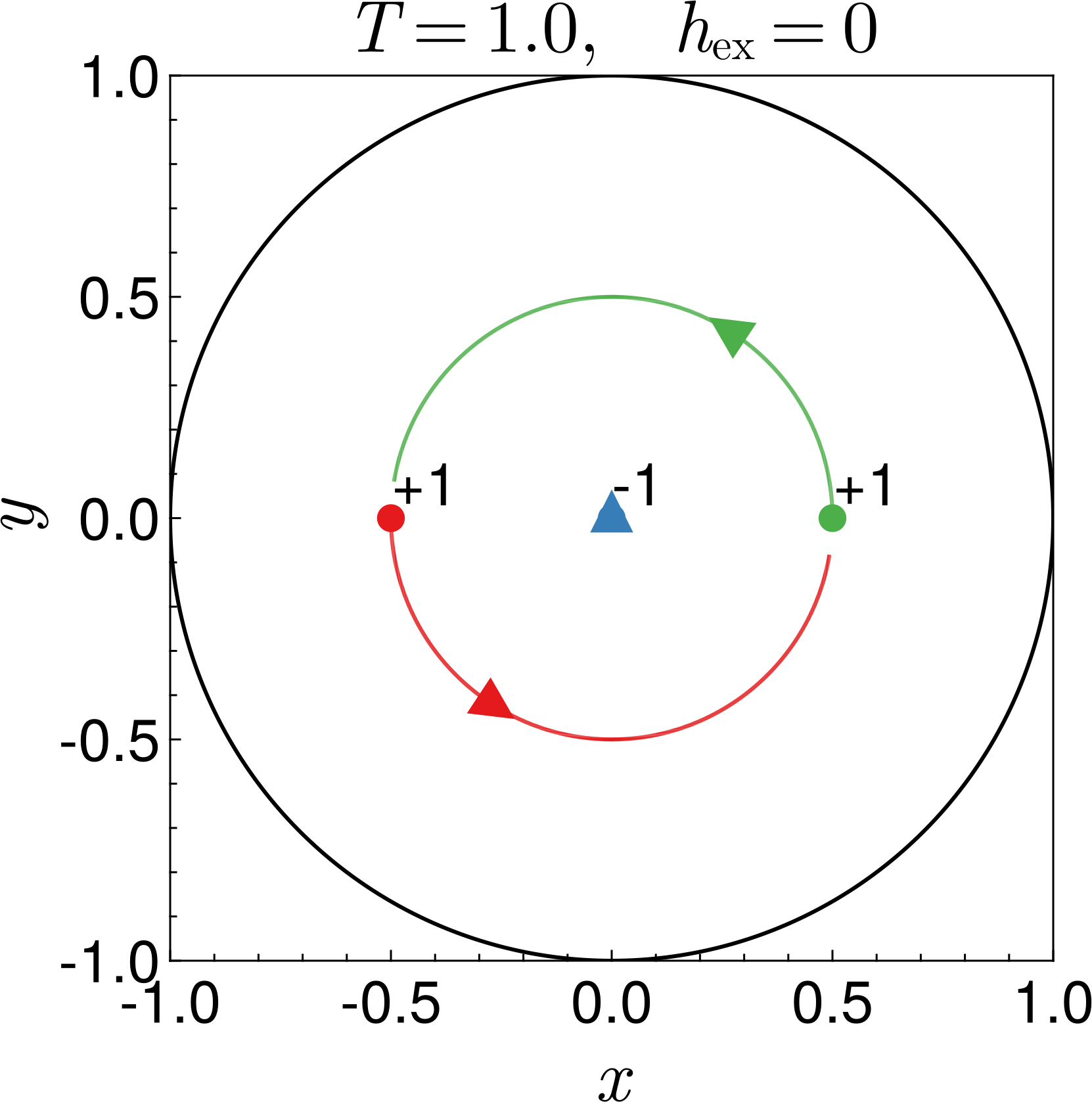}\hfill
  \includegraphics[width=0.32\linewidth]{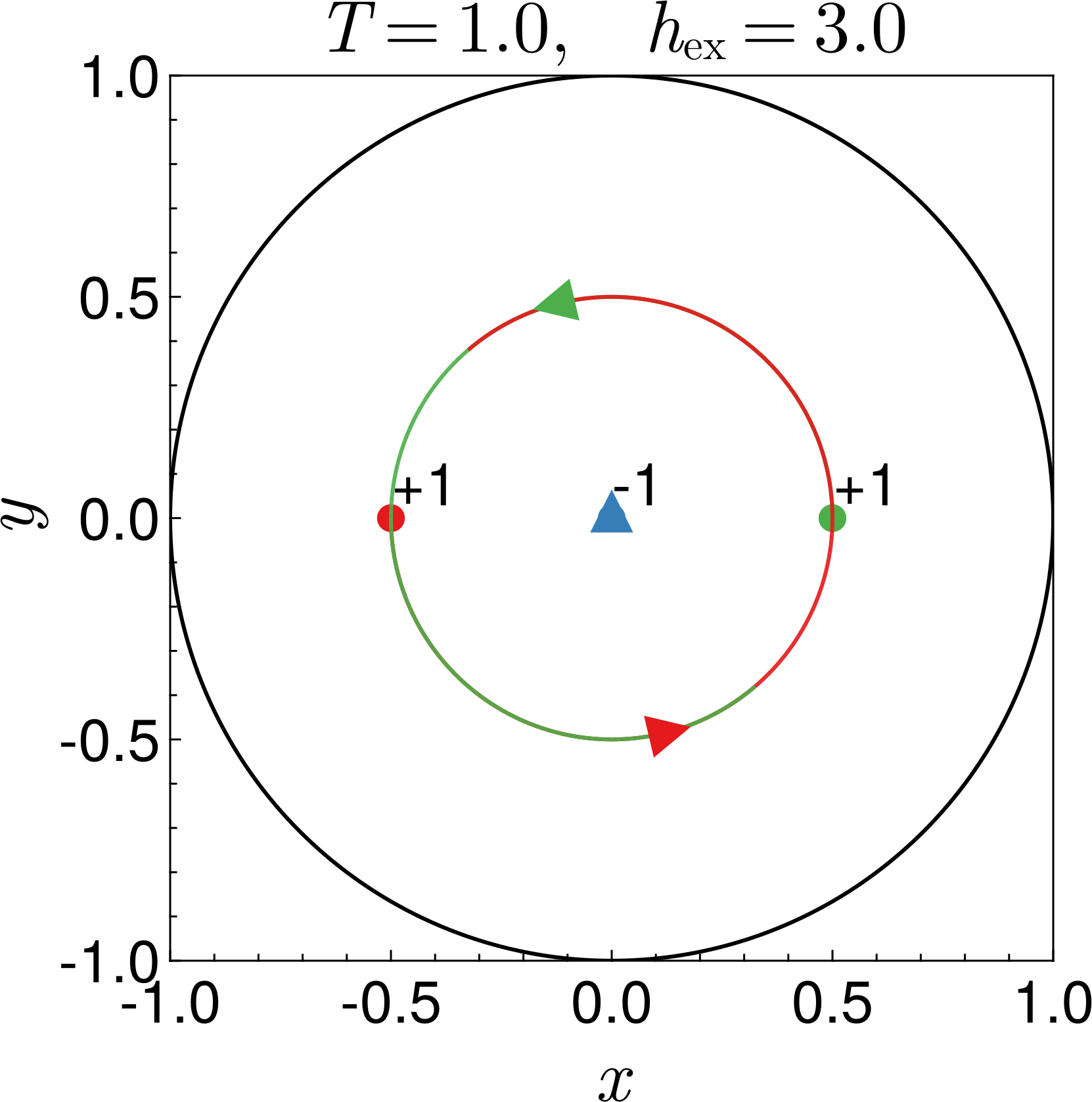}\hfill
  \includegraphics[width=0.32\linewidth]{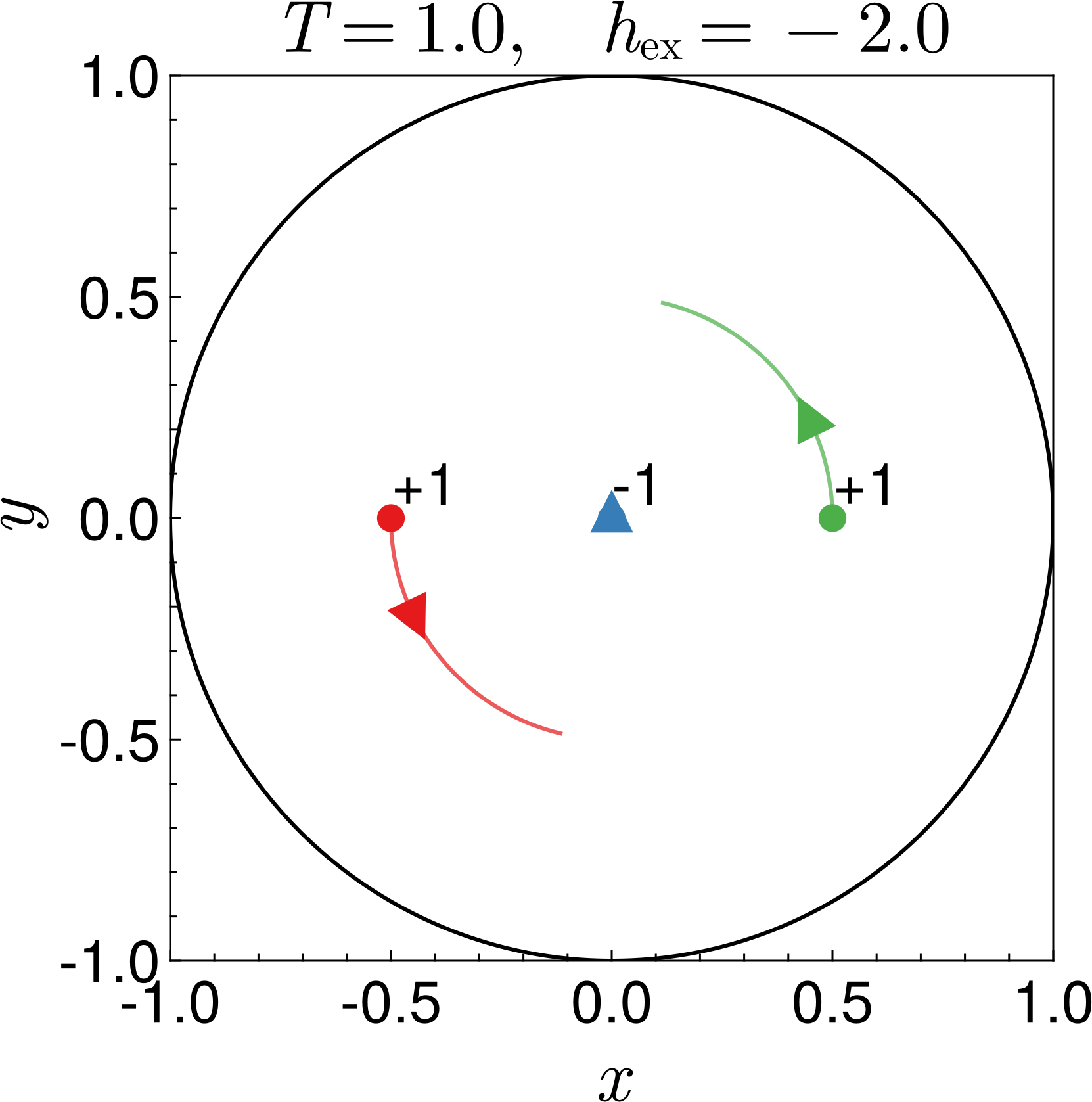}\\
  \includegraphics[width=0.32\linewidth]{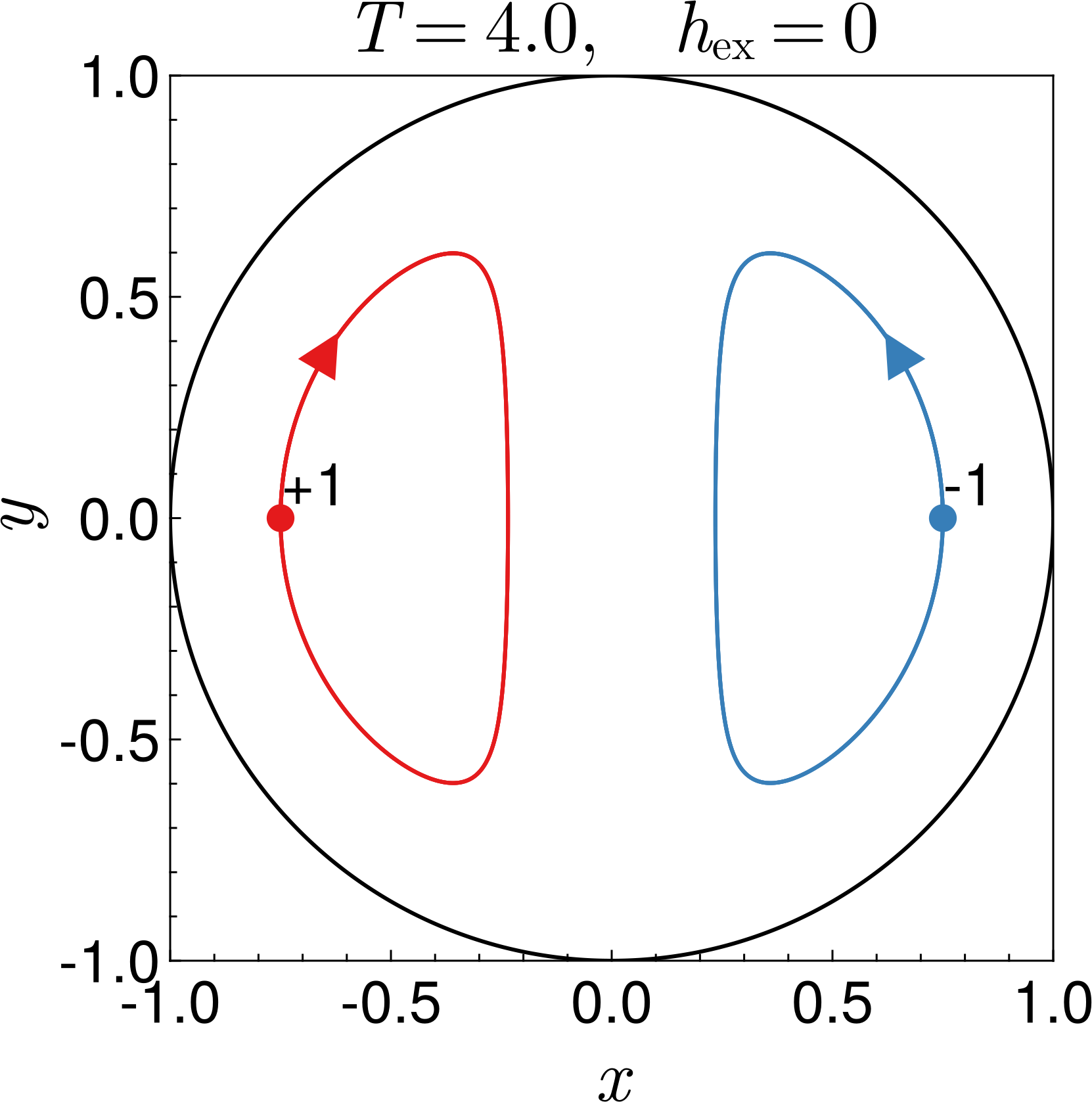}\hfill
  \includegraphics[width=0.32\linewidth]{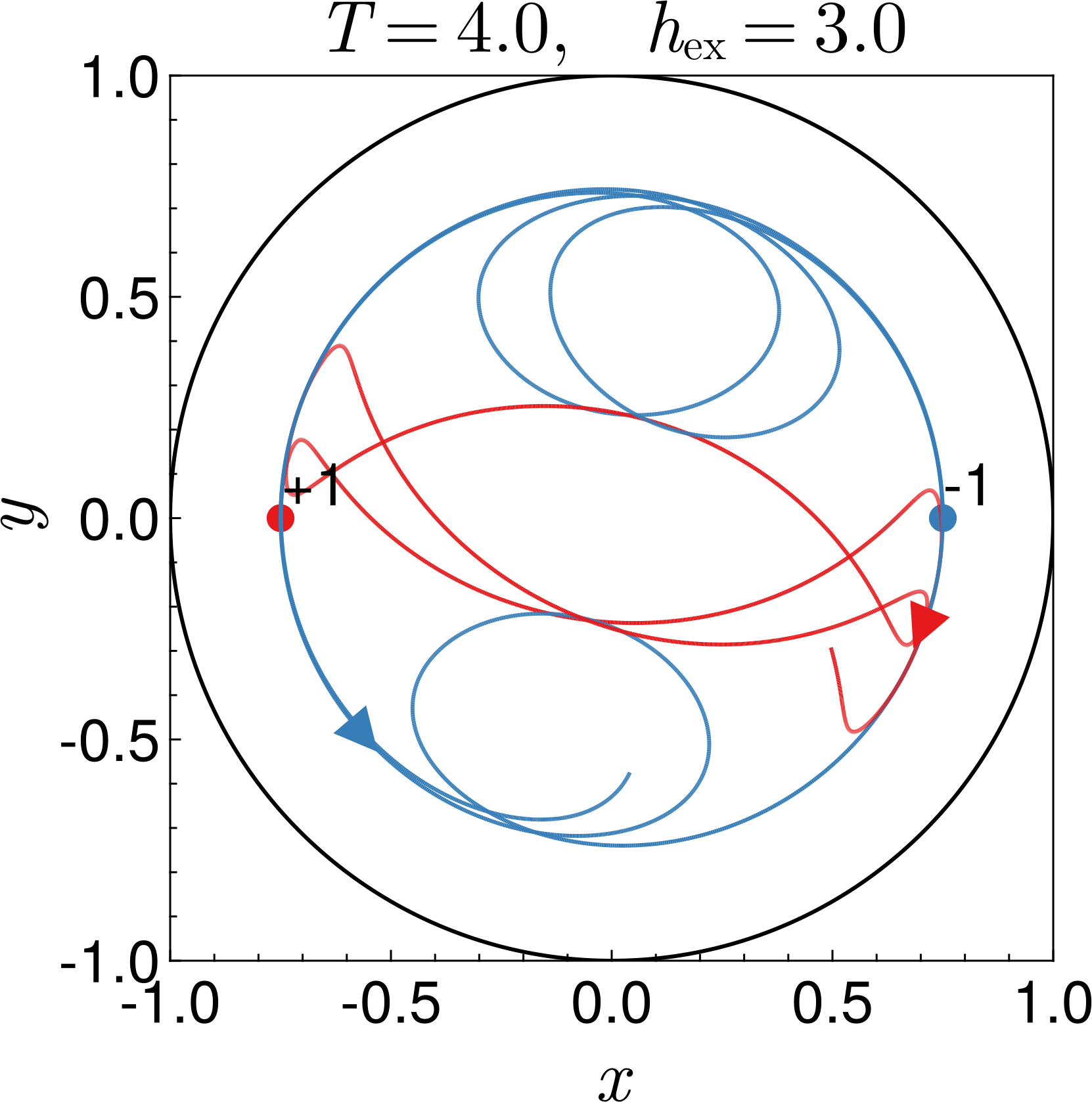}\hfill
  \includegraphics[width=0.32\linewidth]{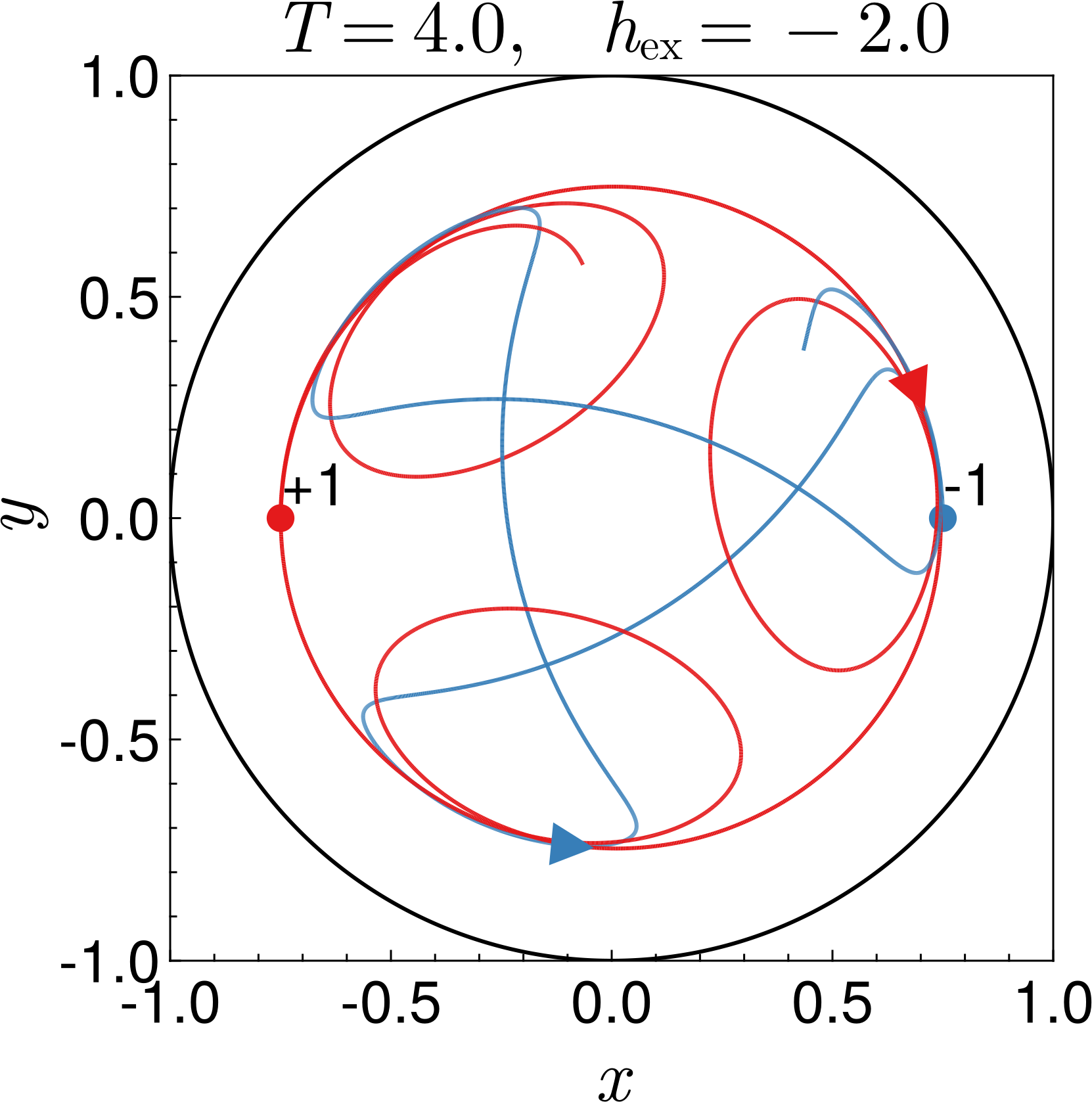}\hfill
  \includegraphics[width=0.32\linewidth]{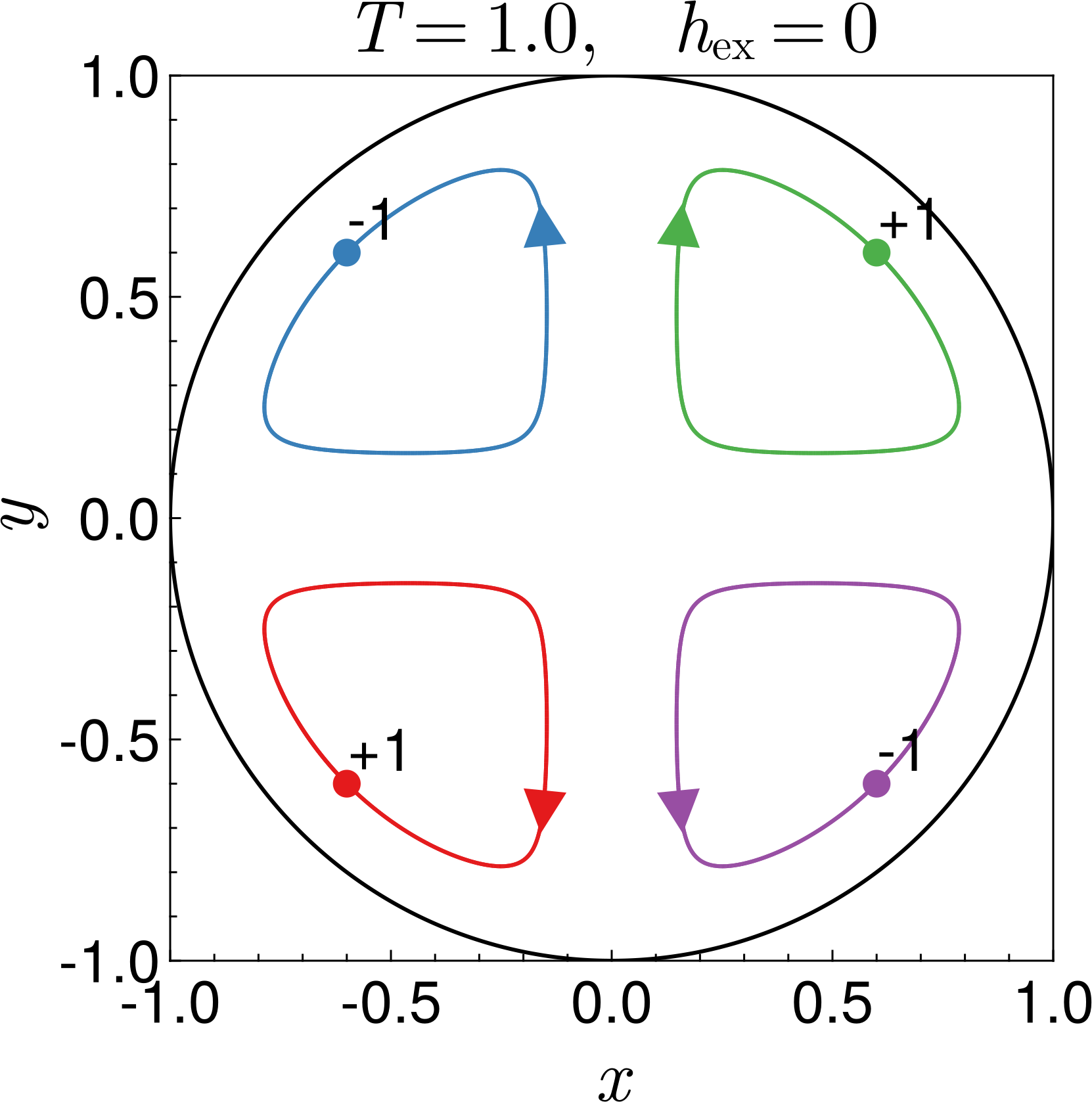}\hfill
  \includegraphics[width=0.32\linewidth]{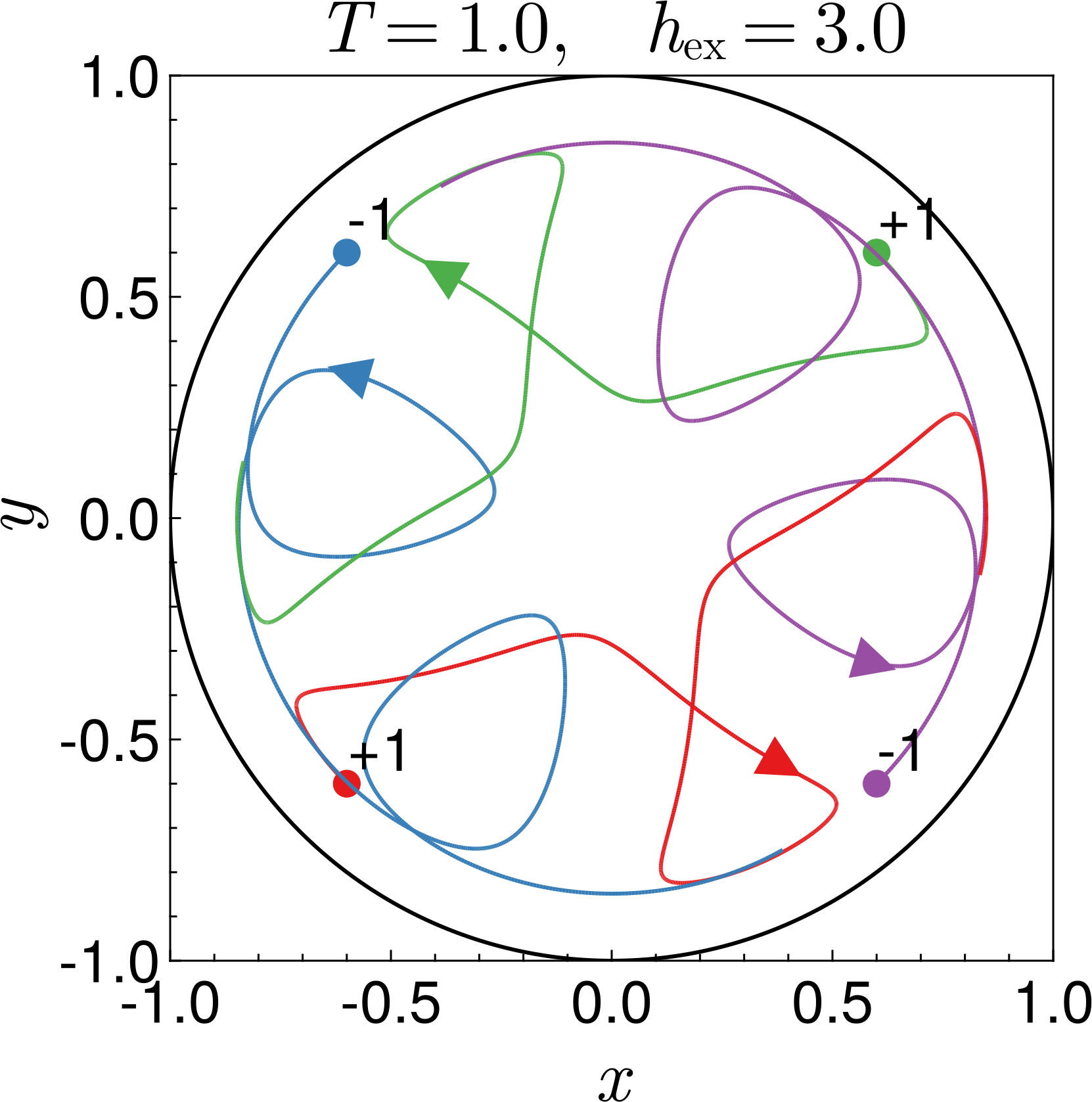}\hfill
  \includegraphics[width=0.32\linewidth]{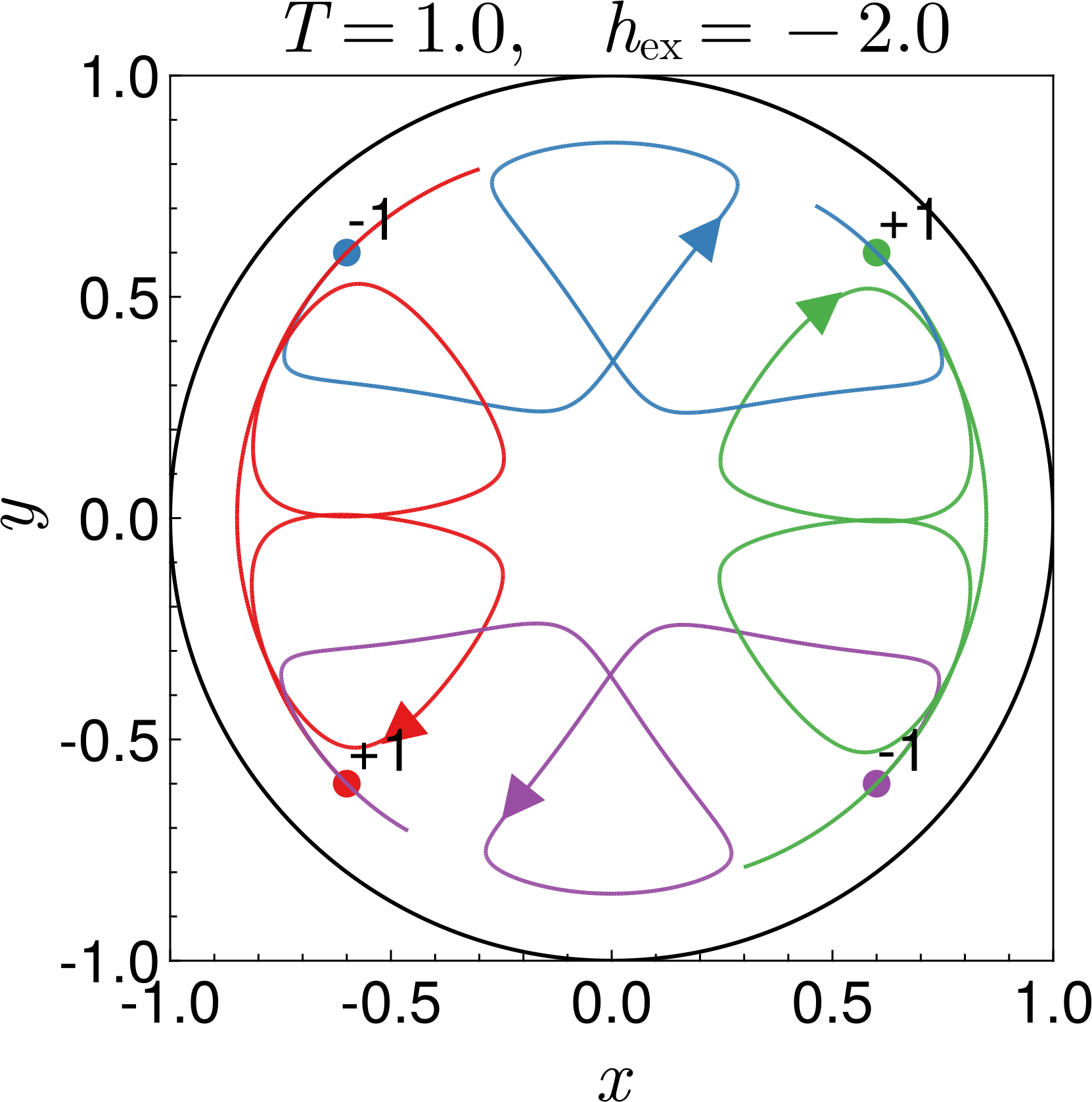}\\
  \caption{Trajectories of vortices in the singular limit $\epsilon\searrow0$.
    Initial conditions are represented by
    dots and the different degrees used are specified for each trajectory.
    From top to bottom: cases 1 to 4.
    From left to right: $h_{\rm ex} = 0$,
    $h_{\rm ex} = 3$, $h_{\rm ex} = -2$.}
  \label{fig:trajectories}
\end{figure}

In Figure~\ref{fig:num_cvg}, we illustrate the numerical convergence of the
trajectories, both in time and in $m$, on the numerical scheme we used to
generate the limiting trajectories for $h_{\rm ex} = 3$ and various initial
conditions: we use $d = \{-1,+1\}$ and $a^0 = \{(-\ell,0), (\ell,0)\}$ with
different values of $\ell$ to mimic vortices that are either far apart, close to
each other or close to the boundary $\partial\Omega$. As expected, the observed
convergence rates are compatible with a 4th order in time and spectral accuracy
in $m$ independently of the inter-vortex distance. However, when the vortices
are too close to each other, or too close to the boundary, the prefactors in the convergence
rate increase, as proven in \cite{CKMS25} without magnetic field.

\begin{figure}[h!]
  \centering
  \includegraphics[width=0.5\linewidth]{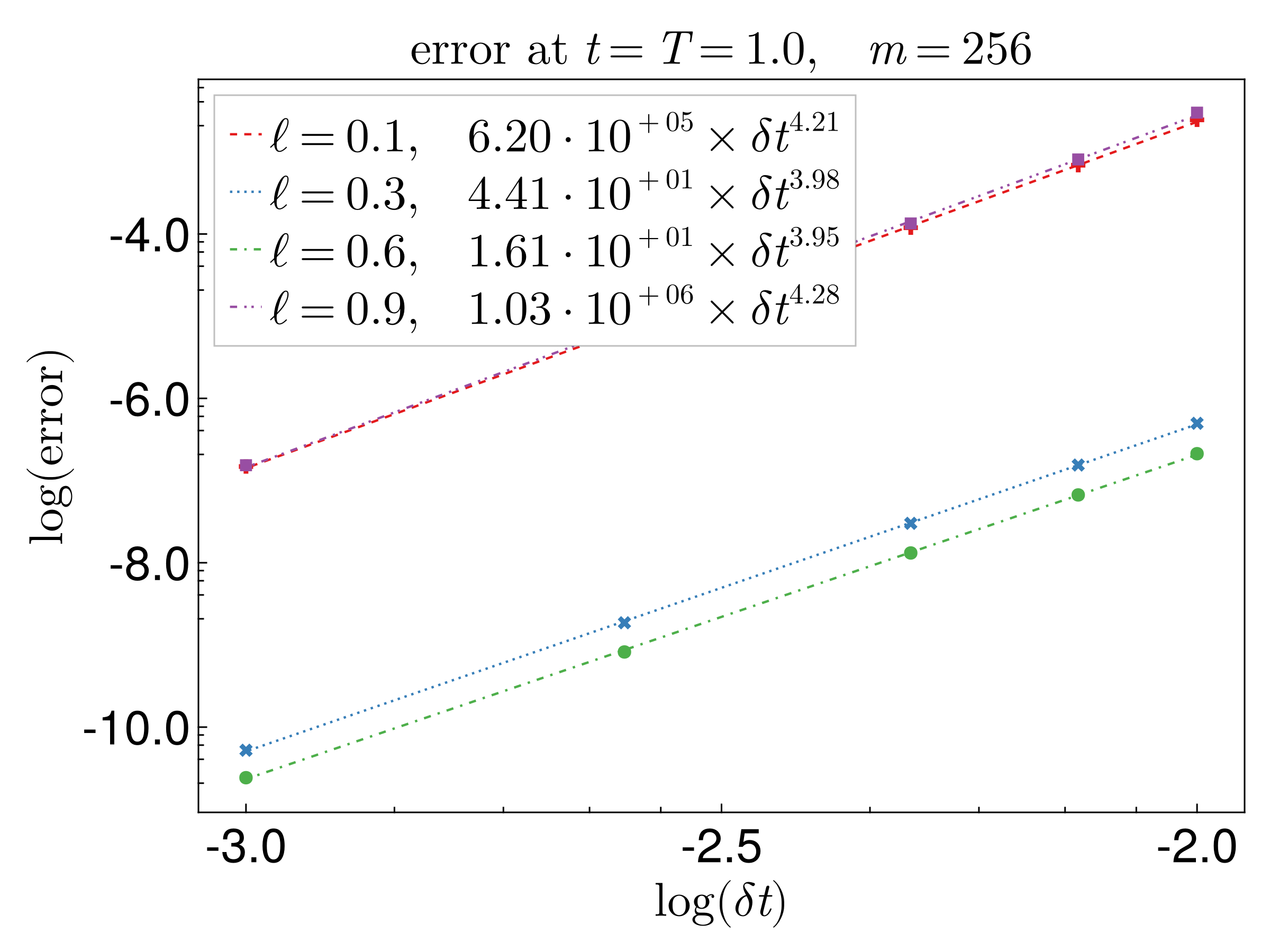}\hfill
  \includegraphics[width=0.5\linewidth]{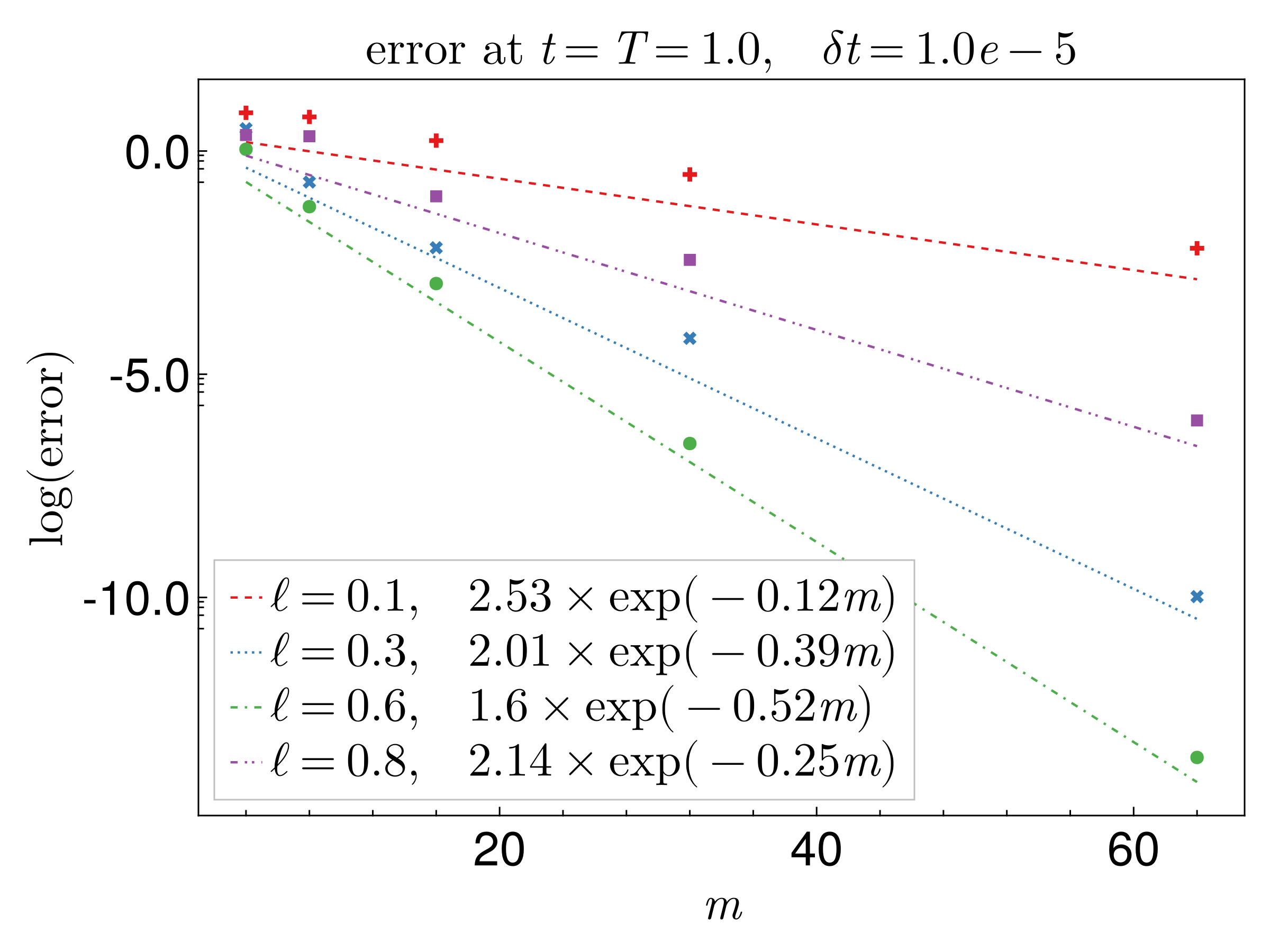}\\
  \caption{Numerical illustration of the convergence of the numerical method to
    obtain the limiting trajectories, for $h_{\rm ex} = 3$. (Left) The slope
    shows the expected 4th order in time accuracy. (Right) The plots are
    compatible with spectral accuracy in the number of Fourier modes on the
    boundary.}
  \label{fig:num_cvg}
\end{figure}

In the next section, we present how to recover an approximate solution to the
PDE from the trajectories obtained in the singular limit.

\section{Numerical scheme and results}
\label{sec:numerical scheme}

The goal of this section is to introduce a numerical strategy to recover, at the
PDE level, an approximation of the solution to the magnetic Ginzburg-Landau
equation \eqref{eq:GPE_mag} from the, finite dimensional, vortex trajectories.
First, we start with some complementary remarks on the numerical construction of
well-prepared states.

\subsection{Complements on well-prepared states}\label{ssec:WP}

We focus here on the (numerical) construction of well-prepared states
$(u_\epsilon^0, A_\epsilon^0)_{\epsilon>0}$, in the sense of
Definition~\ref{def:wellprepared}. Looking for a radially symmetric solution to
the (magnetic free) stationary Ginzburg-Landau equation in a ball, one finds
that its radial profile is given by the function $f_{\epsilon,r_0}$ that
satisfies the ODE
\begin{equation}\label{eq:radialODE}
  \frac1r\left(rf_{\epsilon,r_0}'(r)\right)' - \frac{1}{r^2}f_{\epsilon,r_0}(r) +
  \frac1{\epsilon^2}\left(1-|f_{\epsilon,r_0}(r)|^2\right)f_{\epsilon,r_0}(r) = 0,
\end{equation}
together with the boundary conditions $f_{\epsilon,r_0}(0) = 0$ and
$f_{\epsilon,r_0}(r_0) = 1$. For finite $r_0$, this equation can be solved
numerically with very high precision (\emph{e.g.} using a nonlinear finite
element solver), even for $\epsilon$ as small as $10^{-3}$, see
Figure~\ref{fig:feps}. Finally, note that, by a scaling argument,
$f_{\epsilon,r_0}$ only depends on the ratio $r_0/\epsilon$, so that we may
write only $f_\epsilon$, see \cite[Section 5]{KS11} for more details.

\begin{figure}[h!]
  \centering
  \includegraphics[width=0.5\linewidth]{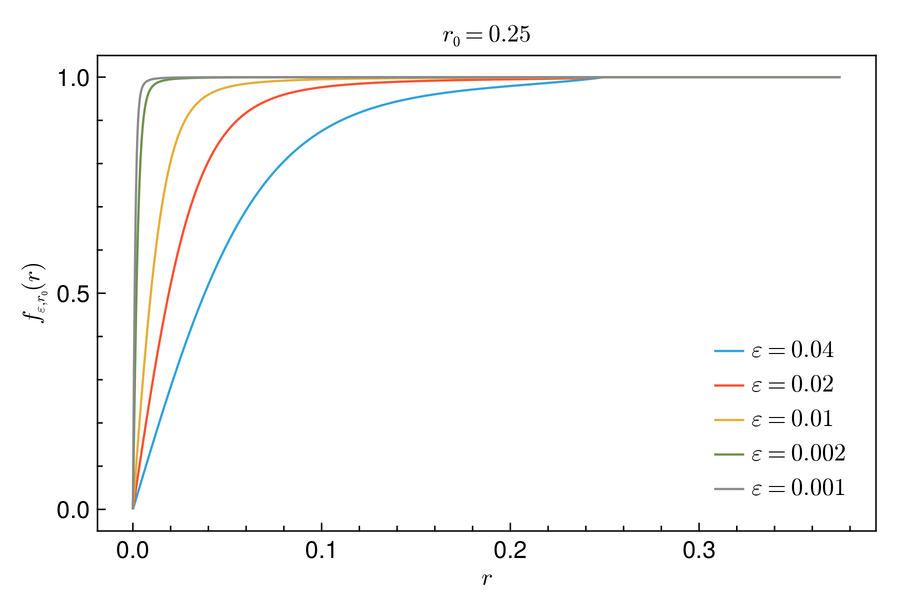}
  \caption{Numerical approximation of $f_{\epsilon,r_0}$.}
  \label{fig:feps}
\end{figure}

Starting from the localized approximation $(r, \theta) \mapsto
f_\epsilon(r)e^{\ii \theta} $ of a single vortex, one can then build the
following initial condition for \eqref{eq:GPE_mag}: given vortex positions $a^0
= (a_j^0)_{j=1,\dots,n}\in\Omega_*^{n}$ with degrees $d =
(d_j)_{j=1,\dots,n}\in\{\pm1\}^n$, we consider
\begin{equation}\label{eq:init}
  u_{\epsilon}^*(x; a^0, d) =
  u^*(x; a^0, d) \prod_{j=1}^n f_{\epsilon}(|x -
  a_j^0|) =
  \exp(\ii \varphi(x)) \prod_{j=1}^n
  \left(\frac{x - a_j}{|x - a_j|}\right)^{d_j} f_{\epsilon}(|x -
  a_j^0|)
  , \quad x\in\overline\Omega
\end{equation}
where $u^*$ is the canonical harmonic map with vortex locations $a^0$ (see
Remark~\ref{rmk:canon}), and the phase $\varphi$ is computed to match Neumann
boundary conditions.
This numerical strategy can be seen as an operator from the manifold
$\{u^*(a,d),\ a\in\Omega_*^{n},\ d\in\{\pm1\}^n\}$ to $H^1(\Omega,\C)$
which smoothes out the canonical harmonic map $u^*(a,d)$ at the
singularities locations. It is identical to the smoothing strategy that we used
in the magnetic-free case \cite{CKMS25}. In the magnetic case, it is sufficient
to consider in addition $A_\epsilon = A_* = \curl\Xi$ where $\Xi$ solves
\eqref{eq:xieq0} to obtain well-prepared states
$(u_\epsilon^*,A_*)_{\epsilon>0}$, see \cite[Lemma 5.8]{KS11}. Using this
smoothing strategy, we are able to reconstruct an approximate solution to the
infinite dimensional PDE as explained in the next section.

\subsection{Description of the method}\label{sec:method_1}

Instead of directly solving \eqref{eq:GPE_mag}, we propose to take advantage of the
trajectories of the vortices in the singular limit $\epsilon\searrow0$ as well as the
properties of the smoothing procedure used to obtain well-prepared initial
conditions above. This yields the following algorithm:
\begin{enumerate}
  \item Define initial vortex positions $a^0\in\Omega_*^{n}$ and degrees
    $d\in\{\pm1\}^n$. Set up the initial phase $\varphi$ such that the initial
    condition $u_{\epsilon}^0 = u_\epsilon^*(a^0,d)$ satisfies the homogeneous
    Neumann boundary conditions. Compute and store the radial function
    $f_{\epsilon}$.
  \item Evolve $a(t)$ according to the ODE \eqref{eq:ODE1} up to some maximum
    time $T$.
  \item At time $t>0$, build back first an approximation of the wave function
    from the vortex positions $a(t)$ as
    \begin{equation}\label{eq:smoothed_approx}
      u_{\epsilon}^*(t) = u_{\epsilon}^*(a(t), d) =
      u^*(\cdot; a(t), d) \prod_{j=1}^n f_{\epsilon}(|\cdot -
      a_j(t)|),
    \end{equation}
    where $u^*(x; a(t), d)$ is the canonical harmonic map (see
    Remark~\ref{rmk:canon}) defined by
    \begin{equation*}
      u^*(x; a(t), d) = \exp(\ii \varphi(x)) \prod_{j=1}^n
      \left(\frac{x - a_j}{|x - a_j|}\right)^{d_j},
    \end{equation*}
    with $\varphi$ the unique zero-mean harmonic function solving
    \begin{equation}\label{eq:phi_rebuilt}
      \begin{cases}
        \Delta \varphi = 0 \quad\text{in } \Omega,\\ \displaystyle
        \partial_\nu \varphi(x) = - \sum_{j=1}^n d_j\partial_\nu\theta(x-a_j)
        = \sum_{j=1}^n d_j \partial_\tau \ln|x-a_j|
        \quad\text{on } \partial\Omega.
      \end{cases}
    \end{equation}
    Here, $\theta$ denotes the angle of a 2D vector (seen as the argument of the
    associated complex number), and $\nu$ (resp. $\tau$) the outward (resp.
    tangential) normal vector.
    Computing $\varphi$ with the appropriate boundary conditions is required so that
    the approximate solution $u_{\epsilon}^*(t)$ satisfies homogeneous
    Neumann boundary conditions. This Laplace's equation
    can be solved using the same harmonic polynomial basis as the
    one used to solve \eqref{eq:Req}. However, note that the harmonic function
    $\varphi$ is defined only up to a constant: this implies that the reconstructed wave
    function $u_\epsilon^*$ is an approximation of $u_\epsilon$ only
    up to a constant phase.
    Nevertheless, we note that a global phase factor is nothing but a trivial gauge transformation (see~\eqref{eq:gauge}), and therefore, all physical quantities of interest such as the super-current, the magnetic field or the vorticity are independent of this phase.

    An approximation of the magnetic potential $A$ can also be recovered by
    computing $\Xi$ as $\Xi_h+\Xi_p$ (see eqs.~\eqref{eq:g function} and
    \eqref{eq:Weq}) and then computing $A_* = \curl\Xi$. The magnetic field
    $h_*$ can be reconstructed with $h_*=\curl A_*$ or directly evaluated using
    Lemma~\ref{lem:simplified magnetic field}.
\end{enumerate}

Finally, the method we propose can be summarized by the diagram in
Figure~\ref{fig:diagram}.
\begin{figure}[h!]
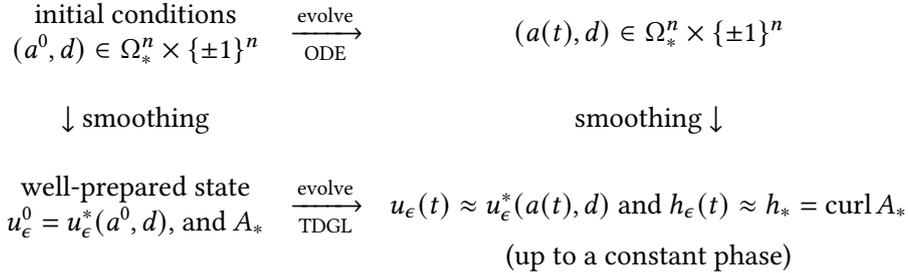

  \centering
  \begin{tabular}{ccc}
    \begin{tabular}[c]{@{}c@{}}initial conditions\\
      $(a^0,d)\in\Omega_*^{n}\times\{\pm1\}^n$\end{tabular}
    & $\xrightarrow[\text{ODE}]{\text{evolve}}$ &
    $(a(t),d)\in\Omega_*^{n}\times\{\pm1\}^n$                      \\
    \multicolumn{1}{l}{}
    & \multicolumn{1}{l}{}                      & \multicolumn{1}{l}{} \\
    $\downarrow$ smoothing                                                                                                          &                                           &  smoothing $\downarrow$\\
    \multicolumn{1}{l}{}
    & \multicolumn{1}{l}{}                      & \multicolumn{1}{l}{} \\
    \begin{tabular}[c]{@{}c@{}}well-prepared state\\ $u_\epsilon^0
      = u_{\epsilon}^*(a^0,d)$, and $A_*$\end{tabular} &
    $\xrightarrow[\text{TDGL}]{\text{evolve}}$ &
    $u_\epsilon(t)\approx u_{\epsilon}^*(a(t),d)$ and
    $h_\epsilon(t)\approx h_* = \curl A_*$ \\ & &
    \text{(up to a constant phase)}
  \end{tabular}
  \caption{Diagram summarizing the numerical simulation of the TDGL
    equation \eqref{eq:GPE_mag} via vortex tracking.}
  \label{fig:diagram}
\end{figure}

\subsection{Vortex positions as initial conditions}\label{sec:numres_1}

We now present some numerical simulations of approximate solutions to
\eqref{eq:GPE_mag} in the regime of small $\epsilon$. The reference
vortex trajectories are those from Section~\ref{ssec:HD_sim}. The solution of
\eqref{eq:radialODE} is approximated numerically with high accuracy using 1D
finite elements with small mesh size $\delta r = 10^{-5}$ and $r_0=0.1$. This is
computed once and for all in an offline stage, before running the simulation.

Finally, the solution to Laplace's equation \eqref{eq:phi_rebuilt} in the
smoothing process at time $t>0$ is computed using harmonic polynomials of degree
$m=128$. Here, the initial vortex positions $a^0$ are provided as
input data and the initial phase $\varphi$ is computed in order to match the
Neumann boundary conditions (see \eqref{eq:init}). We present the
numerical results obtained by the application of our numerical method to cases
1, 2, 3 and 4 from Section~\ref{ssec:HD_sim}, for $\epsilon=0.02$ and
display, for each of them, the modulus and the phase of $u_\epsilon^*$ as
well as the magnetic field at different times $t$. Note that the boundary
conditions of the magnetic field $h_*$ do match the constant value of $h_{\rm ex}$
and that, independently of the sign of $h_{\rm ex}$, vortices with a positive
(resp. negative) degree induce a positive (resp. negative) magnetic field.

\begin{figure}[p!]
  \includegraphics[width=0.5\linewidth]{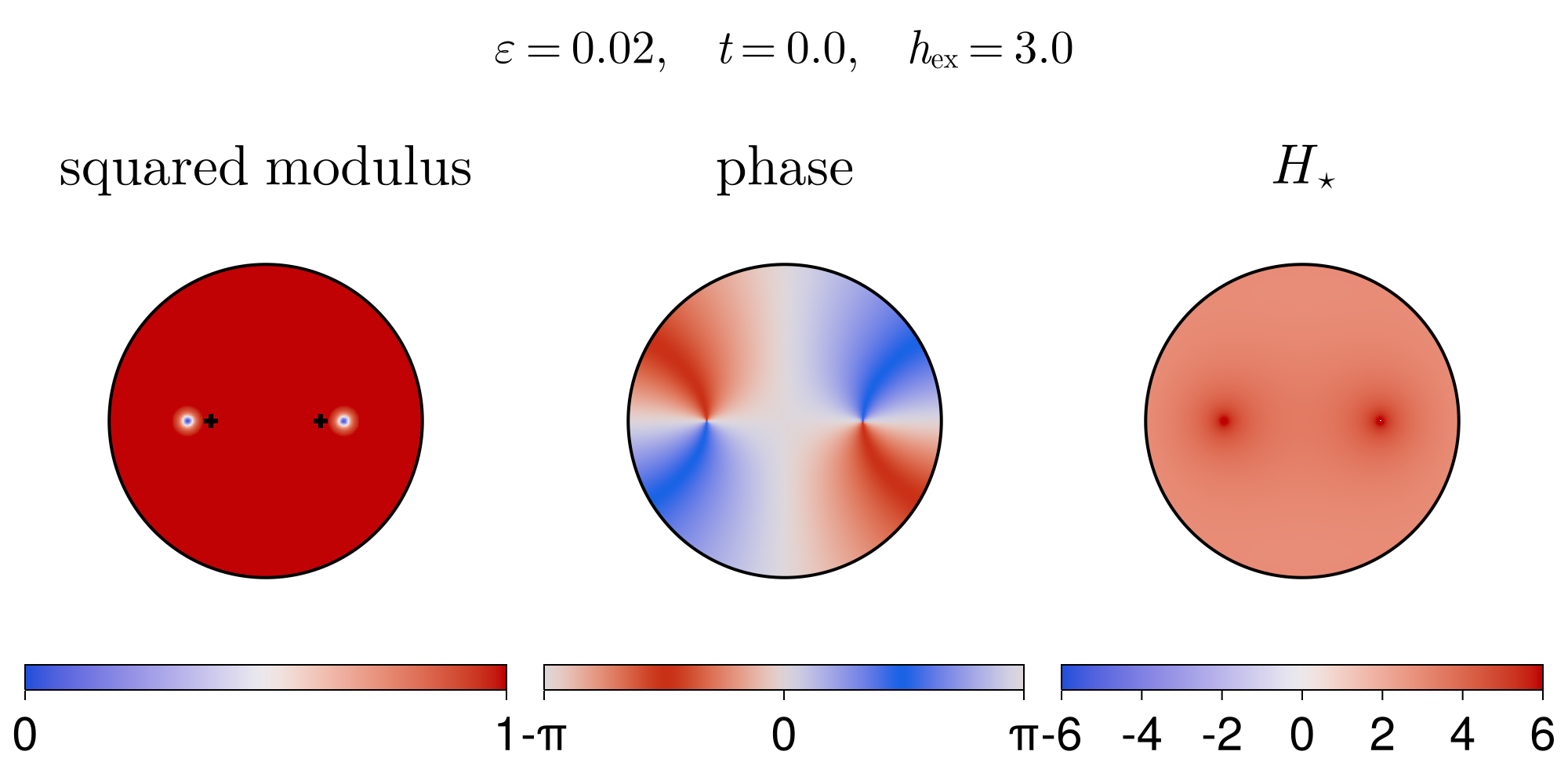}\hfill
  \includegraphics[width=0.5\linewidth]{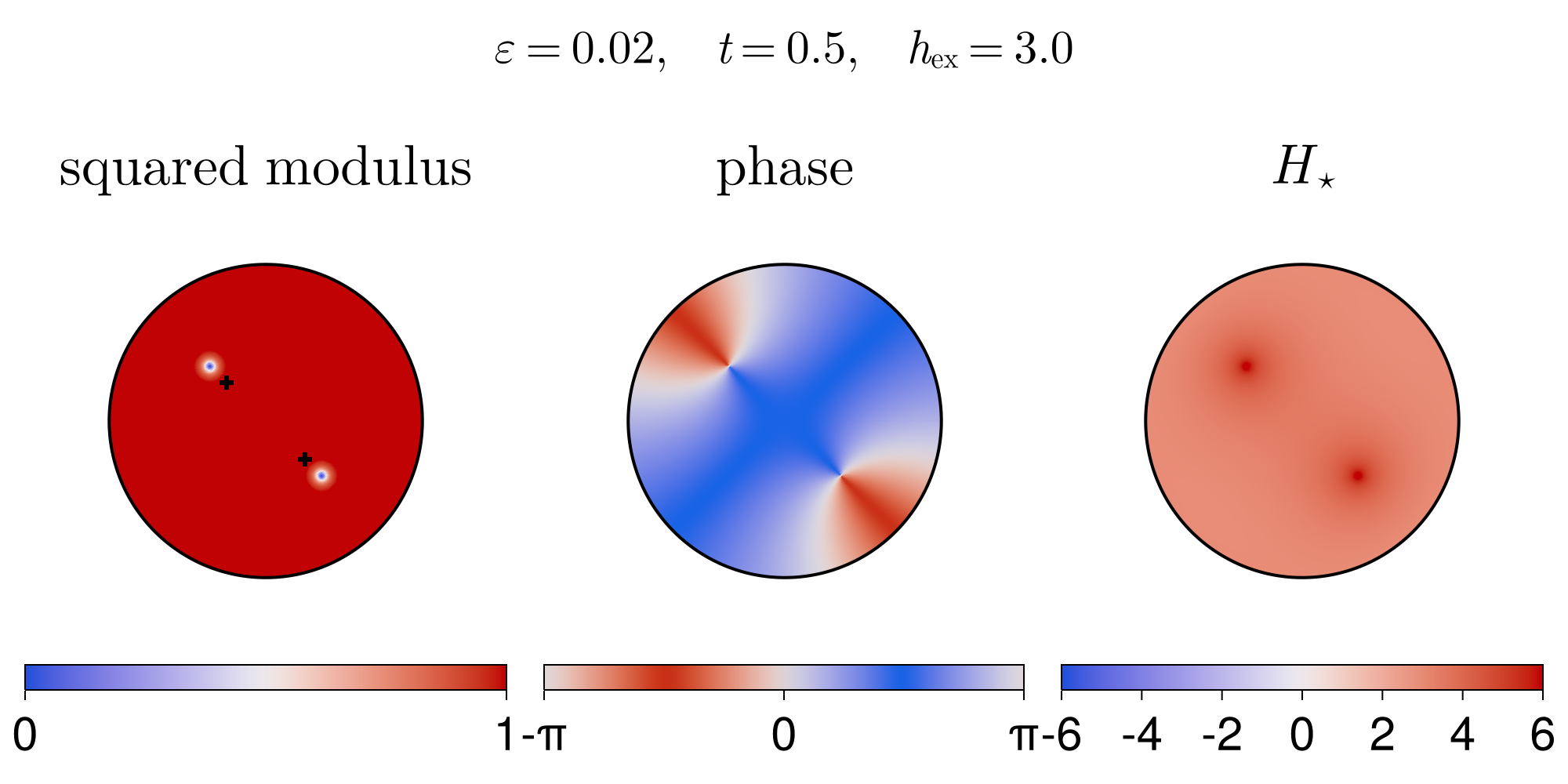}\\
  \caption{\textbf{Case 1} -- $h_{\rm ex}=3$ : Squared modulus and phase of
    $u_{\epsilon}^*(t)$, with amplitude of $h_*$ for different times $t$.}
  \label{fig:smooth_case1_3}
\end{figure}

\begin{figure}[p!]
  \includegraphics[width=0.5\linewidth]{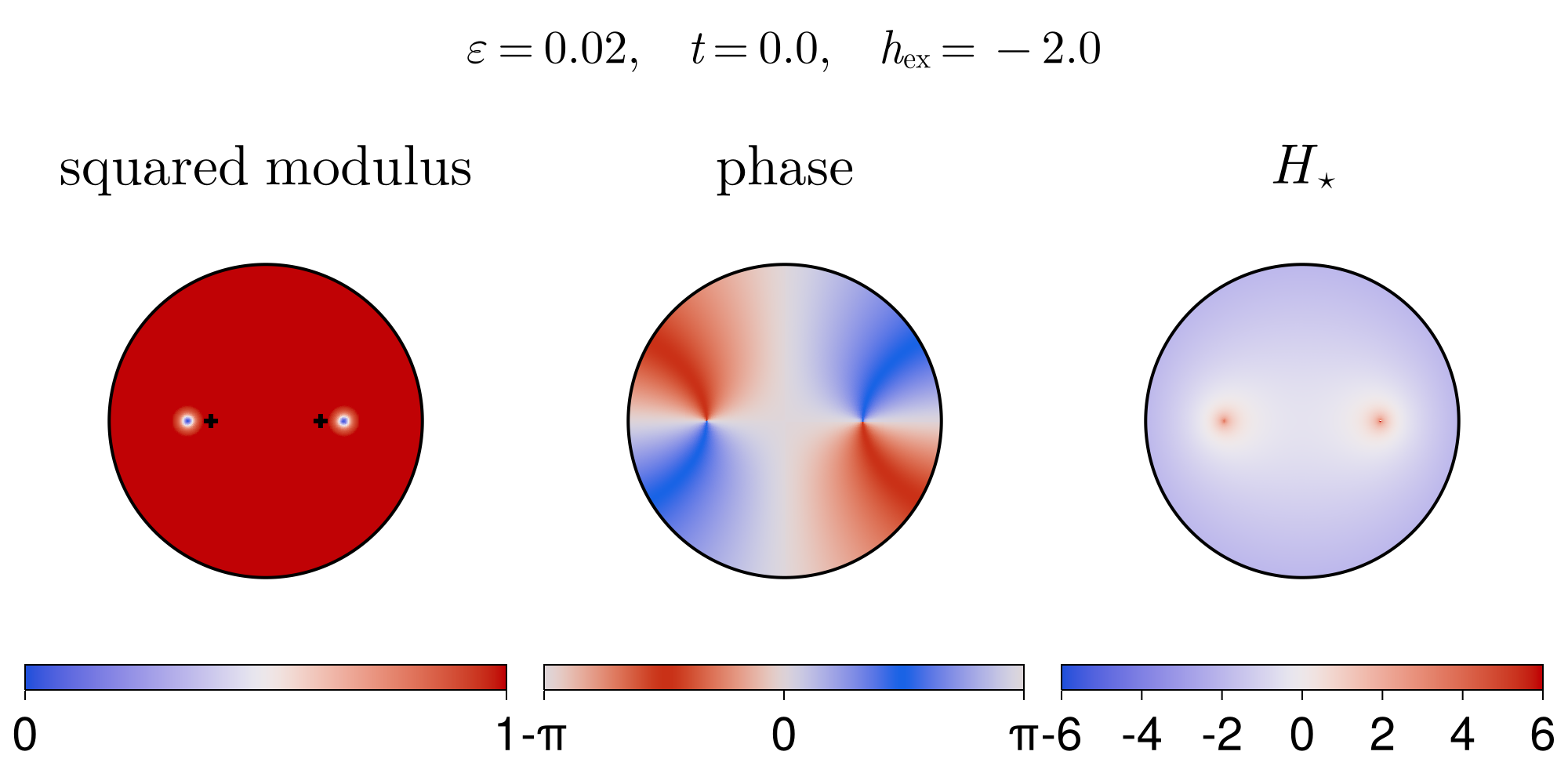}\hfill
  \includegraphics[width=0.5\linewidth]{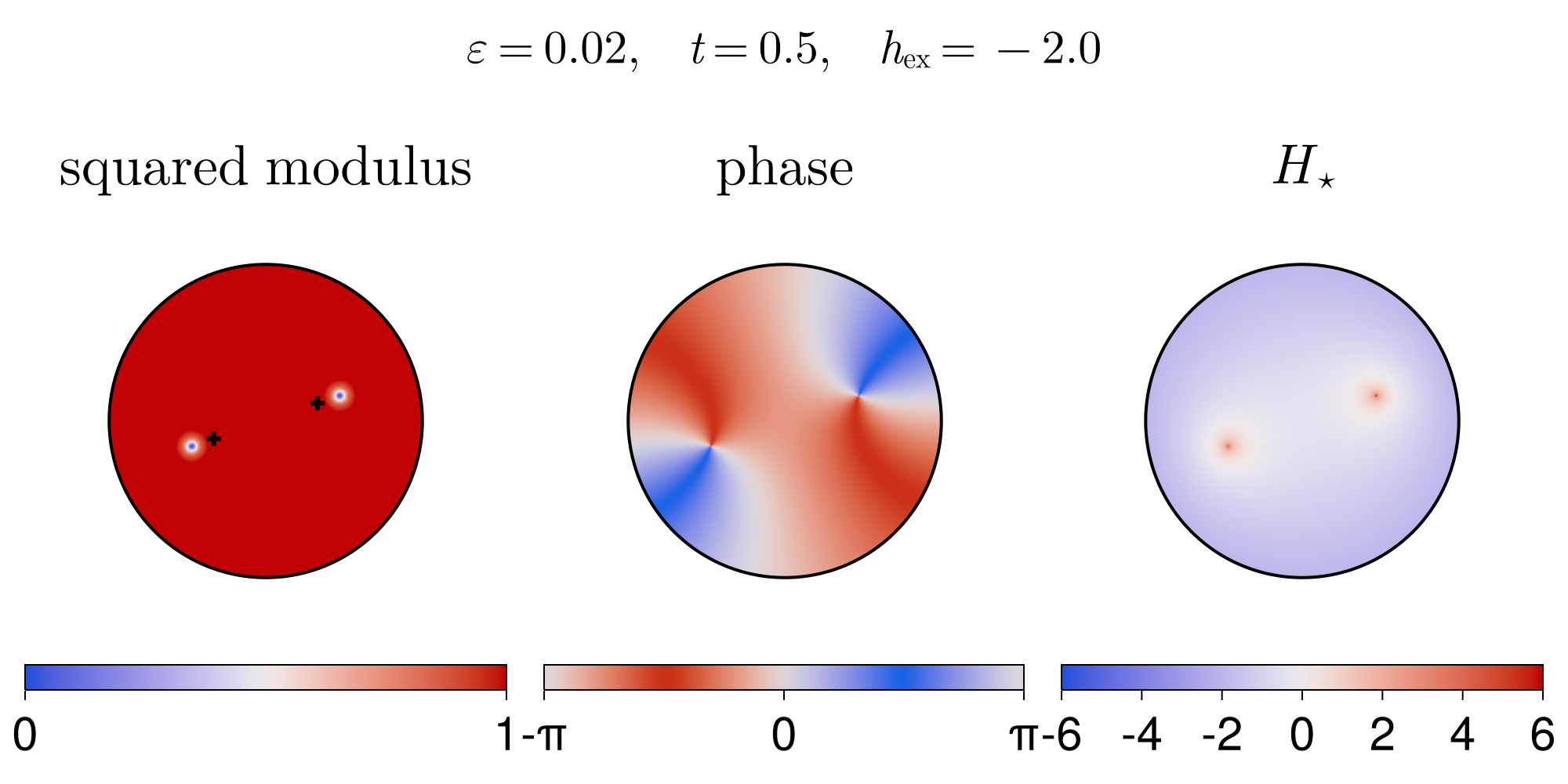}\\
  \caption{\textbf{Case 1} -- $h_{\rm ex}=-2$ : Squared modulus and phase of
    $u_{\epsilon}^*(t)$, with amplitude of $h_*$ for different times $t$.}
  \label{fig:smooth_case1_-2}
\end{figure}

\begin{figure}[p!]
  \includegraphics[width=0.5\linewidth]{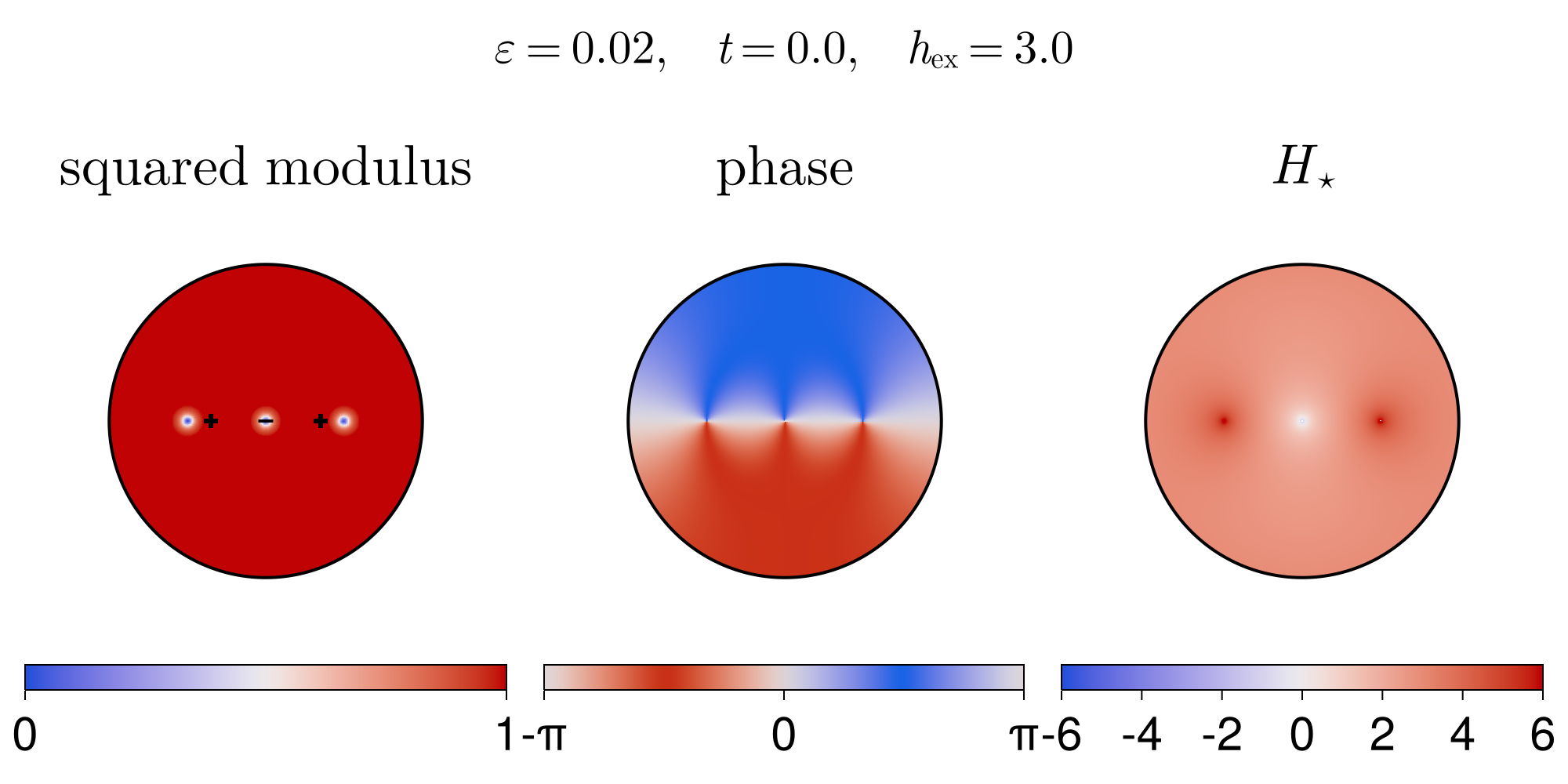}\hfill
  \includegraphics[width=0.5\linewidth]{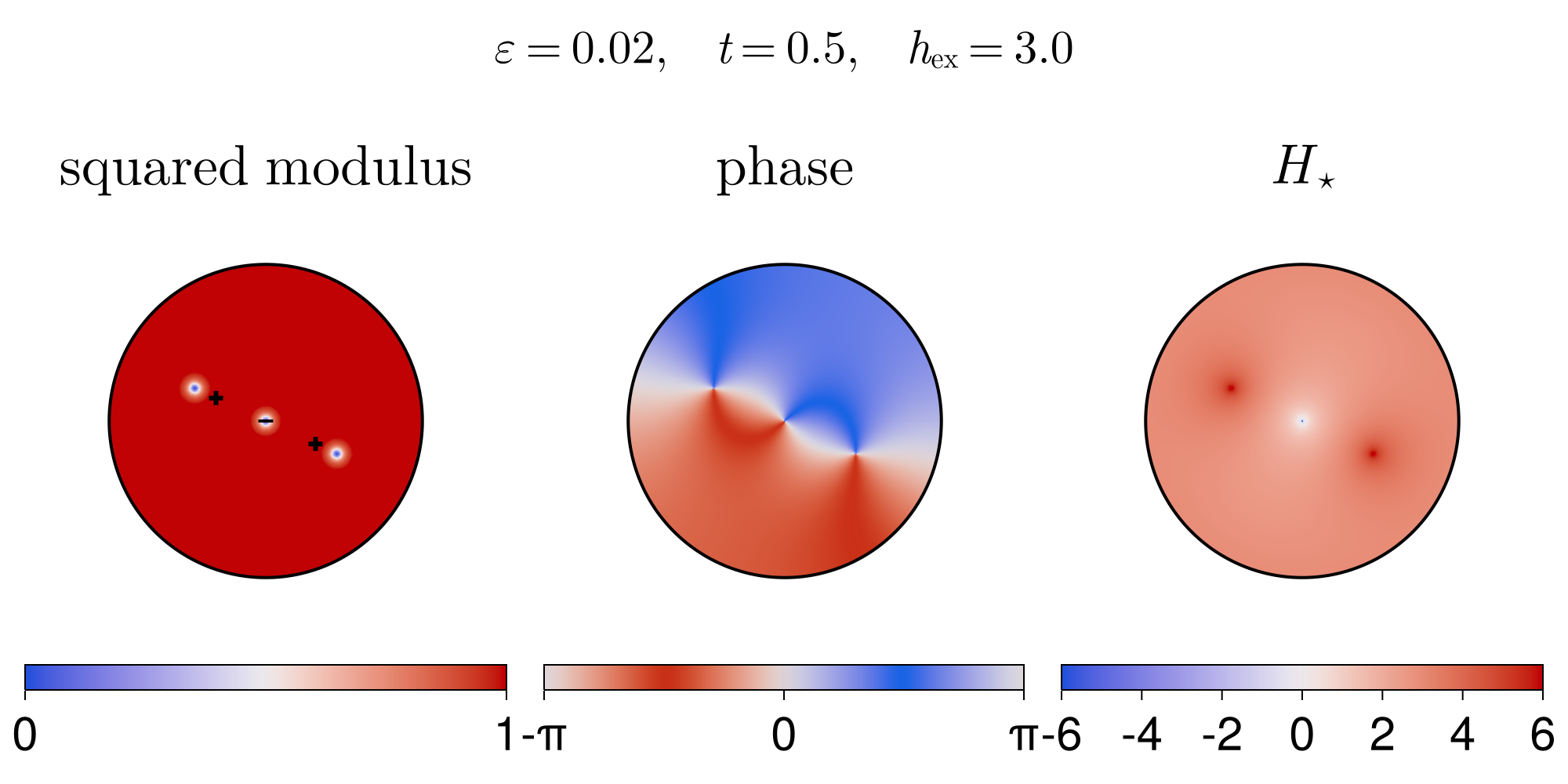}\\
  \caption{\textbf{Case 2} -- $h_{\rm ex}=3$ : Squared modulus and phase of
    $u_{\epsilon}^*(t)$, with amplitude of $h_*$ for different times $t$.}
  \label{fig:smooth_case2_3}
\end{figure}

\begin{figure}[p!]
  \includegraphics[width=0.5\linewidth]{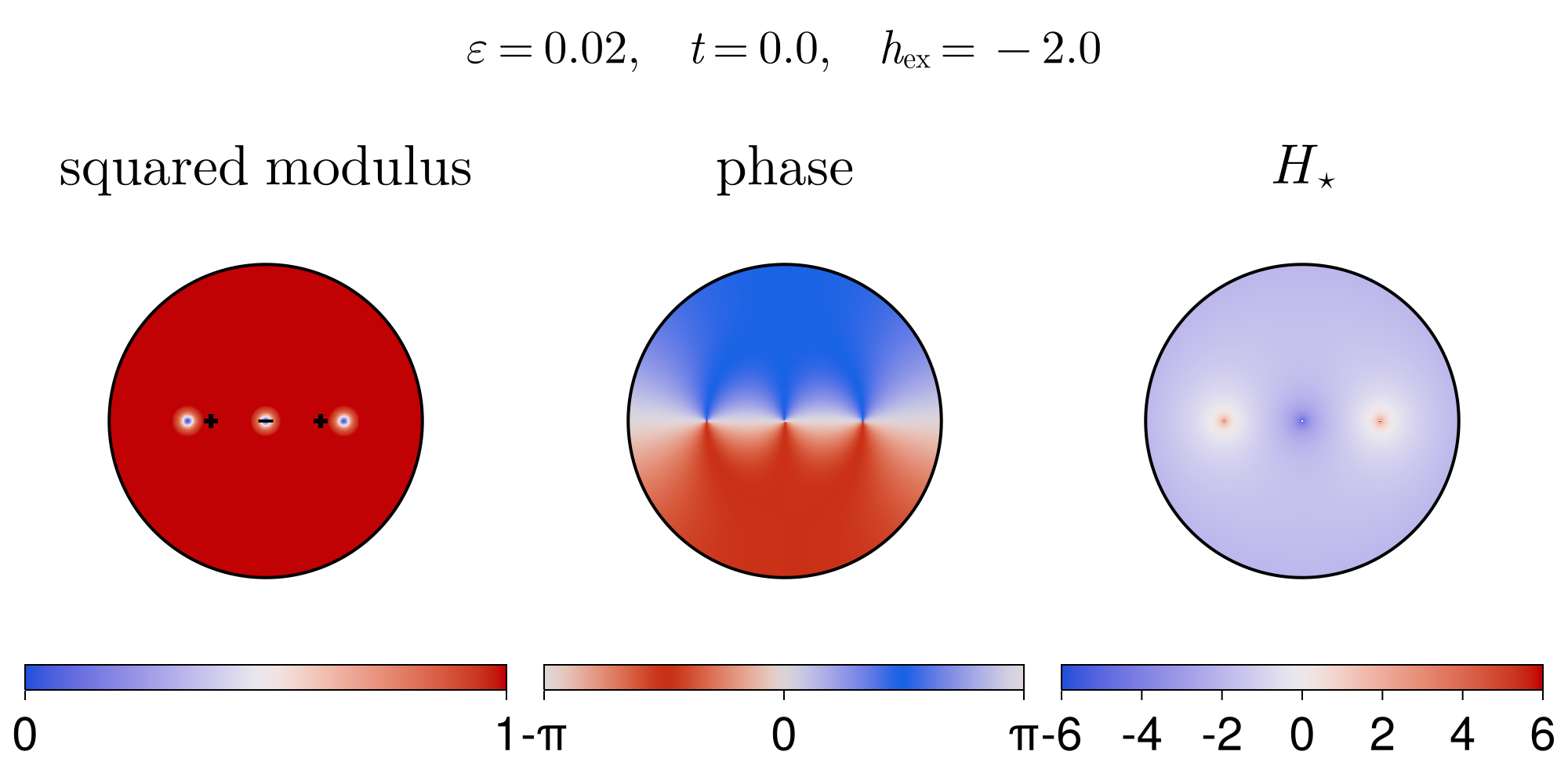}\hfill
  \includegraphics[width=0.5\linewidth]{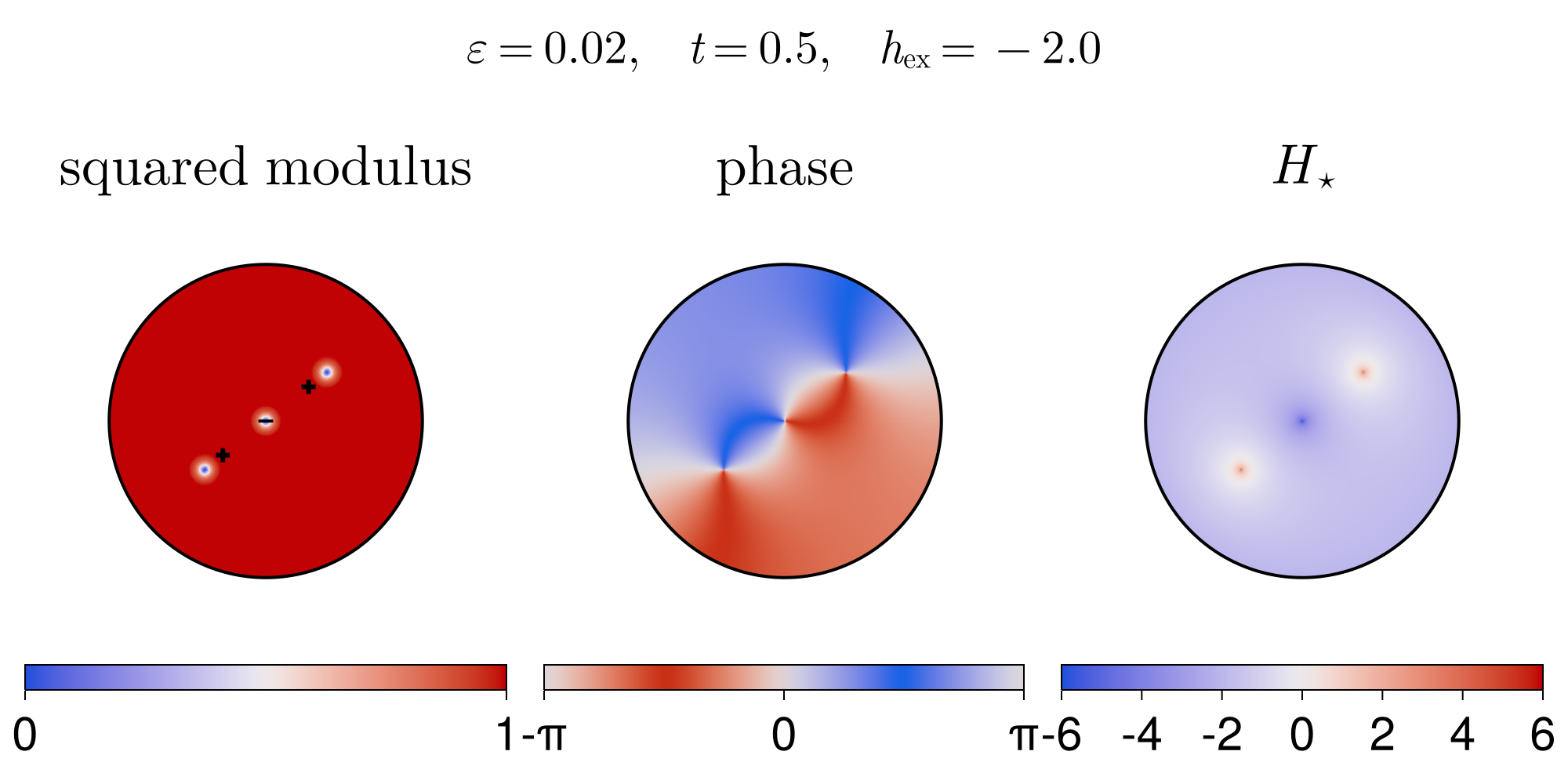}\\
  \caption{\textbf{Case 2} -- $h_{\rm ex}=-2$ : Squared modulus and phase of
    $u_{\epsilon}^*(t)$, with amplitude of $h_*$ for different times $t$.}
  \label{fig:smooth_case2_-2}
\end{figure}

\begin{figure}[p!]
  \includegraphics[width=0.5\linewidth]{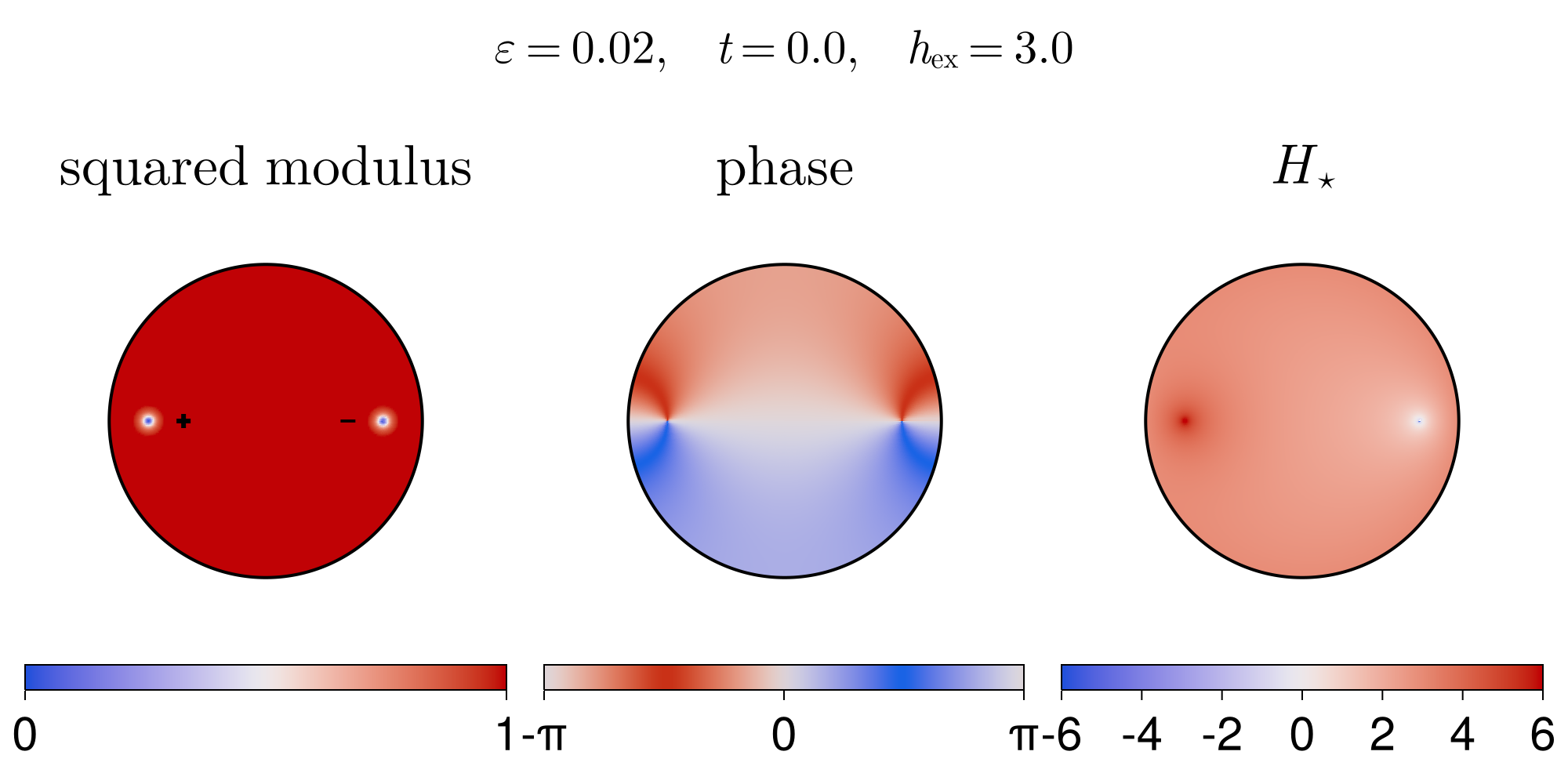}\hfill
  \includegraphics[width=0.5\linewidth]{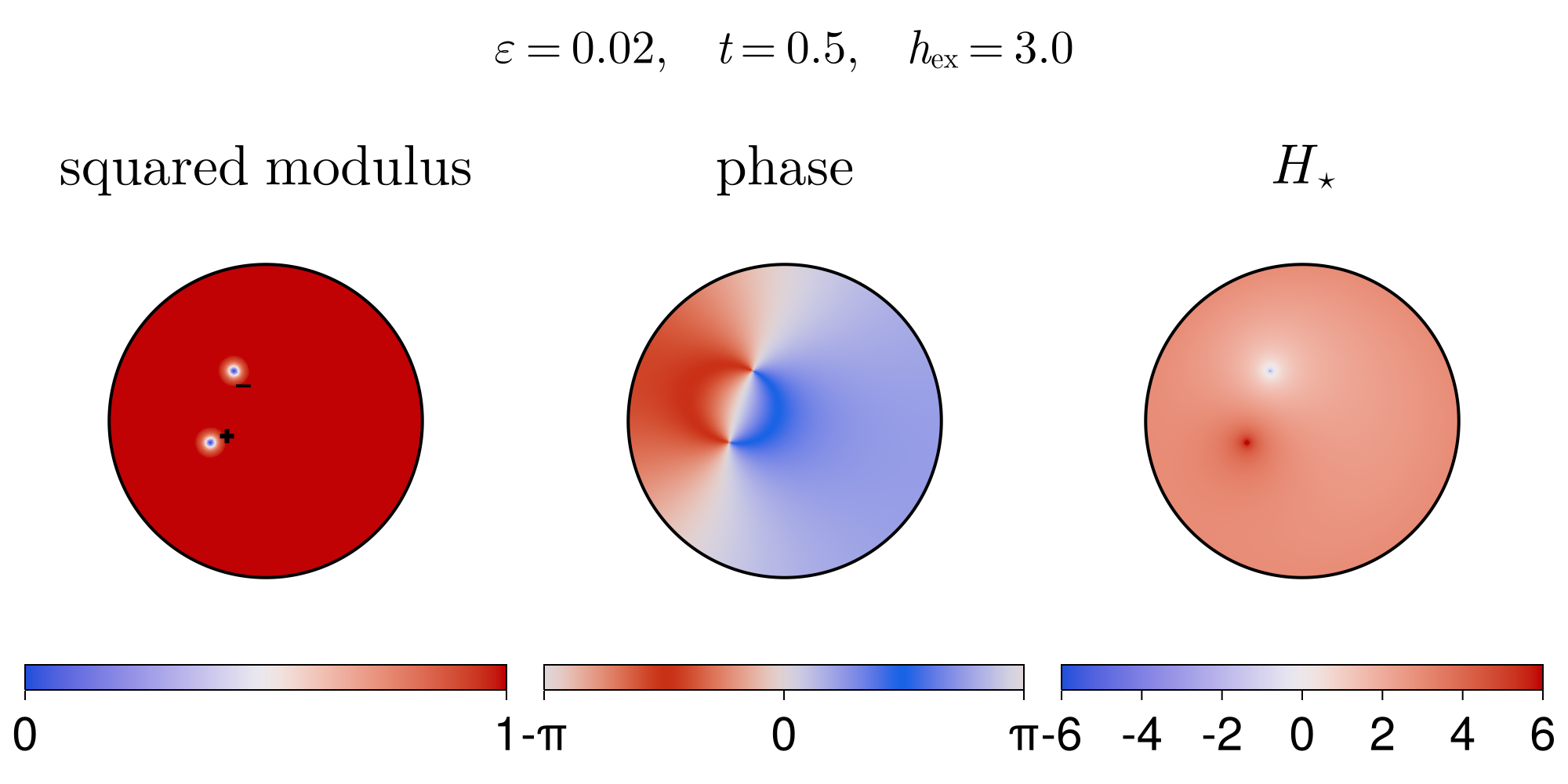}\\
  \caption{\textbf{Case 3} -- $h_{\rm ex}=3$ : Squared modulus and phase of
    $u_{\epsilon}^*(t)$, with amplitude of $h_*$ for different times $t$.}
  \label{fig:smooth_case3_3}
\end{figure}

\begin{figure}[p!]
  \includegraphics[width=0.5\linewidth]{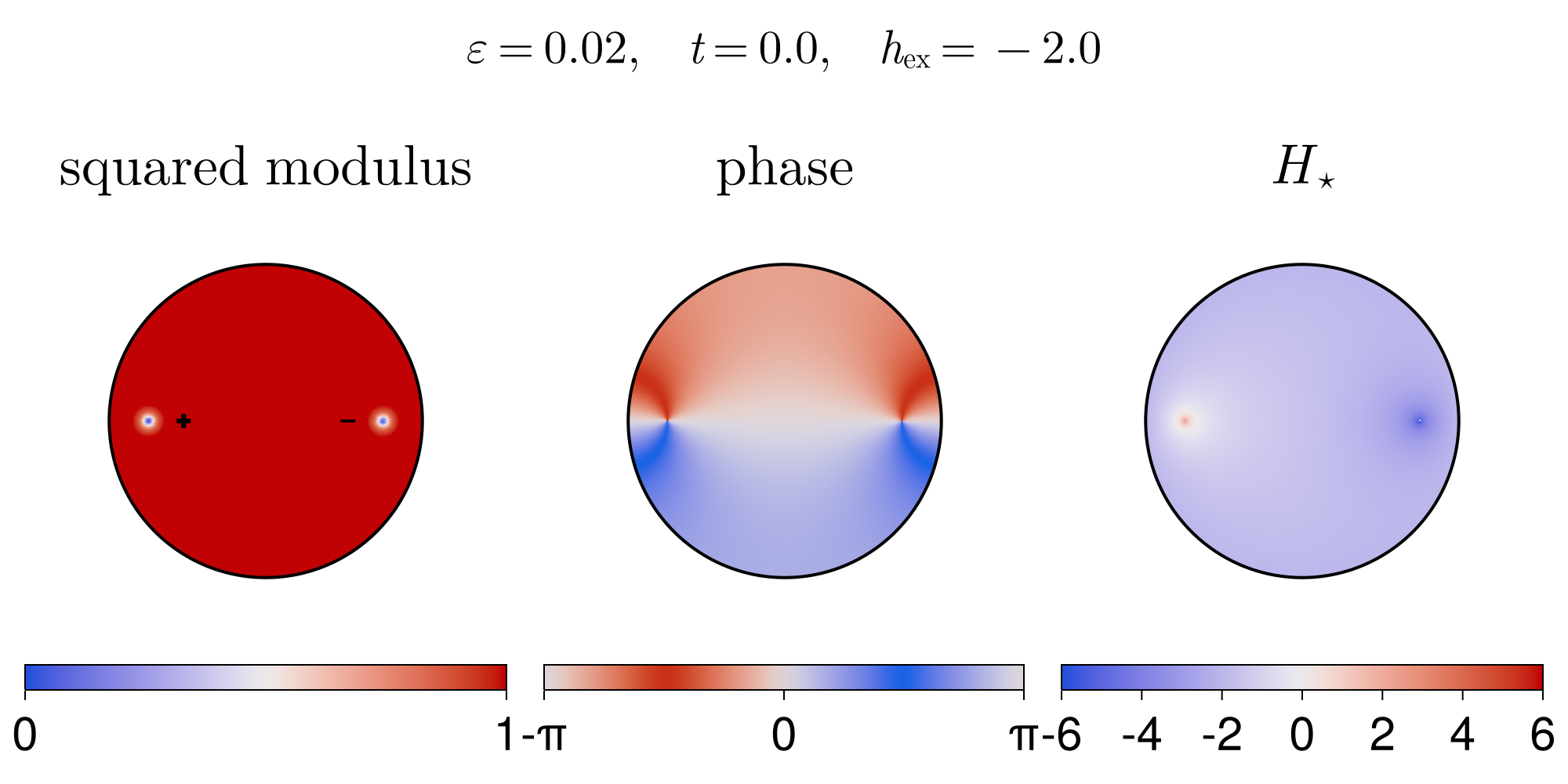}\hfill
  \includegraphics[width=0.5\linewidth]{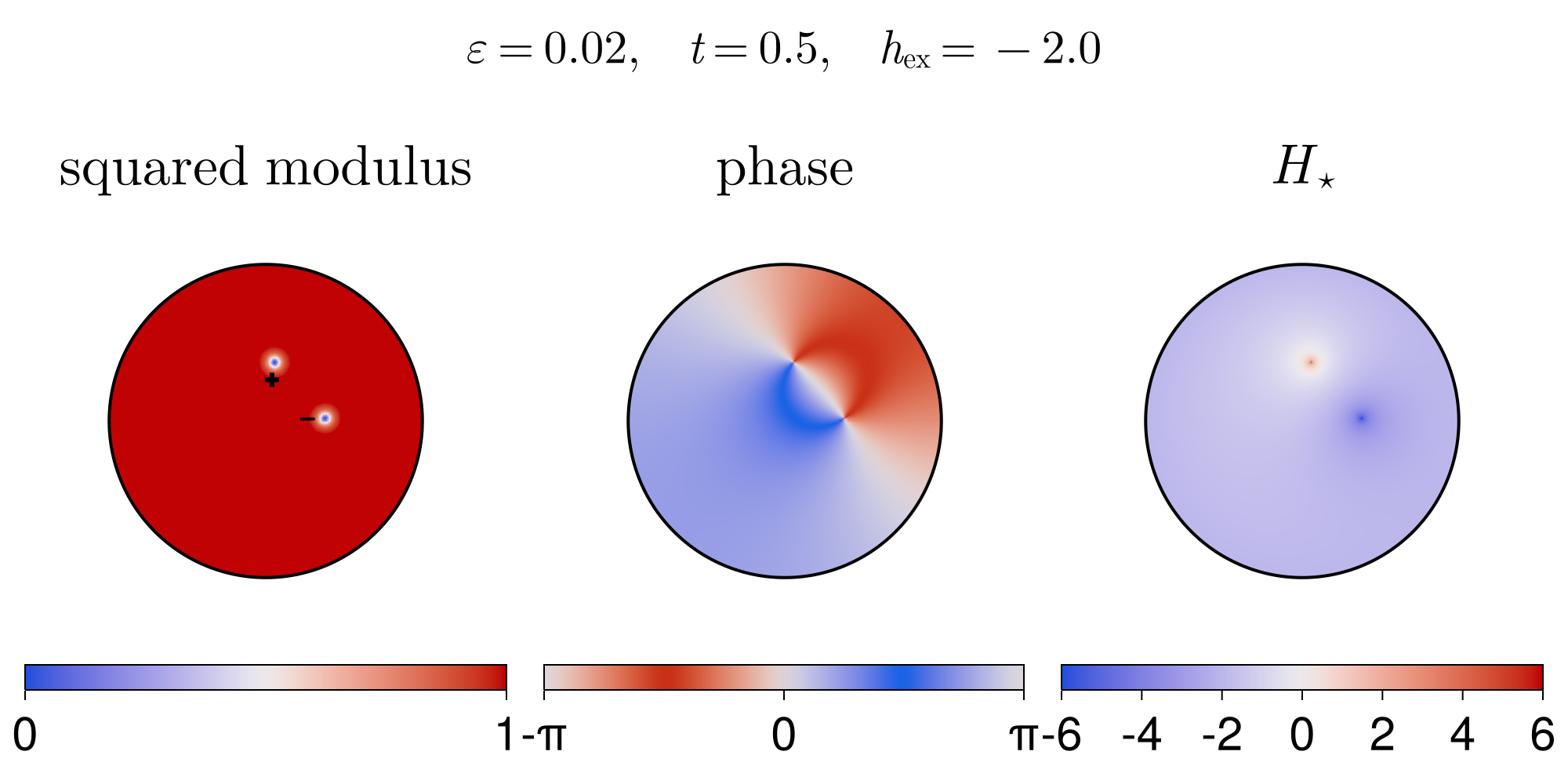}\\
  \caption{\textbf{Case 3} -- $h_{\rm ex}=-2$ : Squared modulus and phase of
    $u_{\epsilon}^*(t)$, with amplitude of $h_*$ for different times $t$.}
  \label{fig:smooth_case3_-2}
\end{figure}

\begin{figure}[p!]
  \includegraphics[width=0.5\linewidth]{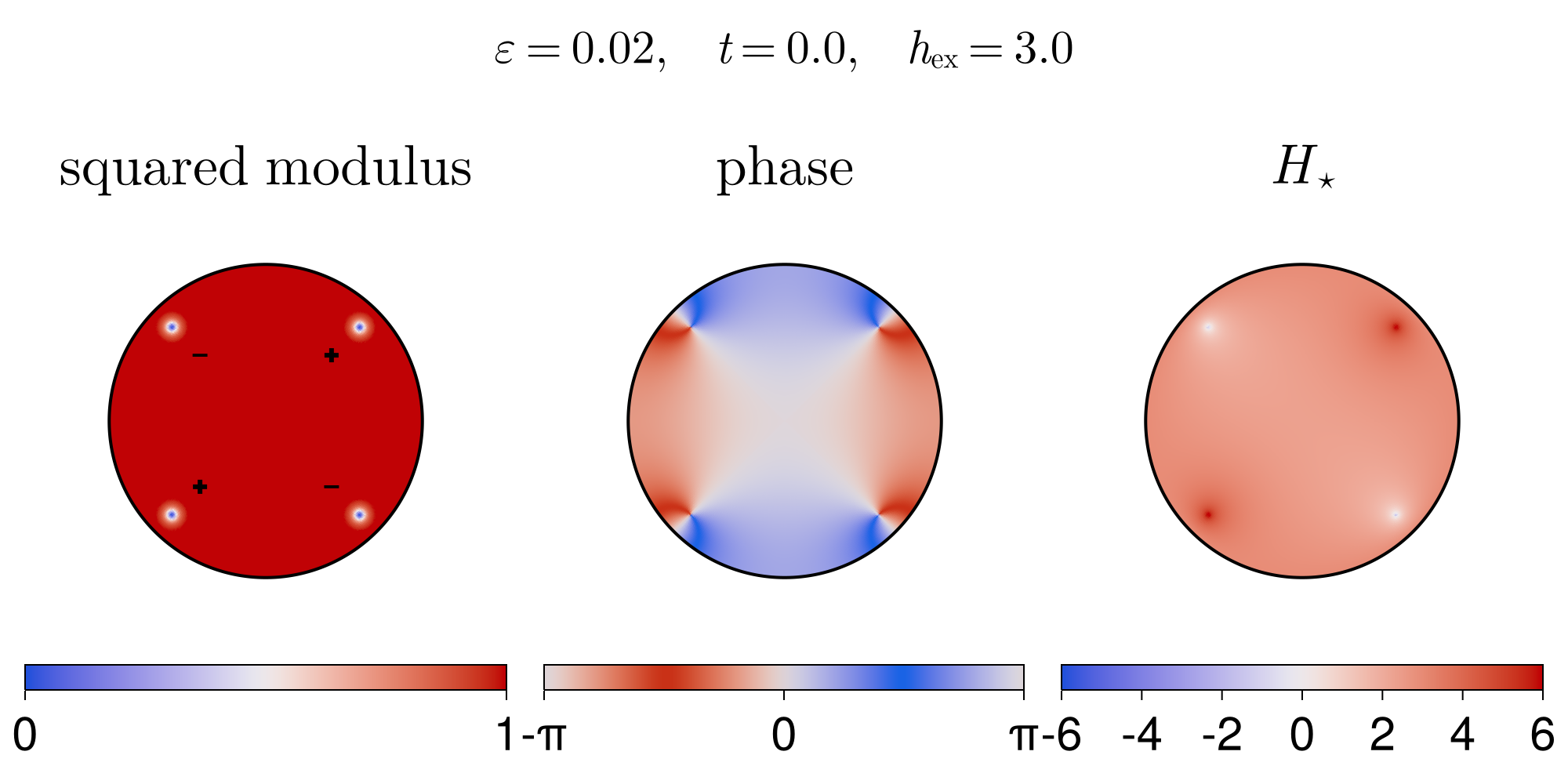}\hfill
  \includegraphics[width=0.5\linewidth]{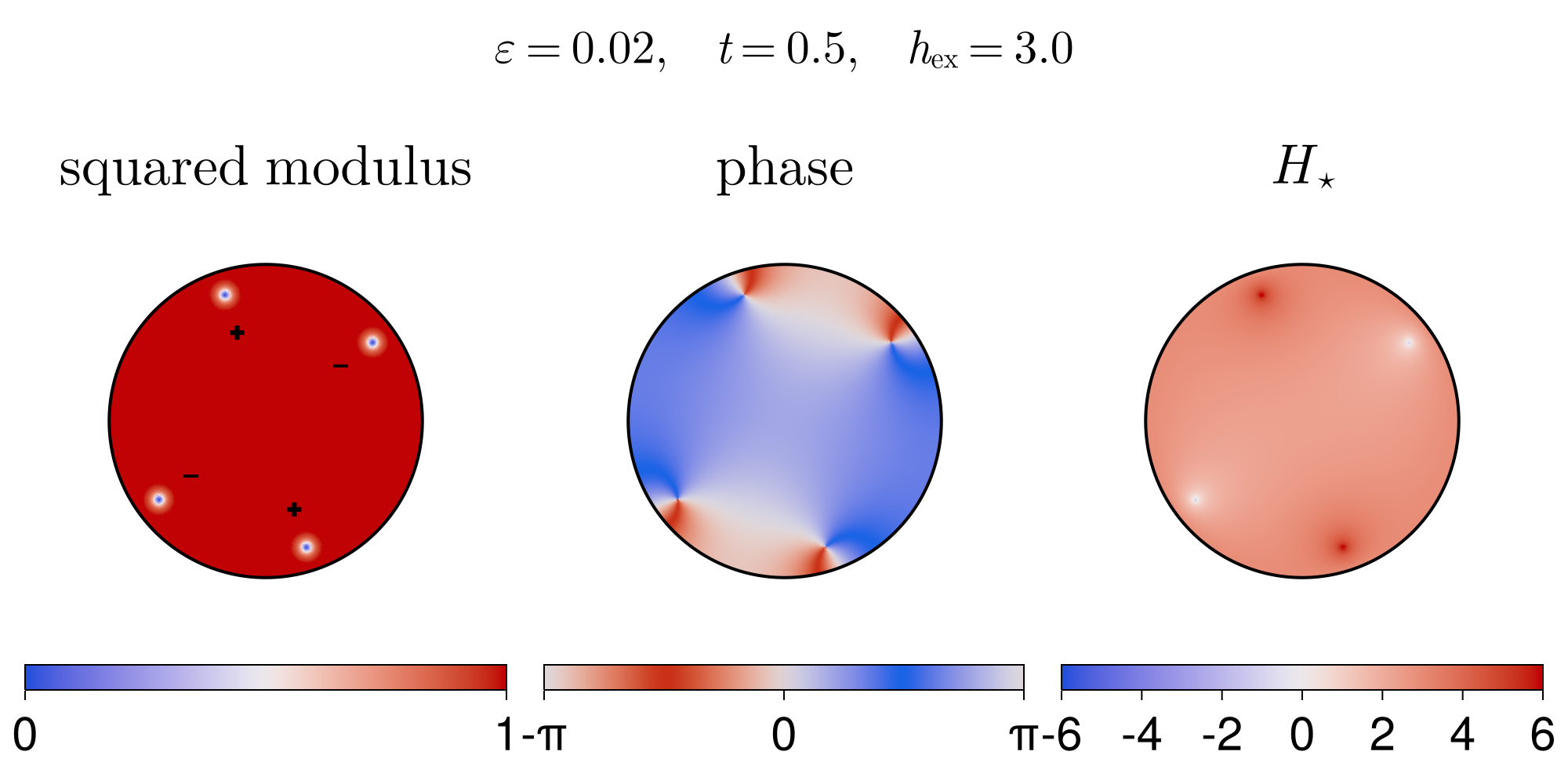}\\
  \caption{\textbf{Case 4} -- $h_{\rm ex}=3$ : Squared modulus and phase of
    $u_{\epsilon}^*(t)$, with amplitude of $h_*$ for different times $t$.}
  \label{fig:smooth_case4_3}
\end{figure}

\begin{figure}[p!]
  \includegraphics[width=0.5\linewidth]{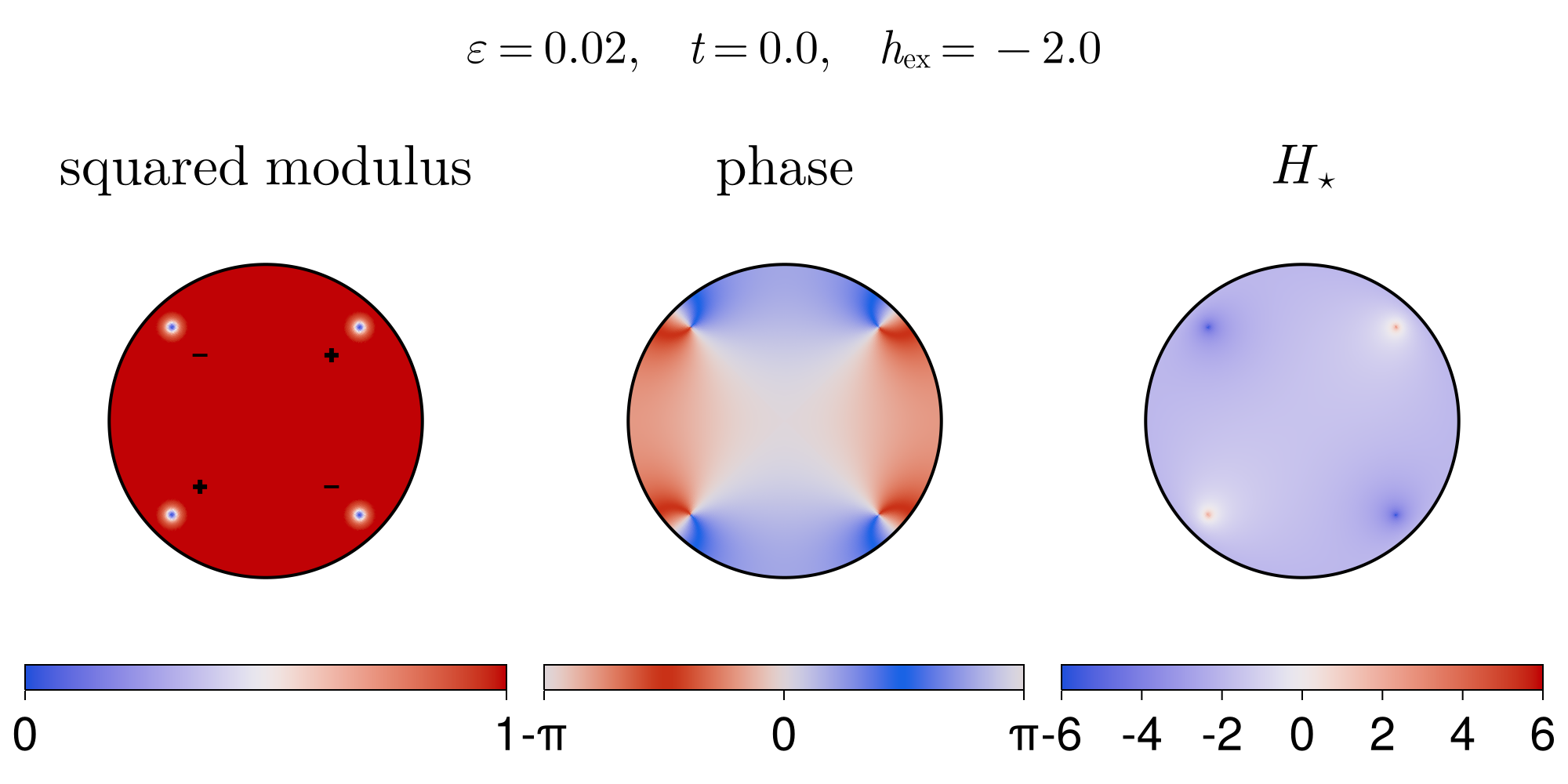}\hfill
\includegraphics[width=0.5\linewidth]{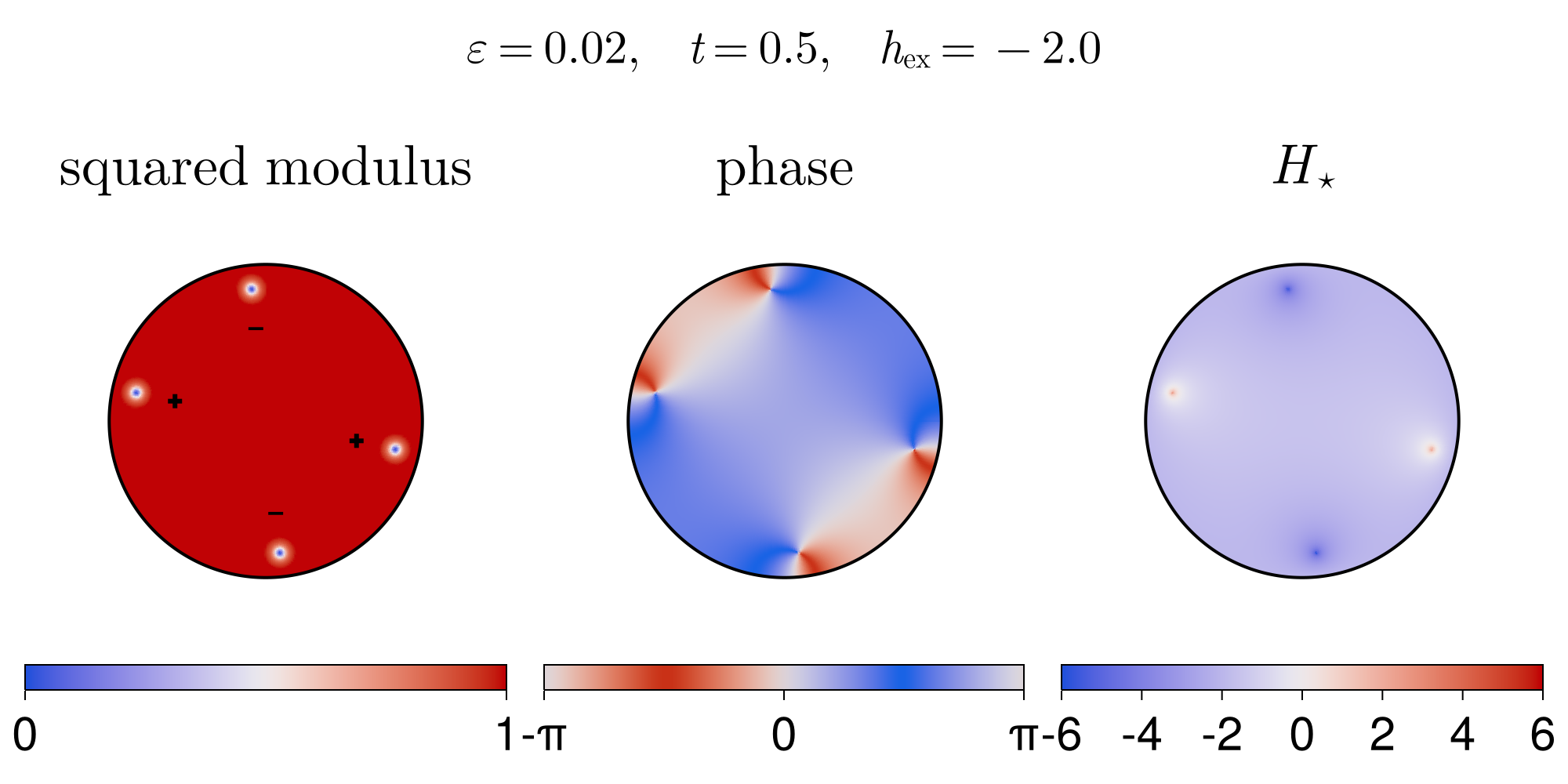}\\
  \caption{\textbf{Case 4} -- $h_{\rm ex}=-2$ : Squared modulus and phase of
    $u_{\epsilon}^*(t)$, with amplitude of $h_*$ for different times $t$.}
  \label{fig:smooth_case4_-2}
\end{figure}

\clearpage
\subsection{Numerical investigation of the small $\epsilon$ limit}

We focus now on Case 1, in the case $h_{\rm ex} =3$: there are two vortices with
identical winding number $+1$ and initial position $(\pm0.5,0)$. From
Figure~\ref{fig:trajectories}, we know that we expect, for small values of
$\epsilon$, the vortices to move around the circle of radius $0.5$ and
centered at $(0,0)$.

\subsubsection{Computational framework}

In order to compare the numerical method presented in Section~\ref{sec:method_1}
to the exact solution of the TDGL equation \eqref{eq:GPE_mag}, we compute a reference
solution $u_\epsilon(t)$ with the finite element solver \texttt{Gridap.jl}
\cite{BV20} and meshing tool GMSH \cite{GR09}. The magnetic TDGL
equation \eqref{eq:GPE_mag} is solved for initial conditions given by the vortex
configuration from Case 1 and the smoothing procedure described in
Section~\ref{ssec:WP}. We use a forward Euler time discretization scheme and
follow a numerical scheme adapted from \cite{GX23}.
The mesh size is of order $\delta x=5\cdot10^{-2}$ with local \emph{a priori} refined
mesh on the vortex trajectory with mesh size $\delta x=3.5\cdot10^{-3}$, see
Figure~\ref{fig:mesh}. The time step is $\delta t=10^{-5}$. The finite element
solver is then used to obtain various reference solutions $u_\epsilon$ for
$\epsilon \in \{0.07, 0.1, 0.13\}$, up to $T=1$. We also compute a
reference solution for $\epsilon = 0.05$, with a finer refinement. For this finer mesh, the total number of degrees of freedom for
  $u_\epsilon$, $A_\epsilon$ and $\phi_\epsilon$ together is
  1.422.686. Note that the simulation with such a mesh can take several days
for small values of $\epsilon$: for instance, in our case, the simulations for
$\epsilon = 0.05$ took about ten days to run with a non-optimized code on a
small sized cluster.

The reconstructed approximation $u_{\epsilon}^*(t)$ from
Section~\ref{sec:method_1} is obtained from the trajectories in the limit
$\epsilon\searrow0$, with $r_0 = 0.3$ and $m=128$ for the reconstruction step. The
Hamiltonian dynamics is solved numerically with a RK4 method, with time step
$\delta t = 10^{-5}$ and $m=128$. Note that this time, the simulation only takes
a few seconds on a personal laptop.

\begin{figure}[h!]
  \centering
  \includegraphics[width=0.4\linewidth]{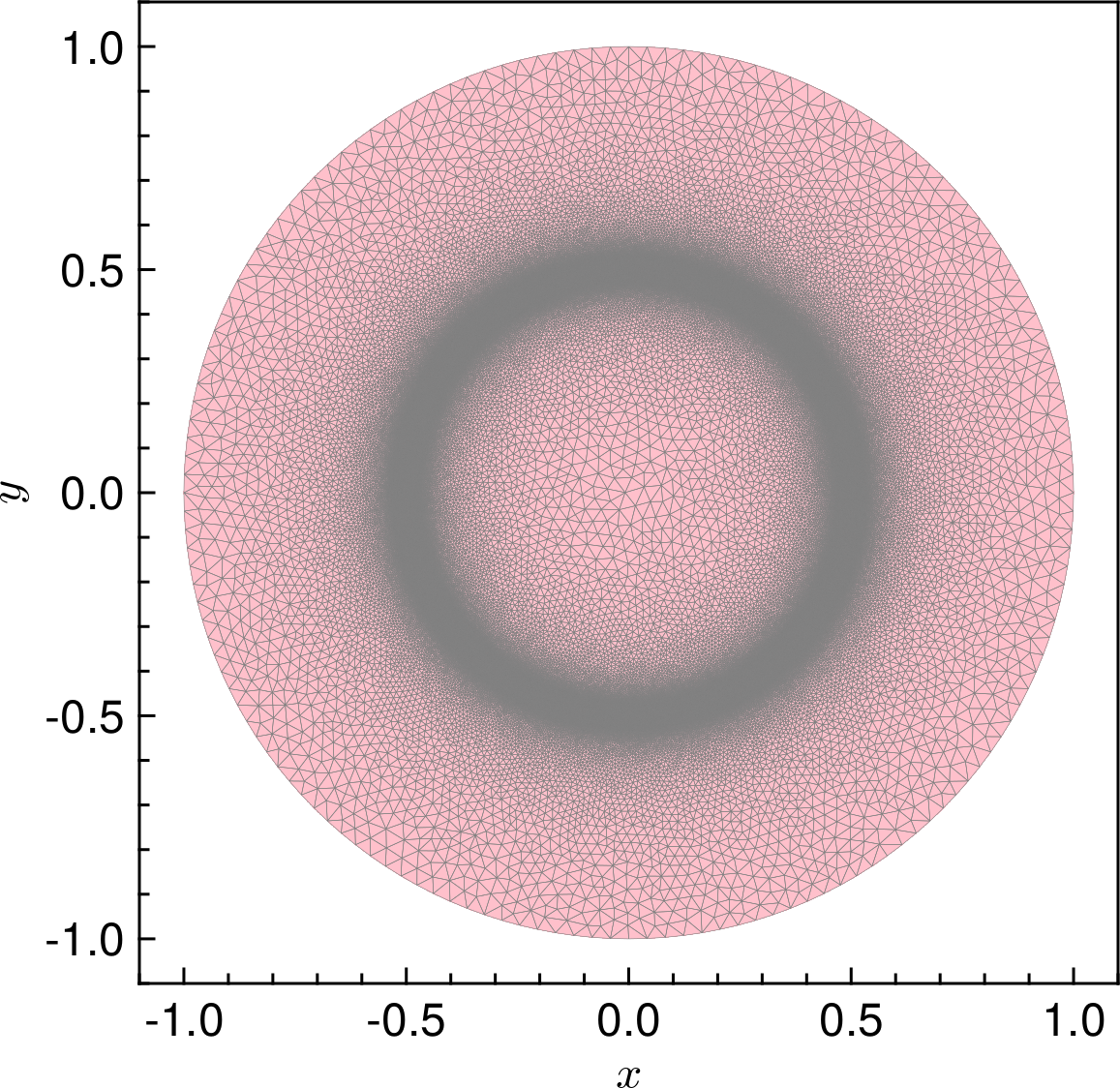}
  \caption{Mesh of the unit disk used to obtain reference solutions.}
  \label{fig:mesh}
\end{figure}

\subsubsection{Numerical illustration of the singular limit $\epsilon\searrow0$}

In order to formally compare the
reference solution $u_\epsilon(t)$ and the reconstructed approximation
$u_\epsilon^*(t)$, we evaluate and compare different physical quantities,
namely: the vortex trajectories, the super-current, the induced magnetic fields
and the order parameter.

First, by tracking its isolated zeros, for instance using one of the algorithms
introduced in \cite{DLVN25,KSDH23}, we are able to localize the vortices of the
reference solution $u_\epsilon(t)$ and plot the associated trajectories in
Figure~\ref{fig:traj_eps}, in the spirit of the study realized in \cite{BT14}.
Then, we plot in Figure~\ref{fig:error_eps} the following quantities:

\begin{equation*}
  \frac{\norm{j_{A_\epsilon}(u_\epsilon)(t) -
      j_{A_*}(u_{\epsilon}^*)(t)}_{L^{\frac43}}}
  {\norm{j_{A_\epsilon}(u_\epsilon)(t)}_{L^{\frac43}}}, \quad
  \frac{\norm{h_\epsilon(t) - h_*(t)}_{L^{2}}}
  {\norm{h_\epsilon(t)}_{L^{2}}}, \quad\text{and}\quad
  \frac{\norm{u_\epsilon(t) - u_{\epsilon}^*(t)}_{L^2}}
  {\norm{u_\epsilon(t)}_{L^2}}.
\end{equation*}

Despite the increasing errors along time, the
smaller $\epsilon$, the better the approximation of these quantities. This
motivates further such a scheme to reduce the computational time required to
solve the TDGL equation in the regime of small $\epsilon$. One should keep in
mind though that this approximation remains valid as long as the vortices stay
far from each other and from the boundary and is not capable for instance of describing the annihilation of vortices and the resulting nonlinear phenomena.

\begin{figure}[h!]
  \centering
  \includegraphics[width=0.33\linewidth]{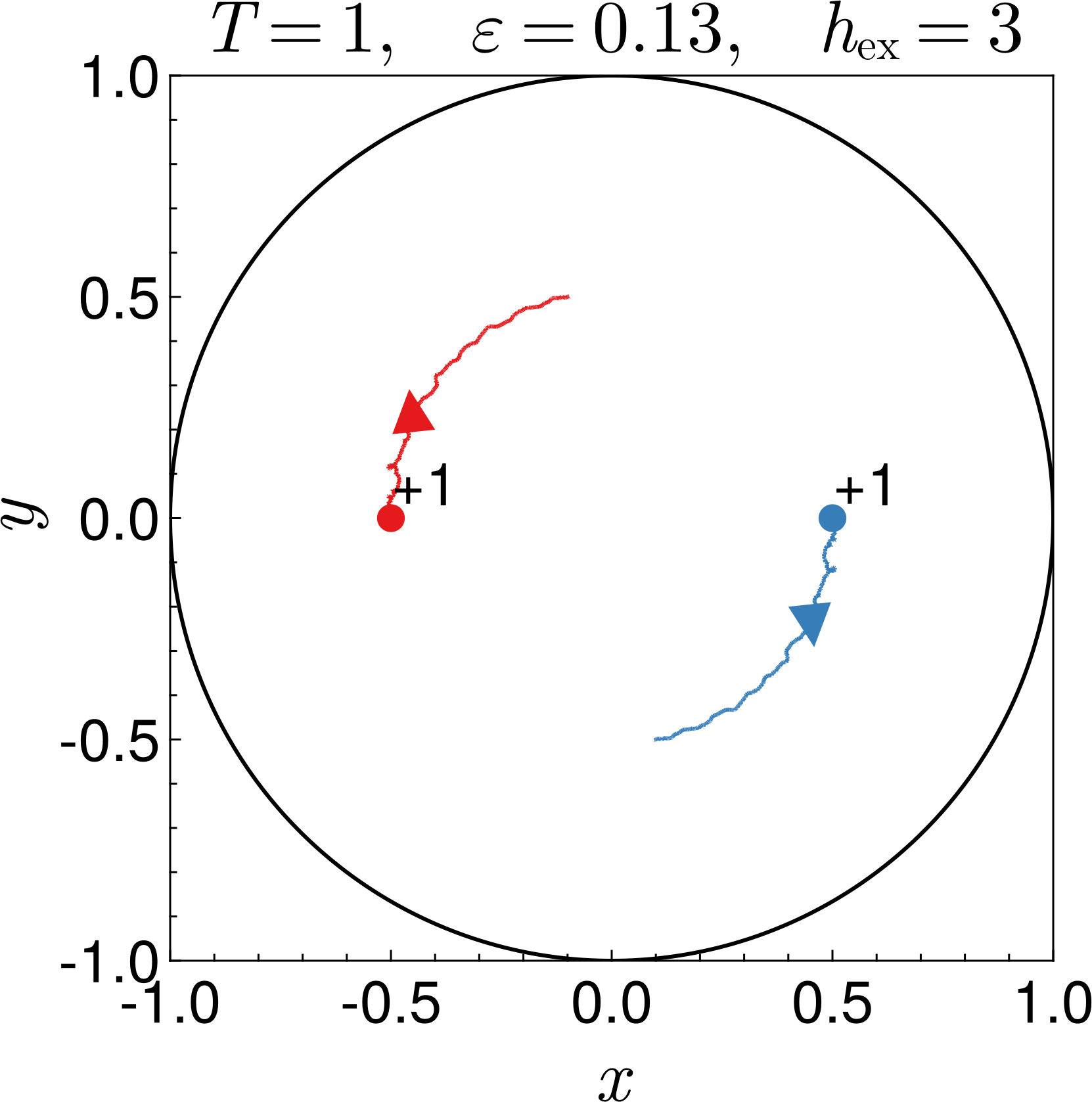}\hfill
  \includegraphics[width=0.33\linewidth]{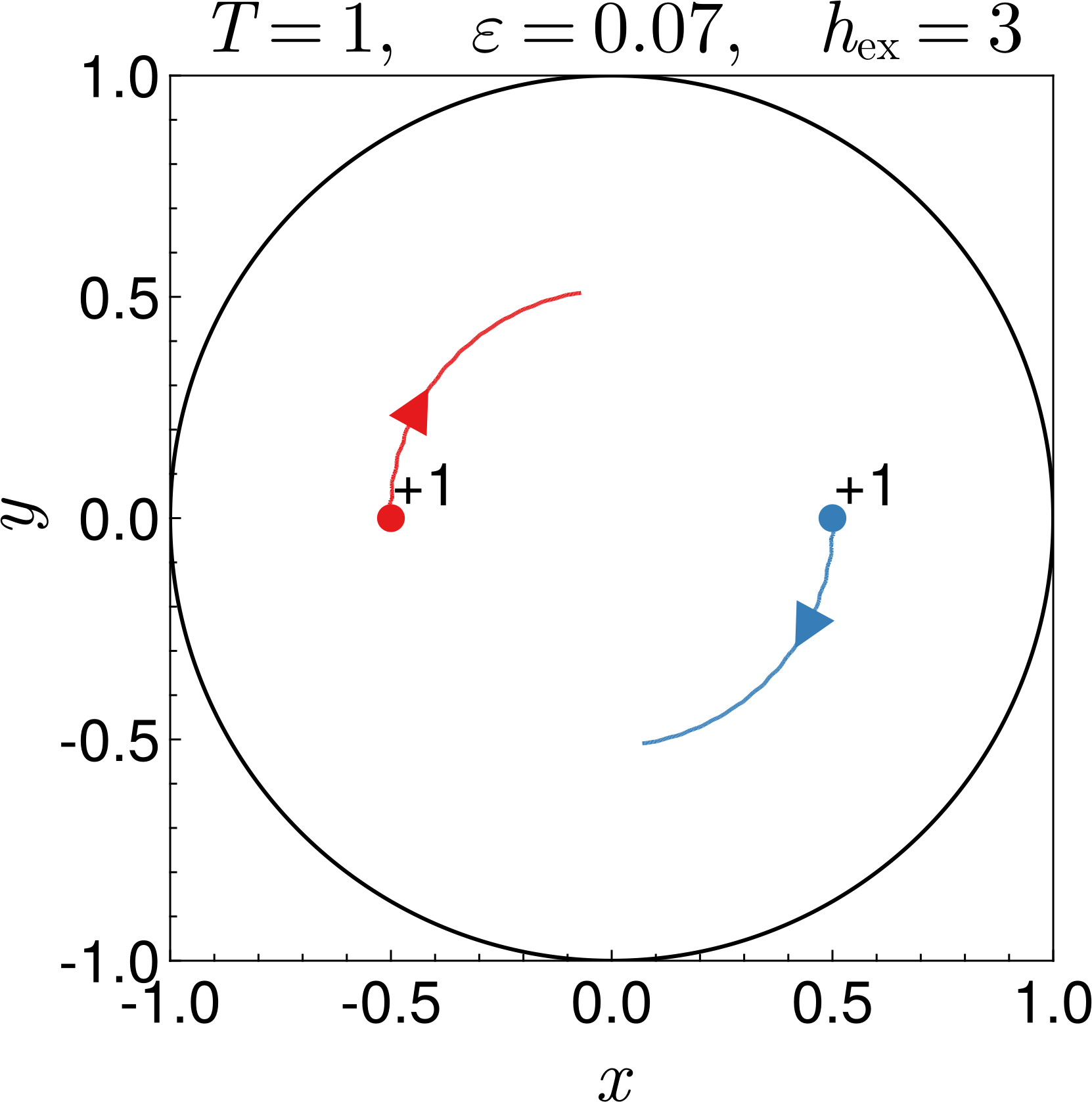}\hfill
  \includegraphics[width=0.33\linewidth]{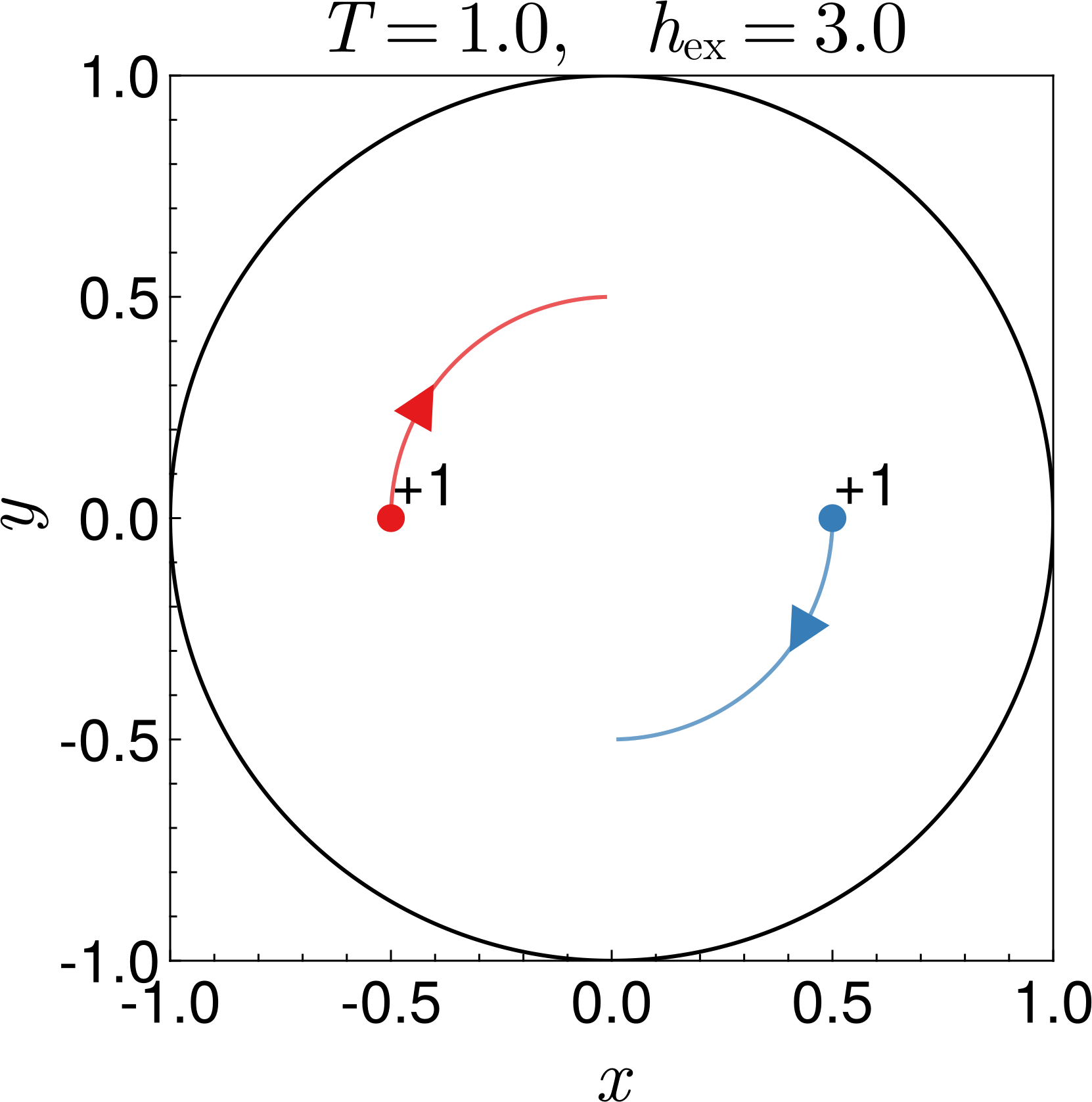}\\
  \caption{Vortex trajectories in the limit case for different values of
    $\epsilon > 0$ and for $\epsilon=0$ (right).}
  \label{fig:traj_eps}
\end{figure}

\begin{figure}[h!]
  \centering
  \includegraphics[width=0.48\linewidth]{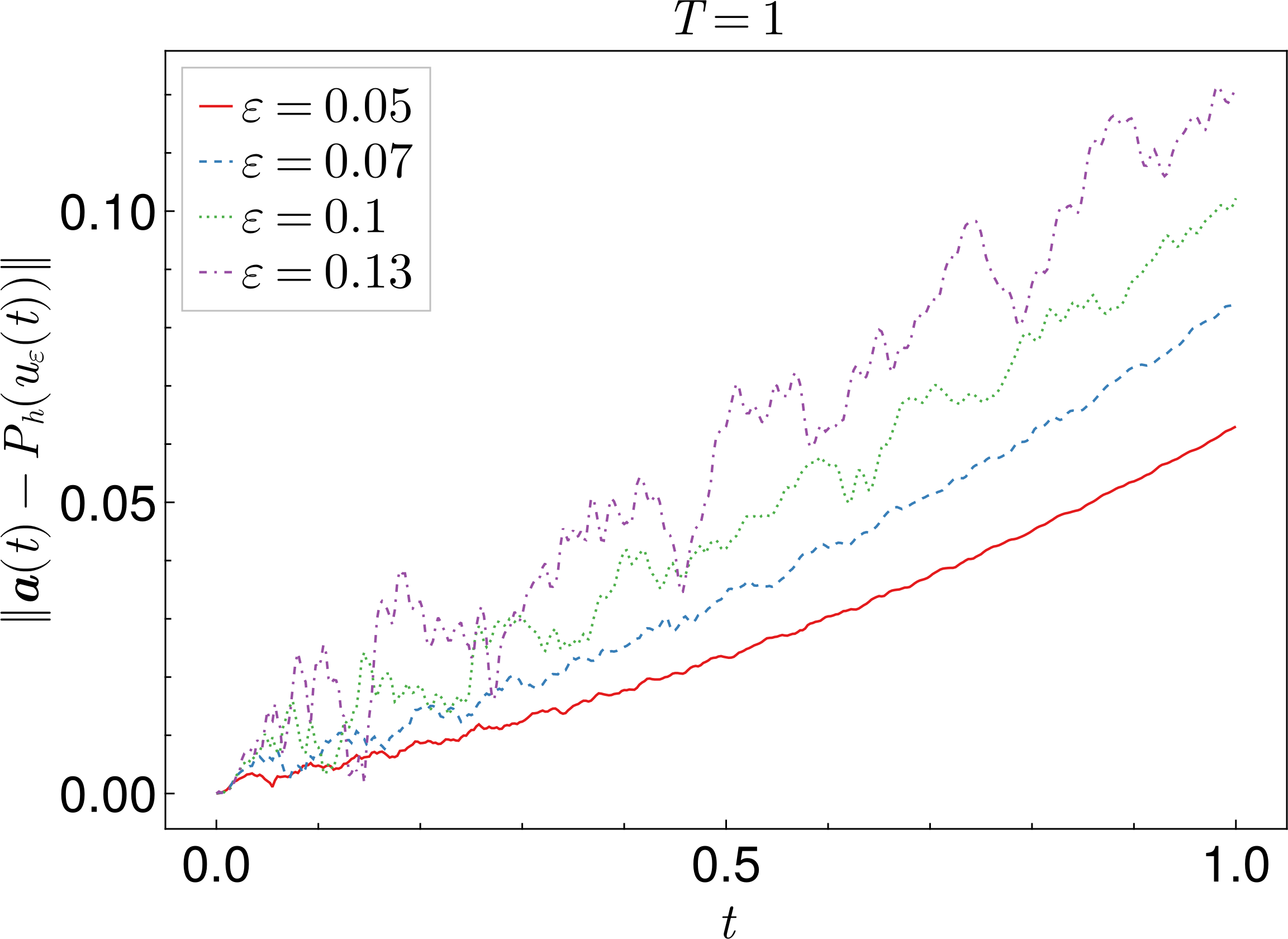}\hfill
  \includegraphics[width=0.48\linewidth]{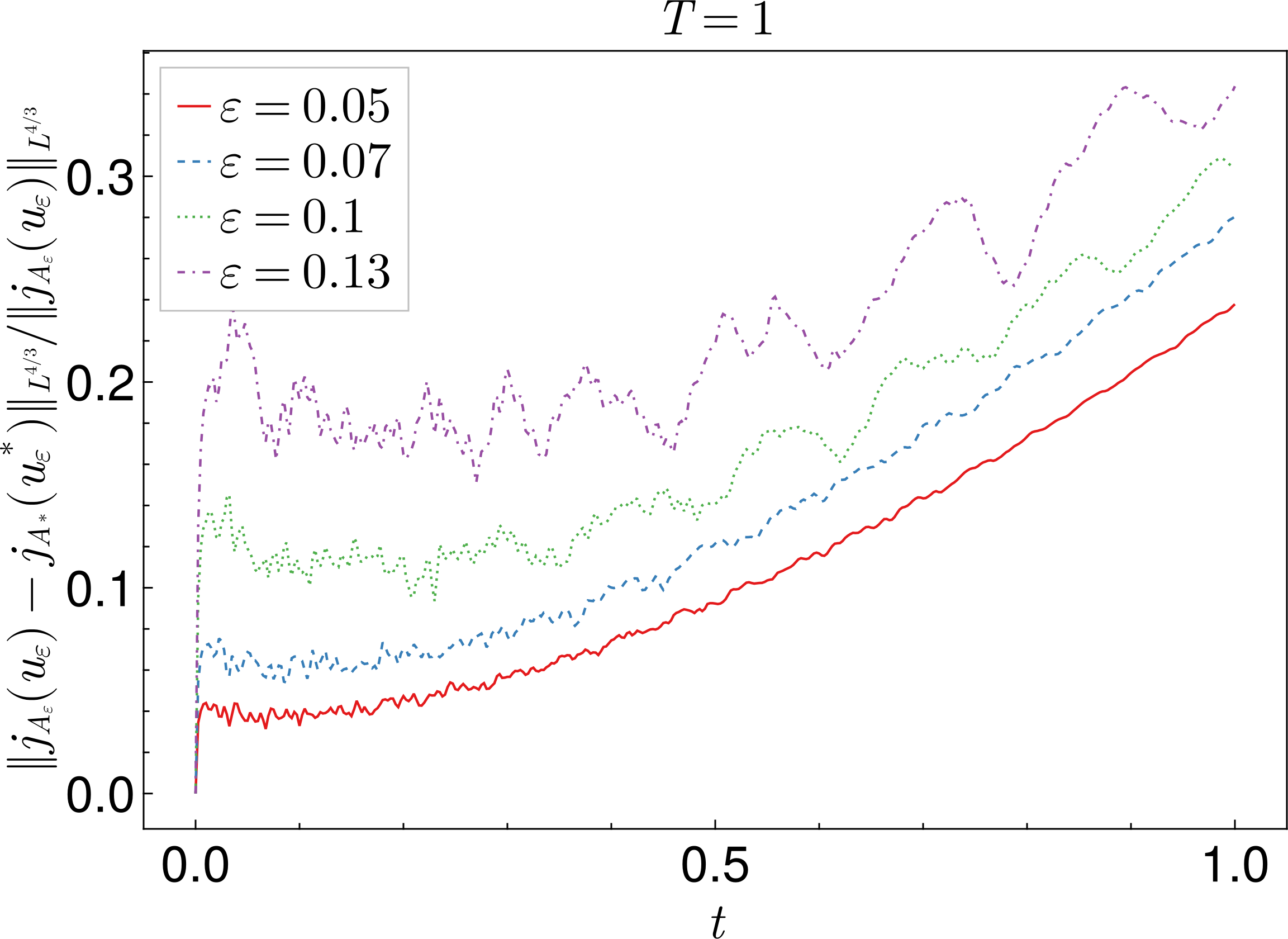}\\
  \includegraphics[width=0.48\linewidth]{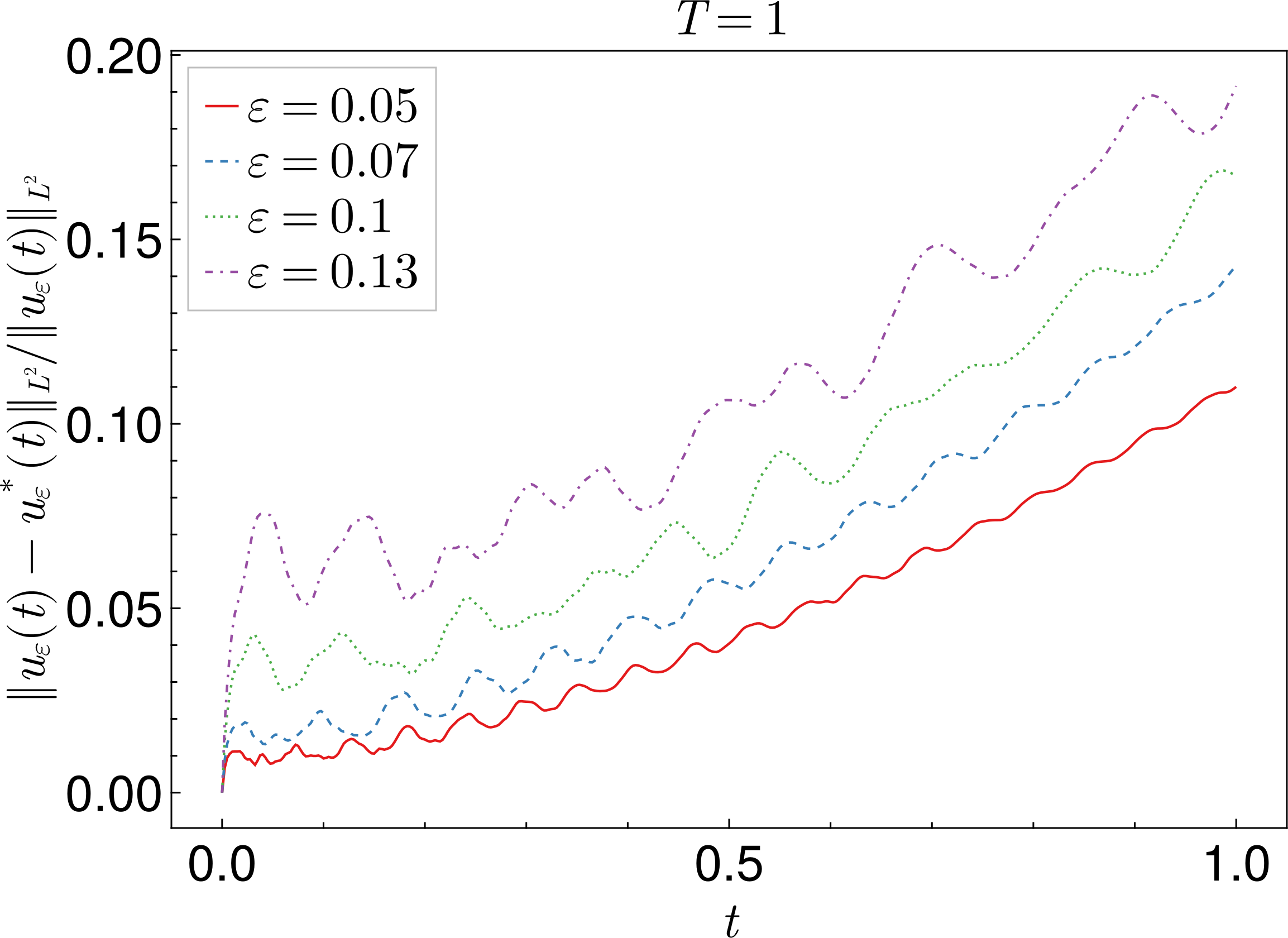}\hfill
  \includegraphics[width=0.48\linewidth]{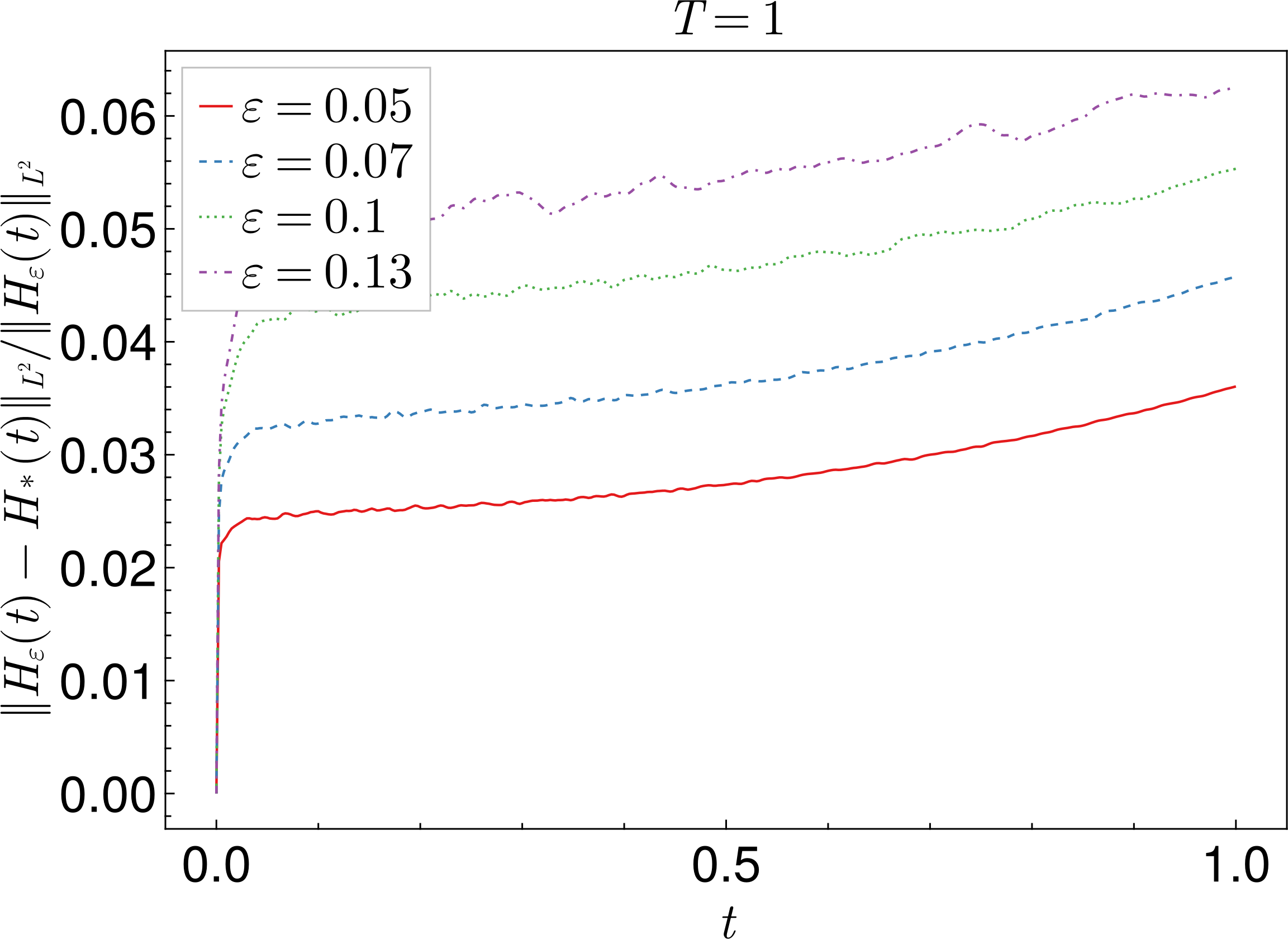}\\
  \caption{Error on different quantities up to time $T$ for case 1. (Top left)
    Error on vortices positions in $\Omega$. (Top right)
    $L^{\frac43}$ norm of the error on the super-currents.
    (Bottom left) $L^2$ norm of the error between $u_\epsilon(t)$
    and the reconstructed wave function $u_\epsilon^*(t)$. (Bottom
    right) $L^2$ norm of the error between $h_\epsilon(t)$ and the
    reconstructed magnetic field $h_*(t)$ .}
  \label{fig:error_eps}
\end{figure}

Finally, one can exploit the numerical results to extrapolate convergence rates with respect to $\epsilon$. In Figure~\ref{fig:linreg_eps}, we perform a linear regression on the (absolute) error on the magnetic field and the super-current. For the magnetic field, the linear regression is performed at four different times while for the super-current, given the noisy results observed in Figure~\ref{fig:error_eps}, the linear regression is performed after averaging the errors over time. For the magnetic field, the regressions clearly show that the error increase in time, but also that the power of $\epsilon$ decreases in time. As for the super current, the regression seems to indicate a slower convergence than for the magnetic field, with a power of epsilon of about $1/2$.

\begin{figure}[h!]
    \centering
    \includegraphics[width=0.48\linewidth]{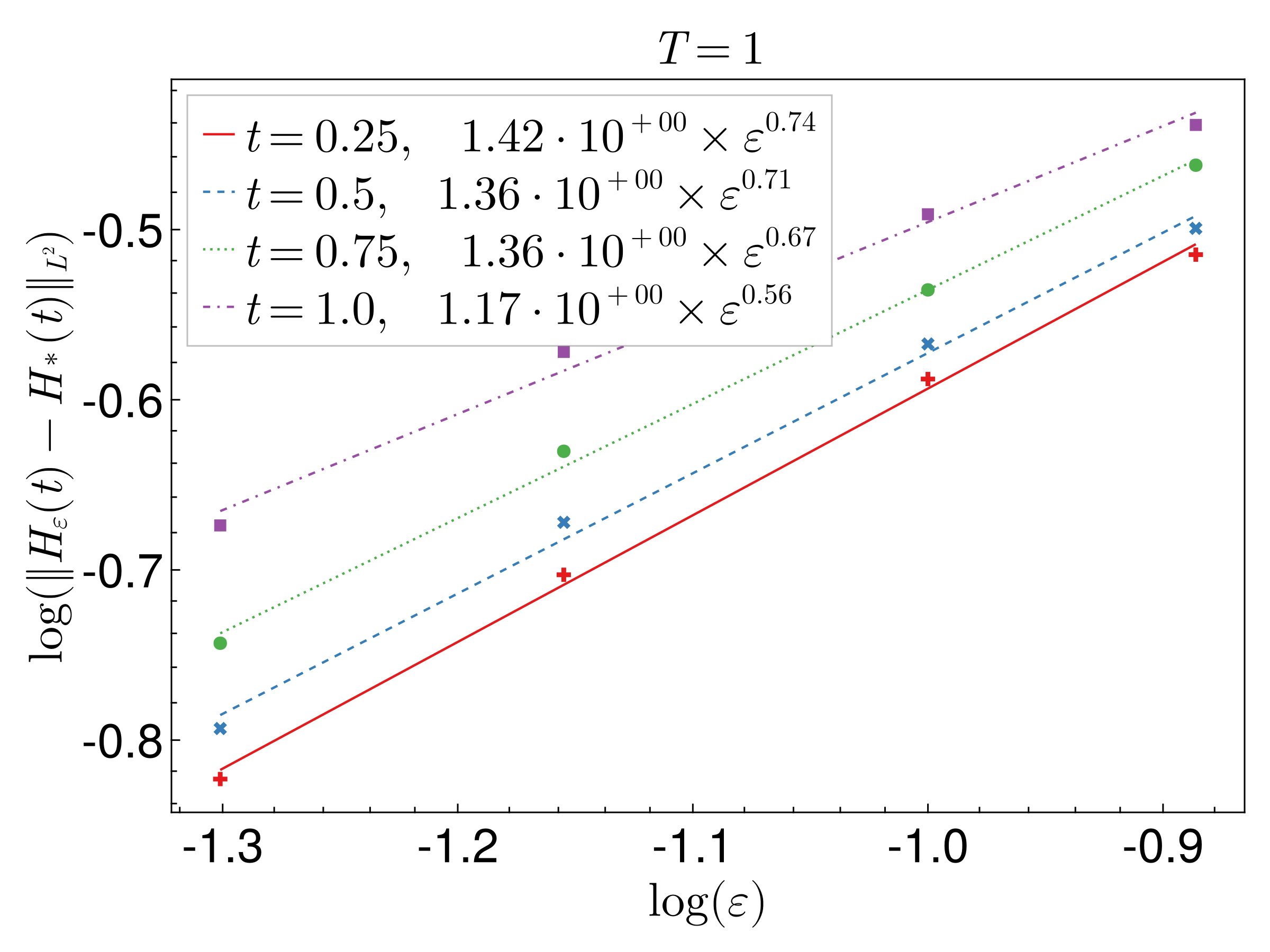}\hfill
    \includegraphics[width=0.48\linewidth]{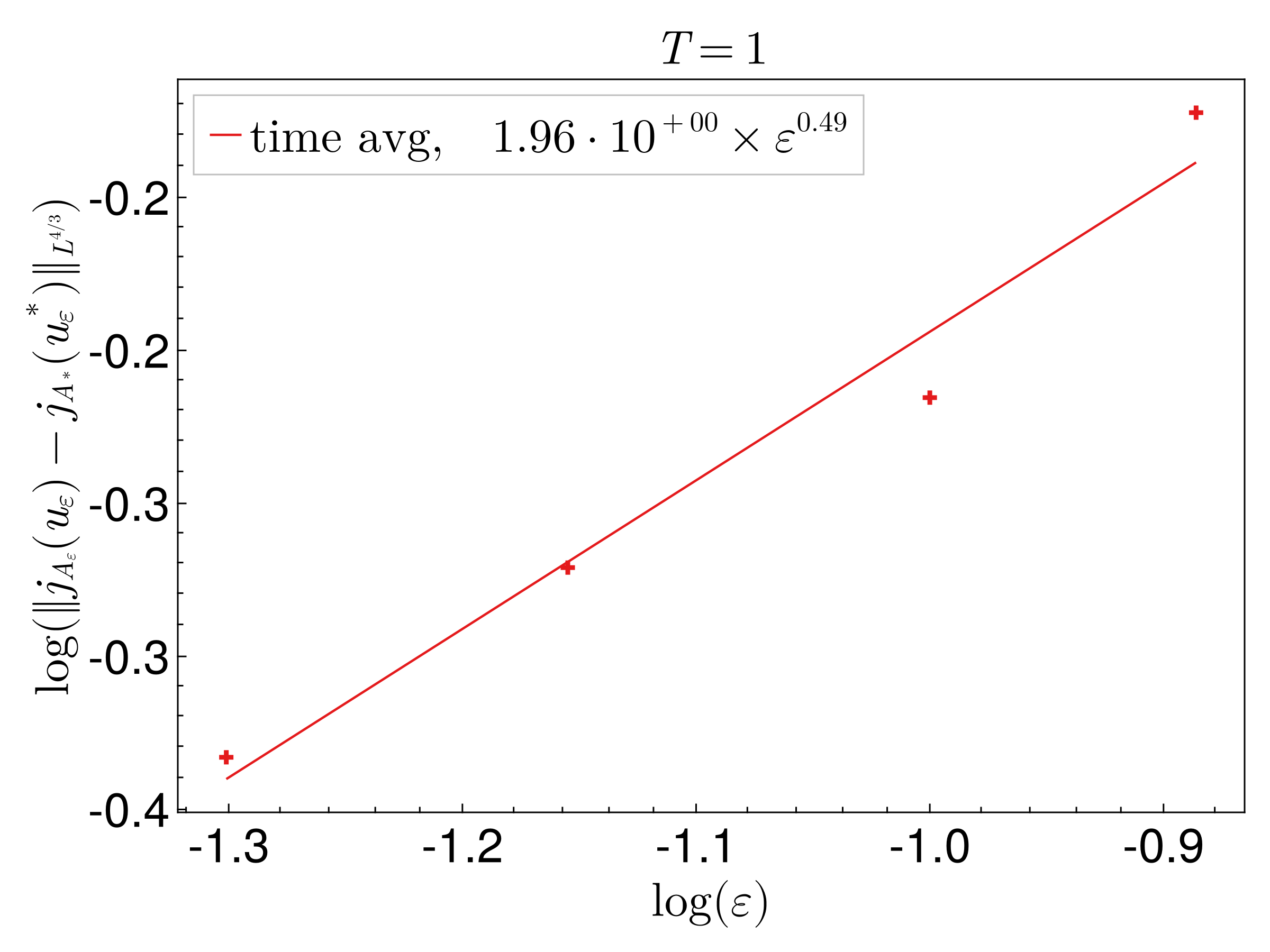}\\
    \caption{Linear regression with respect to $\epsilon$ of the error on the magnetic field and the super-current.}
    \label{fig:linreg_eps}
\end{figure}

\section{The mixed-flow case} \label{sec:mixed-flow}

To conclude, we end this paper by simulating the ODE system
  \eqref{eq:ODE1} in the mixed flow case, that is the dissipative coefficient
  $\alpha_0$ is nonzero. In this case, it is expected that, for a large enough
  external magnetic field $h_{\rm ex}$, the dissipation forces the vortices to
  arrange, for large times, into a lattice pattern \cite{SS07}. In Figure~\ref{fig:dissip_trajectories} we plot, for four different cases, the
  limiting pattern observed at time $T=10$. The discretization parameters of the ODE solver described in Section~\ref{ssec:HD_sim} are
  $\delta t=10^{-3}$ with $m=64$ Fourier modes on the boundary. The initial conditions are given by random positions inside the disc of
  $n\in\{10,30,50,80\}$ vortices with same degree $+1$. As expected, the vortices align according to some pattern for large times. Note also that the
  more numerous the vortices, the larger the magnetic field $h_{\rm ex}$ has to be in order to confine the vortices inside the domain $\Omega$: for instance, with $n=80$ vortices, if $h_{\rm ex} = 100$, we observe numerically that the vortices escape from $\Omega$ in finite time.

  Moreover, we present in Figures~\ref{fig:dissip_approx_case1}--\ref{fig:dissip_approx_case4} the approximate solutions to the full PDE in the mixed-flow case for $\epsilon=0.01$ at different time $t$ obtained with the method introduced in
  Section~\ref{sec:method_1}, showing how the vortices are converging towards the lattice pattern.

  \begin{table}[h!]
    \begin{tabular}{@{}llllll@{}}
      \toprule
      \textbf{Case}   & $n$ & $h_{\rm ex}$ & $\alpha_0$ & $\beta_0$ & vortex degree \\ \midrule
      \textbf{Case 1} & 10  & 50           & 1          & 1         & $+1$ for all \\
      \textbf{Case 2} & 30  & 100          & 1          & 5         & $+1$ for all \\
      \textbf{Case 3} & 50  & 150          & 1          & 10        & $+1$ for all \\
      \textbf{Case 4} & 80  & 250          & 1          & 10        & $+1$ for all \\ \bottomrule
    \end{tabular}
    \caption{Initial configurations for simulation of the ODE system
      \eqref{eq:ODE1} in the mixed-flow case.}\label{tab:dissip_cases}
  \end{table}

  \begin{figure}[h!]
    \hfill{~}
    \includegraphics[width=0.32\linewidth]{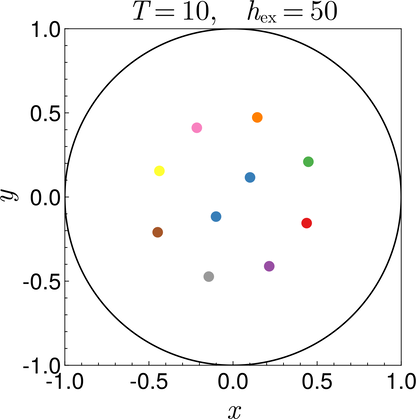}\hfill
    \includegraphics[width=0.32\linewidth]{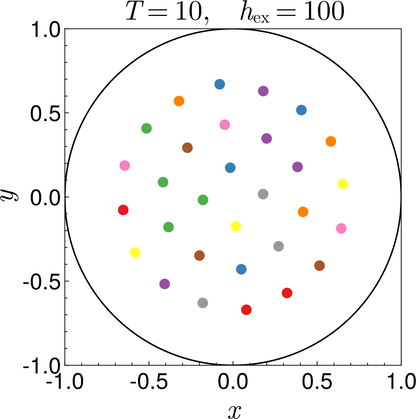}
    \hfill{~}\\ \hfill{~}
    \includegraphics[width=0.32\linewidth]{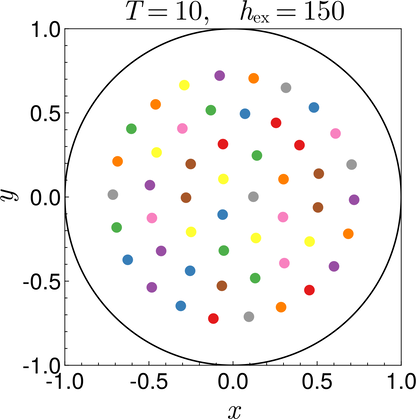}\hfill
    \includegraphics[width=0.32\linewidth]{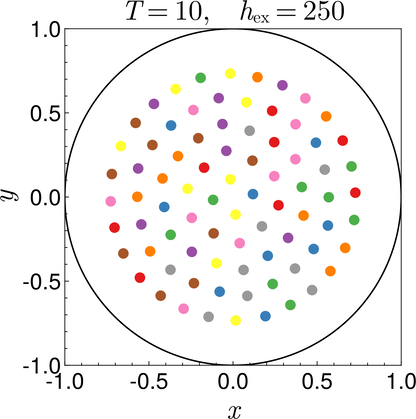}
    \hfill{~}\\
    \caption{Asymptotic pattern of point vortices under the mixed-flow ODE
      \eqref{eq:ODE1}. From left to right, top to bottom: case 1, case 2,
      case 3 and case 4 from Table~\ref{tab:dissip_cases}.}
    \label{fig:dissip_trajectories}
  \end{figure}

  \begin{figure}[p!]
    \includegraphics[width=0.45\linewidth]{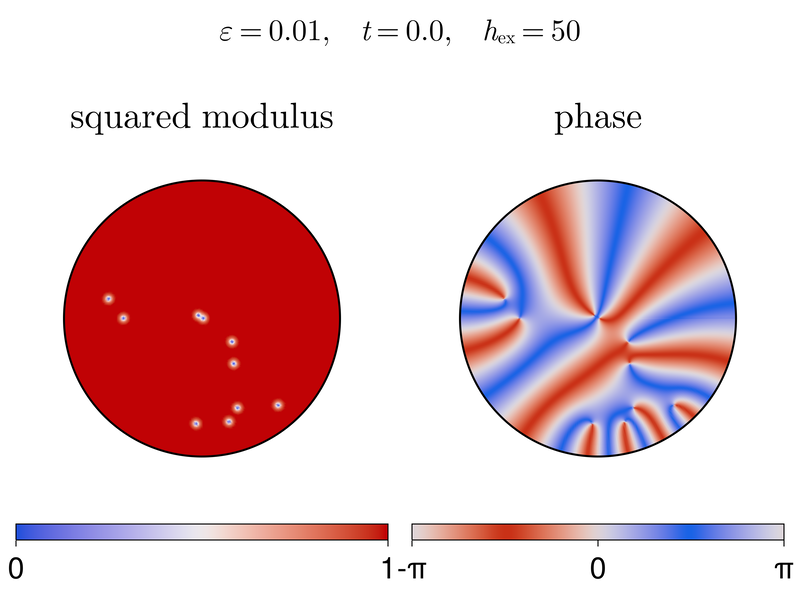}\hfill
    \includegraphics[width=0.45\linewidth]{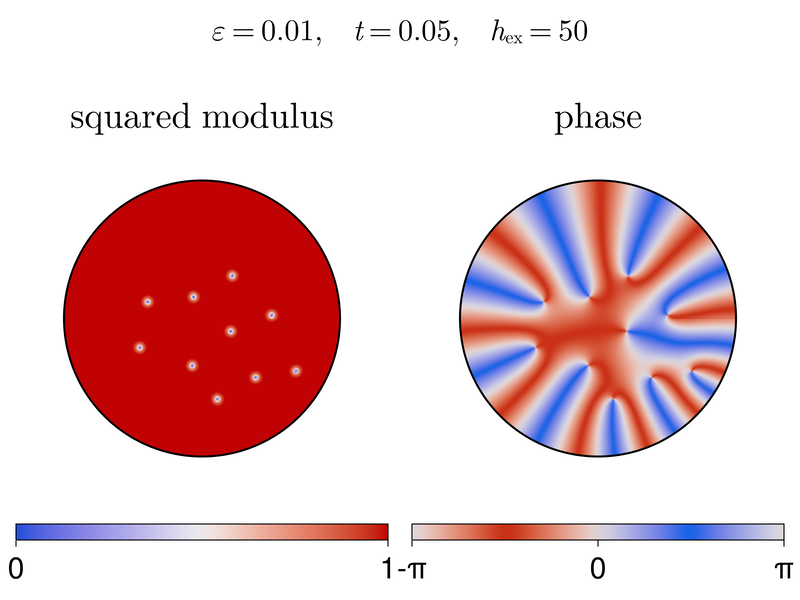}\\
    \includegraphics[width=0.45\linewidth]{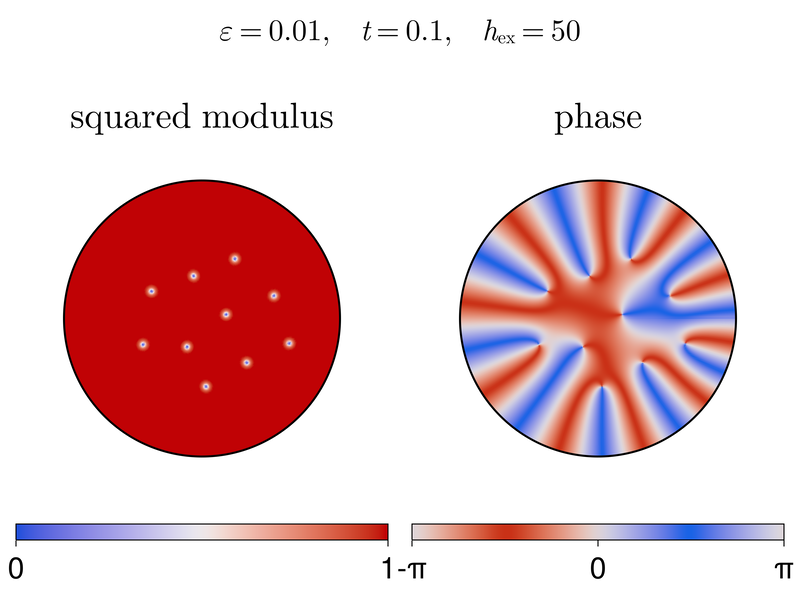}\hfill
    \includegraphics[width=0.45\linewidth]{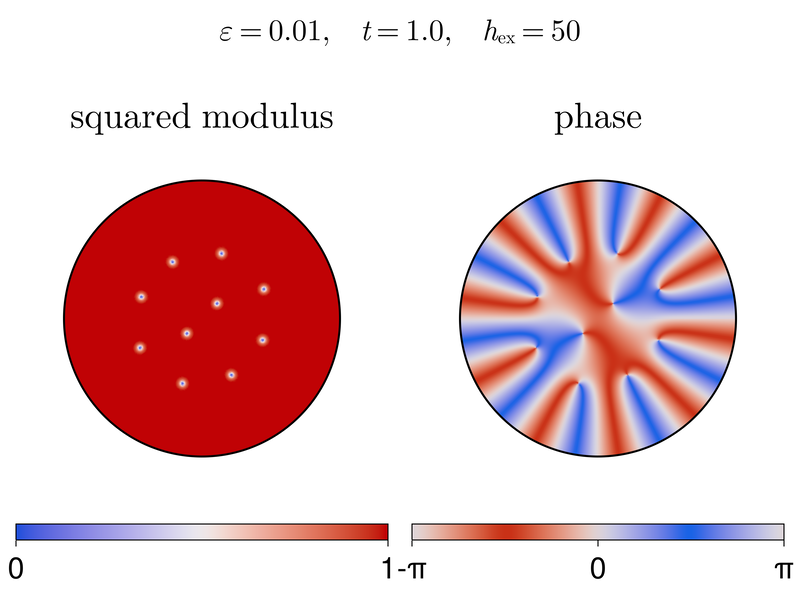}\\
    \caption{Squared modulus and phase of $u_\epsilon^*(t)$ at various times $t$
    under a mixed-flow (Case 1).}
    \label{fig:dissip_approx_case1}
  \end{figure}

  \begin{figure}[p!]
    \includegraphics[width=0.45\linewidth]{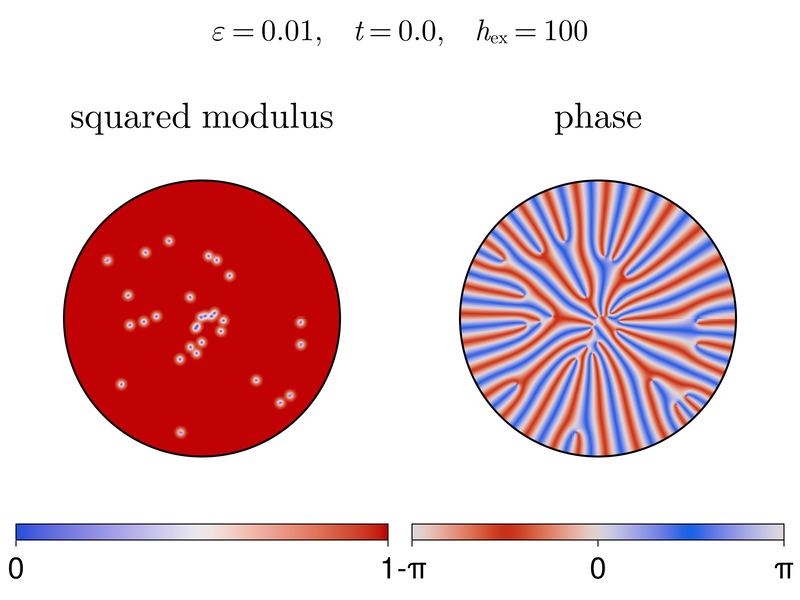}\hfill
    \includegraphics[width=0.45\linewidth]{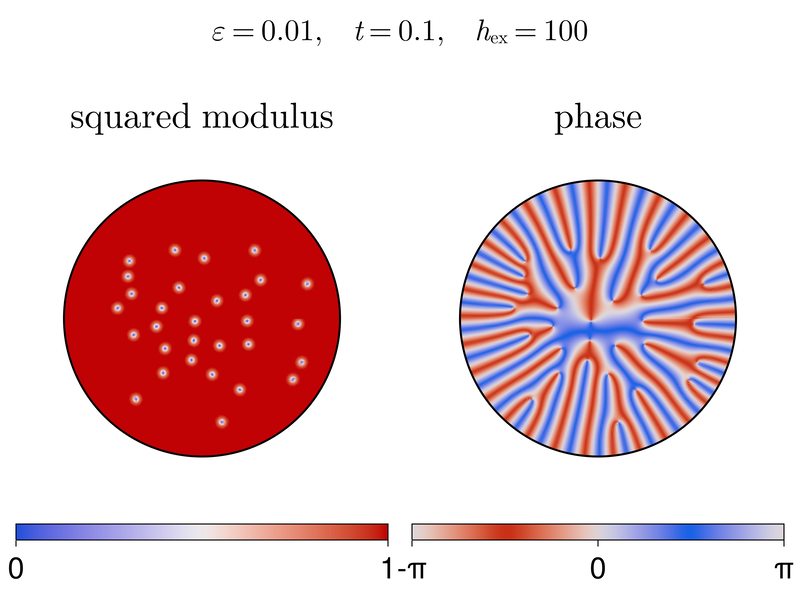}\\
    \includegraphics[width=0.45\linewidth]{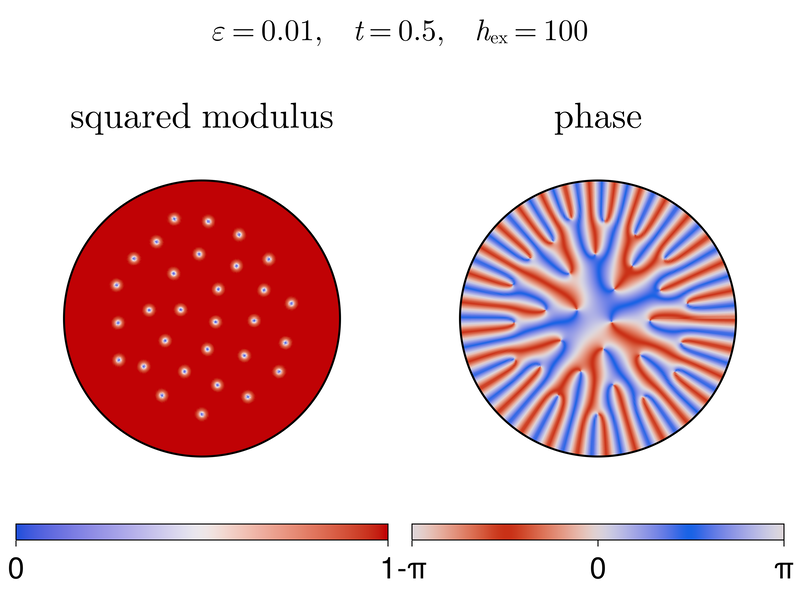}\hfill
    \includegraphics[width=0.45\linewidth]{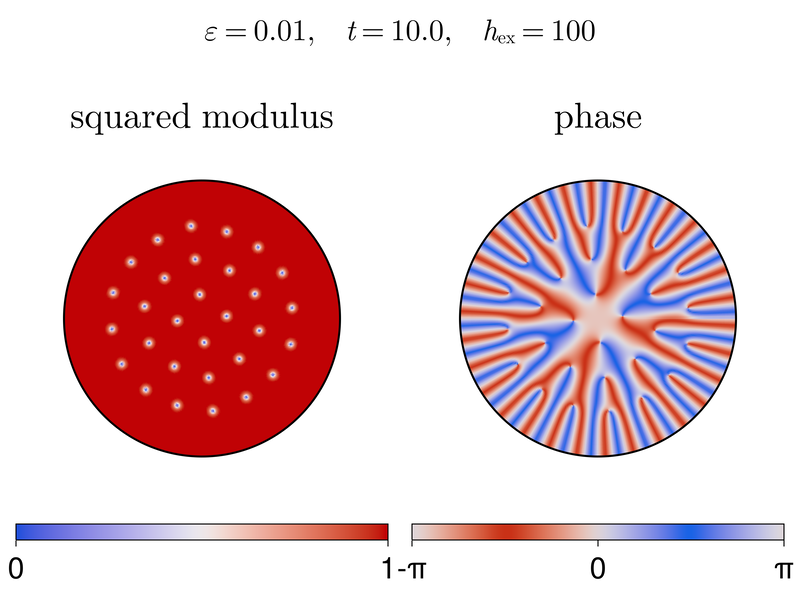}\\
    \caption{Squared modulus and phase of $u_\epsilon^*(t)$ at various times $t$
    under a mixed-flow (Case 2).}
    \label{fig:dissip_approx_case2}
  \end{figure}

  \begin{figure}[h!]
    \includegraphics[width=0.45\linewidth]{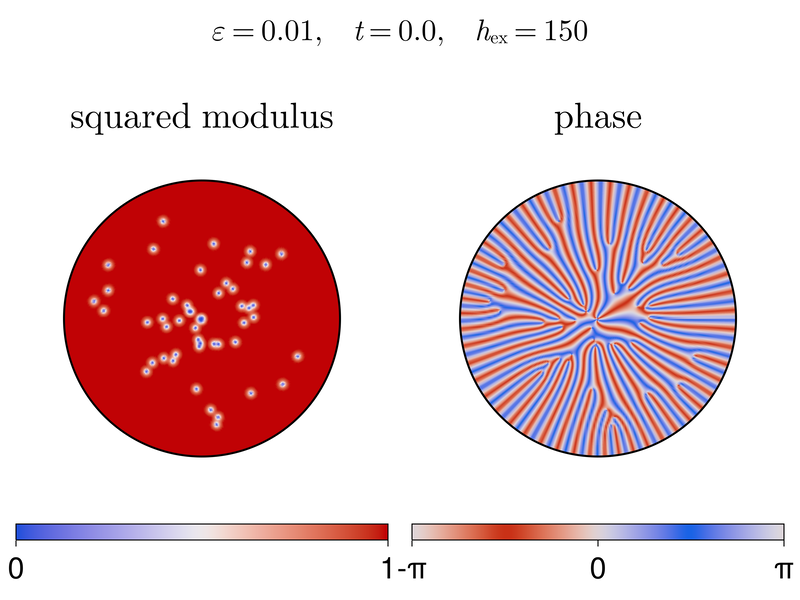}\hfill
    \includegraphics[width=0.45\linewidth]{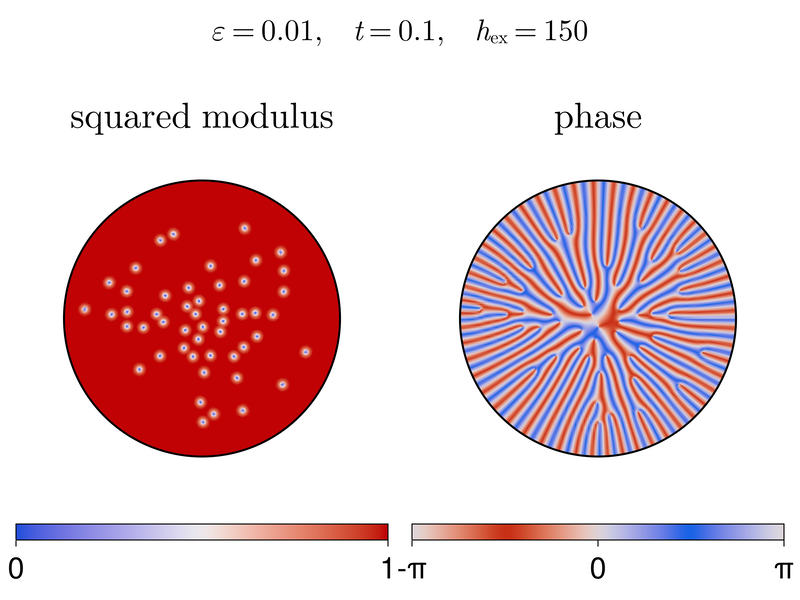}\\
    \includegraphics[width=0.45\linewidth]{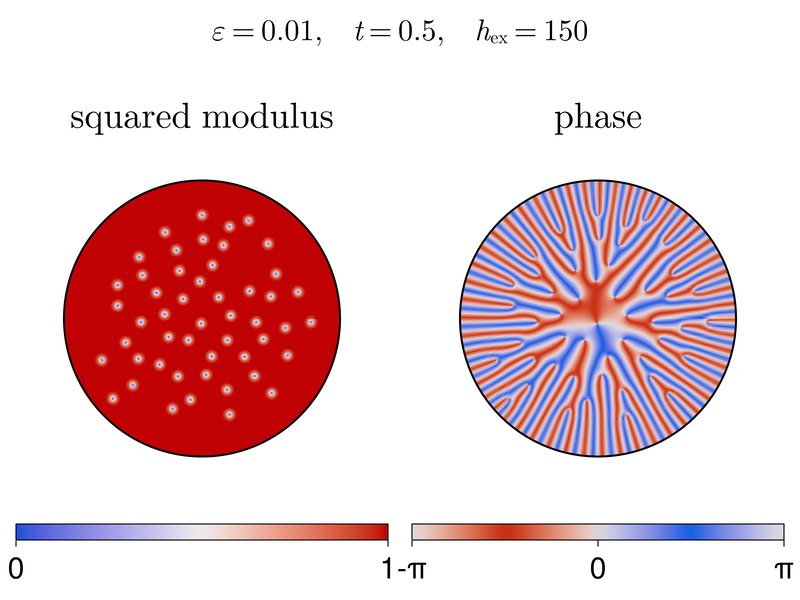}\hfill
    \includegraphics[width=0.45\linewidth]{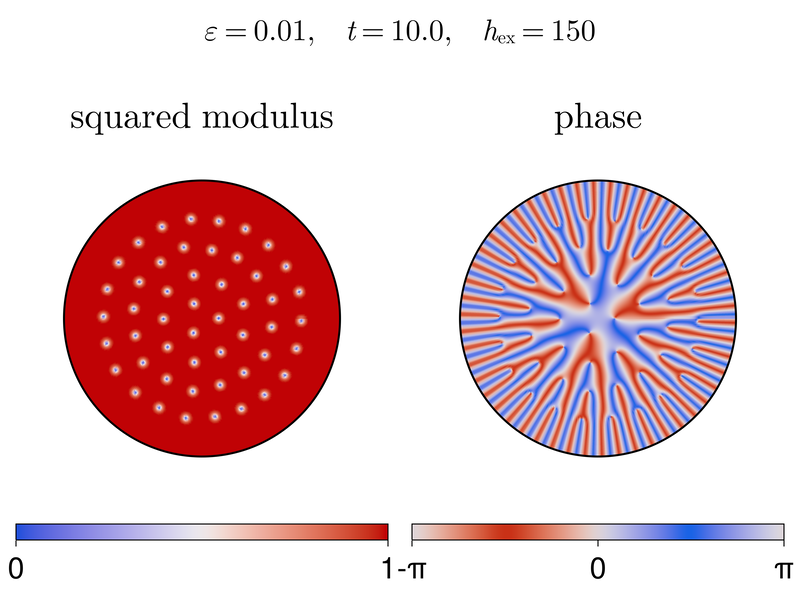}\\
    \caption{Squared modulus and phase of $u_\epsilon^*(t)$ at various times $t$
    under a mixed-flow (Case 3).}
    \label{fig:dissip_approx_case3}
  \end{figure}

  \begin{figure}[h!]
    \includegraphics[width=0.45\linewidth]{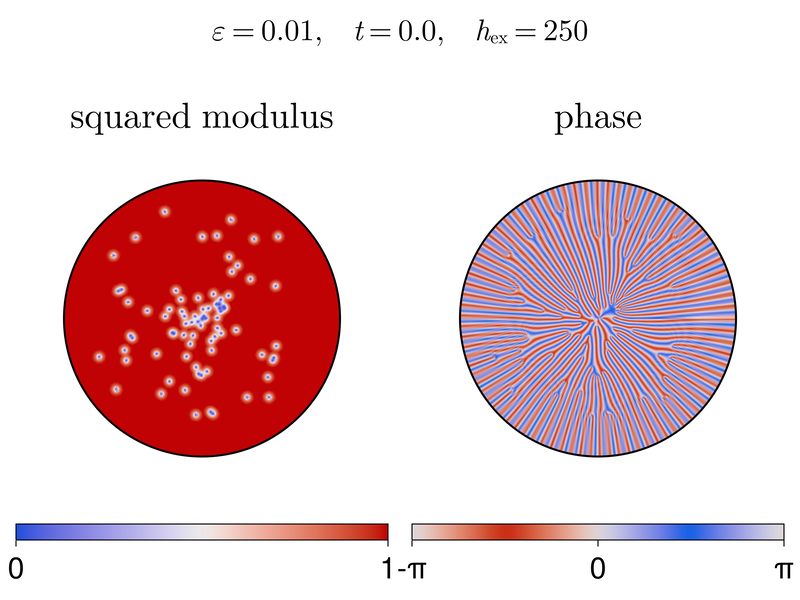}\hfill
    \includegraphics[width=0.45\linewidth]{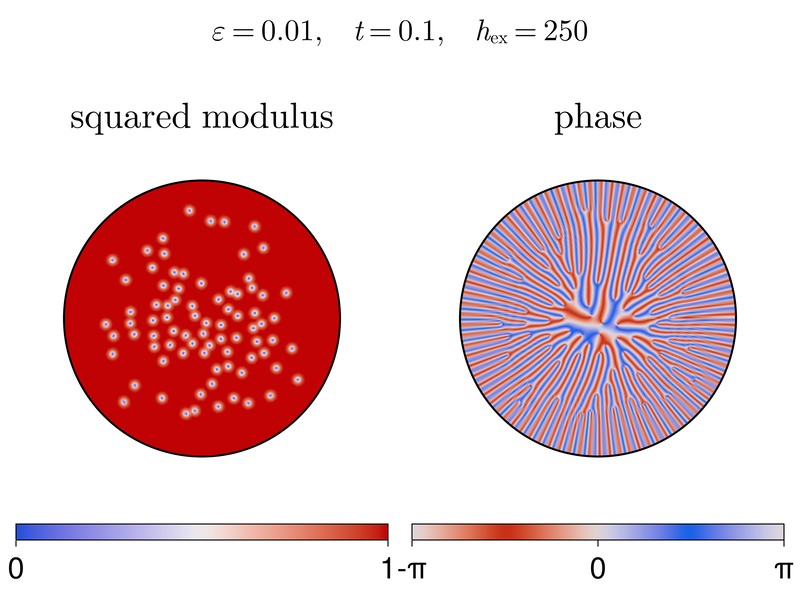}\\
    \includegraphics[width=0.45\linewidth]{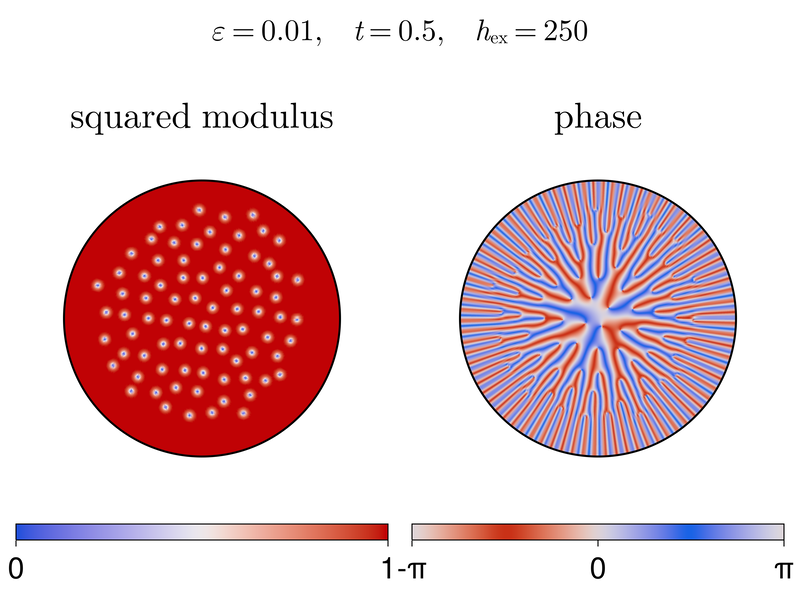}\hfill
    \includegraphics[width=0.45\linewidth]{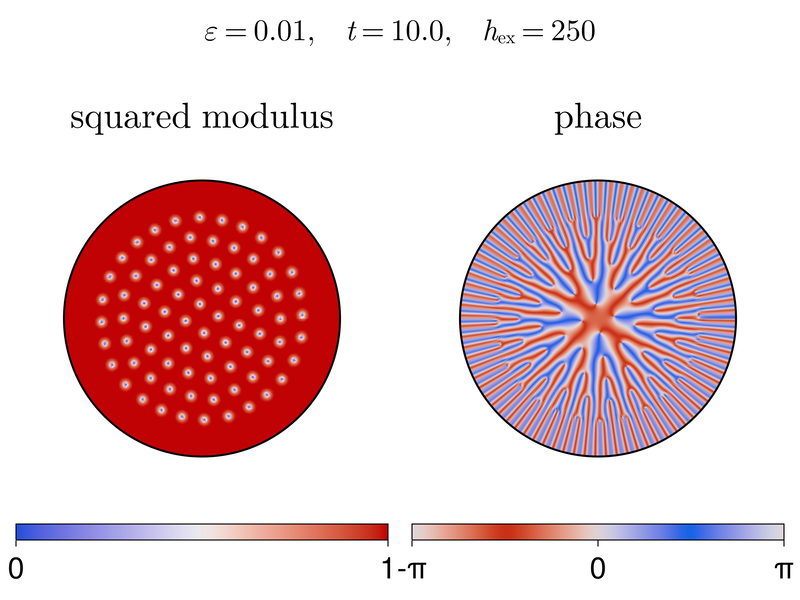}\\
    \caption{Squared modulus and phase of $u_\epsilon^*(t)$ at various times $t$
    under a mixed-flow (Case 4).}
    \label{fig:dissip_approx_case4}
  \end{figure}

\section{Conclusion}

In this work, we first presented a numerical method to simulate the limiting ODE
dynamics of point-like magnetic vortices arising from the TDGL equations with
constant magnetic fields. From these trajectories, we proposed a numerical
method to recover approximate solutions of the TDGL equations in the regime of
small but finite $\epsilon$ under a constant external magnetic
field and applied it for both the Schrödinger-flow and
  mixed-flow cases. This method avoids the use of very fine meshes to properly
capture the vortices core, thereby allowing for significantly faster and
accurate simulations of the physical quantities of interest such as the super
current, the magnetic field, and the density of Cooper pairs. Remarkably, the
method becomes more accurate the smaller $\epsilon$ is. However, the accuracy is
only guaranteed as long as the vortex pairs are sufficiently separated. In
particular, this prevents accurate simulations of some physical phenomena such
as radiation waves emitted by vortex-antivortex annihilation, which is left for
future work.

\section*{Data availability}

All the codes used to generate the plots and run the simulations from this paper
are available at
\begin{center}\url{https://doi.org/10.18419/DARUS-5604}\end{center}

\section*{Acknowledgements}
TC acknowledges funding by the \emph{Deutsche Forschungsgemeinschaft} (DFG,
German Research Foundation) - Project number 442047500 through the Collaborative
Research Center "Sparsity and Singular Structures" (SFB 1481). The authors are
also grateful to Huadong Gao and Xie Wen for providing the Python code of
\cite{GX23} for preliminary tests.
\addtocontents{toc}{\protect\setcounter{tocdepth}{2}}
\clearpage
\appendix

\section{Renormalized energy}\label{app:W}

In this section we present a derivation of~\eqref{eq:renormalized explicit} from~\eqref{eq:renormalized variational}. This result is well-known in the literature, see e.g. \cite{LD97,Spi03}. In these works, however, the precise meaning of Neumann boundary conditions as well as the precise weak formulation of~\eqref{eq:Xieq} is not explicitly stated. Since this weak formulation is important for our approximability result (cf. Corollary~\ref{cor:discrete approximation}) and follows naturally from the derivation of~\eqref{eq:renormalized explicit}, we sketch a self-contained proof of this derivation here.
\begin{proof}[Proof of~\eqref{eq:renormalized explicit}] Since $\Omega$ is simply connected, any $A \in \mH^1_{\rm div}(\Omega;\R^2)$ can be written as $A = -\nabla^\perp \Xi$ for $\Xi \in \mH^2(\Omega;\R) \cap \mH^1_0(\Omega;\R)$. Moreover, by standard elliptic regularity
\begin{align}
 \norm{A}_{\mH^1(\Omega)} \leq \norm{\Xi}_{\mH^2(\Omega)} \lesssim \norm{\curl A}_{\mL^2(\Omega)} = \norm{\Delta \Xi}_{\mL^2(\Omega)} \label{eq:standard elliptic est}
 \end{align}
 Similarly, by the results in \cite[Chapter 2]{BM21}, any $u\in \mathscr{H}^1_\rho(a)$ can be written as
\begin{align*}
	u = \prod_{j=1}^n \underbrace{\left(\frac{x-a_j}{|x-a_j|}\right)^{d_j}}_{=: \phi_j} \exp(\ii \Theta)
\end{align*}
for some $\Theta\in \mH^1(\Omega;\R)$. Next, note that from straightforward calculations
\begin{align}
 \frac{\nabla \phi_j}{\phi_j}(x) = \ii d_j \frac{(x-a_j)^\perp}{|x-a_j|^2} = \ii d_j \nabla^\perp \log(|x-a_j|),
\end{align}
and therefore
\begin{align*}
\int_{\Omega_\rho(a)} |\nabla_A u|^2 \mathrm{d} x &= \int_{\Omega_\rho(a)} |\nabla \Theta + \nabla^\perp \Xi + \sum_{j=1}^n d_j \nabla^\perp \log(|x-a_j|)|^2  \mathrm{d} x  \\
&=  \sum_{j, k =1}^n d_j d_k \underbrace{\int_{\Omega_\rho(a)} \nabla^\perp \log(|x-a_j|) \scpr \nabla^\perp \log(|x-a_k|)}_{=: (I)_{j,k}} \\
&+ \underbrace{\sum_{j=1}^n 2d_j \int_{\Omega_\rho(a)} \nabla^\perp \Xi \scpr \nabla^\perp \log(|x-a_j|)\mathrm{d}x }_{=:(II)}  + \underbrace{\sum_{j=1}^n 2d_j \int_{\Omega_\rho(a)} \nabla \Theta \scpr \nabla^\perp \log(|x-a_j|)\mathrm{d}x }_{=: (III)} \\
&+\underbrace{\int_{\Omega_\rho(a)} 2 \nabla \Theta \scpr \nabla^\perp \Xi \mathrm{d} x}_{=: (IV)} + \int_{\Omega_\rho(a)}  |\nabla \Theta|^2 + |\nabla \Xi|^2 \mathrm{d} x.
\end{align*}
The terms $(I)_{j,k}$ are independent of $h$ and $\Xi$. The diagonal terms $(I)_{j,j}$ are the ones corresponding to the vortex energy contribution and require the renormalization factors $\log(1/\rho)$. This can be seen as follows. By integration by parts (divergence theorem) and the fact that $\Delta \log(|x-a_j|) = 0$ on $\Omega_\rho(a)$, we find
\begin{align}
(I)_{j,j} &= \int_{\Omega_\rho(a)} |\nabla \log(|x-a_j|)|^2 \mathrm{d} x = \int_{\Omega\setminus B_\rho(a_j)} |\nabla \log(|x-a_j|)|^2 \mathrm{d} x + \mathcal{O}(\rho^2) \nonumber \\
&= \int_{\partial \Omega} \log(|x-a_j|) \partial_\nu \log(|x-a_j|) \mathscr{H}^1(\mathrm{d} x) - \int_{\partial B_\rho(a_j)} \frac{\log(|x-a_j|)}{|x-a_j|} \mathscr{H}^1(\mathrm{d} x)  + \mathcal{O}(\rho^2)  \nonumber \\
&= \int_{\partial \Omega} \log(|x-a_j|) \partial_\nu \log(|x-a_j|) \mathscr{H}^1(\mathrm{d} x) - 2\pi \log(\rho) + \mathcal{O}(\rho^2) \label{eq:diagonal I term}
\end{align}
For $(I)_{j,k}$ with $j \neq k$, one can also use integration by parts and the fact that $|x-a_j|^{-1} \in \mL^p_{\rm loc}(\R^2)$ for any $1\leq p <2$ together with H\"older's inequality to obtain
\begin{align}
	(I)_{j,k} = \int_{\partial \Omega} \log(|x-a_j|) \partial_\nu \log(|x-a_k|) \mathscr{H}^1(\mathrm{d} x) - 2\pi \log(|a_k-a_j|)  +  \mathcal{O}(\rho^{2\frac{p-1}{p}}) \quad \mbox{(for $1\leq p <2$).} \label{eq:off-diagonal I term}
\end{align}
For the term $(II)$, we note that, since $\Xi = 0$ on $\partial \Omega$, $\Xi$ belongs to the H\"older continuous class $C^{0,\alpha}(\Omega)$ for any $\alpha<1$ (by the Sobolev embedding), and $\Delta \log(|x-a_j|) = 0$ on $\Omega_\rho(a)$, we have
\begin{align}
 (II) = -\sum_{j=1} 2d_j \sum_{k=1}^n \int_{B_\rho(a_k)} \Xi(x) \frac{(x-a_j)}{|x-a_j|^2} \scpr \frac{(x-a_k)}{|x-a_k|} \mathscr{H}^1(\mathrm{d} x) = -\sum_{j=1} 4 \pi d_j \Xi(a_j) + \mathcal{O}_\alpha(\rho^\alpha \norm{\Xi}_{\mH^2}). \label{eq:II term}
\end{align}
For $(III)$, we note that
\begin{align*}
	(III) &= \sum_{j=1}^n 2d_j \int_{\Omega_\rho(a)} \mathrm{div}\left(\Theta \nabla^\perp \log(|x-a_j|)\right) \mathrm{d} x  \\
	&= \sum_{j=1}^n 2d_j \int_{\partial \Omega} \Theta(x) \partial_\tau \log(|x-a_j|) \mathscr{H}^1(\mathrm{d} x)
	- \sum_{k=1}^n 2d_j \int_{\partial B_\rho(a_k)} \Theta(x) \partial_\tau \log(|x-a_j|) \mathscr{H}^1(\mathrm{d} x).
\end{align*}
As $\partial_\tau \log(|x-a_j|) = 0$ along $\partial B_\rho(a_j)$ and $\nabla \log(|x-a_j|)$ is bounded along $\partial B_\rho(a_k)$ for $k\neq j$, we can use the trace estimate $\norm{\Theta}_{\mL^2(\partial B_\rho(a_j))} \lesssim \norm{\Theta}_{\mH^1(\Omega_\rho(a))}$ (with a constant independent\footnote{Such an estimate can be proved, e.g., via polar coordinates and using a Fourier series expansion along circles centred at $a_j$.} of $\rho$) and Cauchy-Schwarz to conclude that
\begin{align}
	(III) = 2\sum_{j=1}^n d_j \int_{\partial \Omega} \Theta(x) \partial_\tau \log(|x-a_j|) \mathscr{H}^1(\mathrm{d} x) + \mathcal{O}(\rho^{\frac12} \norm{\Theta}_{\mH^1(\Omega_\rho(a))}).  \label{eq:III term}
\end{align}
Similarly, for $(IV)$ one can use the previous inequality together with Sobolev's inequality to obtain
\begin{align}
	(IV) &=  2\int_{\Omega_\rho(a)} \mathrm{div} (\Theta \nabla^\perp \Xi) \mathrm{d} x  = \int_{\partial \Omega} \Theta \partial_\tau \Xi \mathscr{H}^1(\mathrm{d}x) - \sum_{j=1}^n \int_{\partial B_\rho(a_j)} \Theta \partial_\tau \Xi \mathscr{H}^1(\mathrm{d}x)\nonumber \\
	&=  \mathcal{O}(\rho^{\frac12} \norm{\Theta}_{\mH^1(\Omega_\rho(a))} \norm{\Xi}_{\mH^2})   \label{eq:IV term}
\end{align}
where the last equality holds because $\nu \scpr \nabla^\perp \Xi = \partial_\tau \Xi = 0$ on $\partial \Omega$ (since $\Xi \in \mH^1_0(\Omega)$).

Now note that, it suffices to consider trial states with uniform (in $\rho$)
bound on $\norm{\Xi}_{\mH^2}$ and $\norm{\Theta}_{\mH^1(\Omega_\rho(a))}$ as
otherwise either the term $\int |\Delta \Xi - h_{\rm ex}|^2$ or the term
$\int|\nabla \Theta|^2$ would blow-up. Thus, we can combine
estimates~\eqref{eq:diagonal I term}--~\eqref{eq:IV term} and pass to the limit
$\rho \searrow 0$ to conclude that
\begin{align*}
	&W_\Omega(a,d;h_{\rm ex}) = \! \inf_{\substack{\Theta \in \mH^1(\Omega) \\ \Xi \in \mH^2(\Omega) \cap \mH^1_0(\Omega)}}\! \biggr\{ -\pi \sum_{j\neq k} d_j d_k \log(|a_j - a_k|) \\
	&+ \underbrace{\frac12 \int_\Omega |\nabla \Theta|^2 \mathrm{d} x + \sum_{j=1}^n d_j \int_{\partial \Omega} \Theta(x) \partial_\tau \log(|x-a_j|) + \frac12 \sum_{k=1}^n d_k \log(|x-a_j|) \partial_\nu \log(|x-a_k|) \mathrm{d} x}_{=: W_1(\Theta)} \\
	&+ \underbrace{\frac12 \int_\Omega |\nabla \Xi|^2 + |\Delta \Xi - h_{\rm ex}|^2 \mathrm{d} x - \sum_{j=1}^n 2\pi d_j \Xi(a_j)}_{=: W_2(\Xi)} \biggr\}
\end{align*}
In particular, the above minimization is a separate minimization in $\Theta$ and $\Xi$.

For the minimization in $\Theta$, note that $W_1(\Theta)$ is strictly convex, and therefore $W_1 = \inf_\Theta W_1(\Theta) = - \frac12 \int |\nabla \Theta|^2$ where $\Theta$ is the unique weak solution of
\begin{align*}
	\begin{dcases} \Delta\Theta = 0 \quad \mbox{in $\Omega$,} \\
	\partial_\nu \Theta = - \sum_{j=1}^n d_j \partial_\tau \log(|x-a_j|), \quad &\mbox{on $\partial \Omega$.} \end{dcases}
\end{align*}
In particular, if we define the conjugate harmonic function of $\Theta$ as the unique up to a constant function $R$ such that $\nabla^\perp R = \nabla \Theta$, then the constant can be chosen such that $R$ solves~\eqref{eq:Req}. Hence
\begin{align*}
	W_1 &= \int_{\Omega} |\nabla \Theta|^2 \mathrm{d}x - \frac12 \sum_{j} d_j \int_{\partial \Omega} \log(|x-a_j|) \partial_\nu \Theta(x)  + R(x) \partial_\nu \log(|x-a_j|) \mathscr{H}^1(\mathrm{d}x) \\
	&= - \frac12 \int_{\Omega} |\nabla \Theta|^2 - \frac12 \sum_{j=1}^n \int_{\partial \Omega} R(x) \partial_\nu \log(|x-a_k|) \mathscr{H}^1(\mathrm{d}x) \\
	&=  - \frac12 \int_{\Omega} |\nabla R|^2 - \frac12 \sum_{j=1}^n \int_{\partial \Omega} R(x) \partial_\nu \log(|x-a_k|) \mathscr{H}^1(\mathrm{d}x) \\
	&= \frac12 \sum_{j=1}^n d_j \int_{\partial \Omega} \log(|x-a_j|) \partial_\nu R(x) - R(x) \log(|x-a_j|) \mathrm{d} x = - \pi \sum_{j=1}^n d_j R(a_j),
\end{align*}
where we used that $\Delta \log = 2\pi \delta_0$ in the last step.

Similarly, the minimizer of $W_2$ is the unique weak solution $\Xi \in \mH^2(\Omega) \cap \mH^1_0(\Omega)$ of
\begin{align*}
	\int_\Omega \nabla \Xi \scpr \nabla \psi + \Delta \Xi \Delta \psi\mathrm{d} x = -h_{\rm ex} \int_\Omega \Delta \psi + 2 \pi \sum_{j=1}^n d_j \psi(a_j) = - h_{\rm ex} \int_{\partial \Omega} \partial_\nu \psi \mathscr{H}^1(\mathrm{d}x) + 2\pi \sum_{j=1}^n d_j \psi(a_j),
\end{align*}
for $\psi \in \mH^2(\Omega) \cap \mH^1_0(\Omega)$.

\end{proof}

\section{Well-posedness of the modified biharmonic equation}\label{app:biharm}

In this section we recall the following elementary result on the existence and uniqueness of solutions of the modified biharmonic equation~\eqref{eq:Xieq}.

\begin{lemma}[Existence and uniqueness] \label{lem:well-posed biharmonic} Let $F: \mH^2(\Omega) \cap \mH^1_0(\Omega) \rightarrow \R$ be a continuous functional, then there exists a unique $\Xi \in \mH^2(\Omega) \cap \mH^1_0(\Omega)$ such that
\begin{align}
	\int_\Omega \nabla \Xi(x) \scpr \nabla \psi(x) + \Delta \Xi(x) \Delta \psi(x) \mathrm{d} x = F(\Psi), \quad \mbox{for any $\psi \in \mH^2(\Omega) \cap \mH^1_0(\Omega)$.} \label{eq:weak formulation}
\end{align}
\end{lemma}

\begin{proof} Uniqueness follow since any such weak solution is a critical point of the strictly convex function $G(\Xi) = \frac12 \int |\nabla \Xi|^2 + |\Delta \Xi|^2 - F(\Xi)$. Existence follows from the standard Lax-Milgram argument, or the direct method.
\end{proof}

An immediate corollary of this result is the existence and uniqueness of weak solutions of~\eqref{eq:weak Weq}, as claimed in Remark~\ref{rem:weak solution}.
\begin{corollary}[Existence and uniqueness] \label{cor:wel-posedness} Let $b:\partial \Omega \rightarrow \R$ be the boundary trace of a function in $\mH^2(\Omega)$ and $F: \mH^2(\Omega) \cap \mH^1_0(\Omega) \rightarrow \R$ be a continuous linear functional, then there exists a unique weak solution $\Xi \in \mH^2(\Omega)$ of the problem
\begin{align}
    \int_\Omega \Delta \Xi(x) (\Delta - 1) \psi(x) \mathrm{d}x = F(\psi) \quad \mbox{for any $\psi \in \mH^2(\Omega) \cap \mH^1_0(\Omega)$.} \label{eq:weak non-homogeneous}
\end{align}
with boundary conditions $\Xi \rvert_{\partial \Omega} = b$.
\end{corollary}
\begin{proof}
    Let $B\in \mH^2(\Omega)$ be a function satisfying $B\rvert_{\partial \Omega} = b$ (which exists by assumption), then let $\Xi_h \in \mH^2(\Omega) \cap \mH^1_0(\Omega)$ be the unique solution of
    \begin{align*}
        \int_\Omega \nabla \Xi_h(x) \scpr \nabla \psi(x) + \Delta \Xi_h(x) \Delta \psi(x) \mathrm{d} x = F(\psi) - \int_{\Omega} \Delta B (\Delta - 1) \psi, \quad \mbox{for any $\psi \in \mH^2(\Omega) \cap \mH^1_0(\Omega)$,}
    \end{align*}
    which exists by Lemma~\ref{lem:well-posed biharmonic}. Then, by integration by parts, the function $\Xi = \Xi_h + B$ satisfies~\eqref{eq:weak non-homogeneous} and has the correct boundary conditions. Uniqueness follows from the uniqueness statement in Lemma~\ref{lem:well-posed biharmonic}.
\end{proof}


\end{document}